\documentclass[a4paper,11pt]{article}

\usepackage{amssymb, amstext, amsmath, amsthm, latexsym, textcomp}

\usepackage[utf8]{inputenc} 
\usepackage[english]{babel}

\usepackage{enumitem}

\newtheorem{defi}{Definition}[section]
\newtheorem{teor}[defi]{Theorem}
\newtheorem{lemma}[defi]{Lemma}
\newtheorem{prop}[defi]{Proposition}
\newtheorem{rem}[defi]{Remark}

\newtheorem{co}[defi]{Corollary}

\DeclareMathOperator{\ad}{ad}
\DeclareMathOperator{\st}{st}
\DeclareMathOperator{\Ad}{Ad}

\DeclareMathOperator{\prr}{Rad}

\DeclareMathOperator{\Ch}{char}
\DeclareMathOperator{\diag}{diag}

\DeclareMathOperator{\Spec}{Spec}
\DeclareMathOperator{\pideg}{p.i.deg}
\DeclareMathOperator{\Id}{Id}
\DeclareMathOperator{\Inder}{Inder}
\DeclareMathOperator{\Ker}{Ker}
\DeclareMathOperator{\Hom}{Hom}
\DeclareMathOperator{\End}{End}

\DeclareMathOperator{\Der}{Der}
\DeclareMathOperator{\Lie}{Lie}
\DeclareMathOperator{\Aut}{Aut}
\DeclareMathOperator{\Soc}{Soc}

\DeclareMathOperator{\CM}{CM}
\DeclareMathOperator{\Ann}{Ann}
\DeclareMathOperator{\Tor}{Tor}
\DeclareMathOperator{\Inn}{Inn}

\DeclareMathOperator{\Mder}{Mder}

\begin{document}

\date{}

\title{Algebraic derivations and radicals of 
automorphism groups of algebras} 

\author{A.~Yu.~Golubkov}

\maketitle

\begin{abstract} 
The paper presents conditions under which algebraic derivations of a semiprime 
associative ring are induced by inner derivations of its symmetric Martindale 
ring of quotients, and extends a theorem on the coincidence of the prime 
and weakly solvable radicals of subgroups of multiplicative groups of 
associative $PI$--algebras to automorphism groups of algebras associated with 
algebraic derivations.
\end{abstract}

\section{Introduction}

The writing of the paper is motivated by \cite{ADN} and the author's desire 
to supplement the line of results of \cite{Gol7}. In Sec. 2 the 
description from \cite{Har2} of derivations of prime associative\linebreak rings with 
generalized identities that are linear over extended centroids is extended 
to algeb\-raic derivations of semiprime associative algebras with strict 
generalized identities and multiplicatively semiprime algebras, in 
Sec. 3 the series from \cite{Gol7} of group versions of the Amitsur --- Levitzki 
theorem on the equality of lower and upper nil-radicals of associative 
$PI$--algebras by subgroups of a number of automorphism groups of algebras 
generated by the exponents of algebraic derivations is expanded. This theorem 
has analogues for alternative and over rings with $1/2$ right-alternative and 
linear Jordan $PI$--algebras with the lower radical replaced by the 
Slin'ko --- McCrimmon and McCrimmon radicals (Addition 1, \cite[Theorem~4]{Zel1}). Similar conclusions include the coincidence of locally and weakly 
solvable radicals $LS$ and $T$ on the class of algebras $\mathfrak{M}$, 
closed under taking ideals and homomorphic images, so that $LS$ is defined and 
special on $\mathfrak{M}$ (locally solvable algebras from $\mathfrak{M}$ form 
a radical subclass of $\mathfrak{M}$, $LS$--semisimple algebras from 
$\mathfrak{M}$ are subdirect products of their $LS$--semisimple prime 
quotient algebras) and the central closures of the prime $LS$--semisimple 
algebras from $\mathfrak{M}$ are locally finite-dimensional over their 
Martindale centroids. For ${R\in \mathfrak{M}}$ with finite-dimensional (of 
dimension at most ${n\geq 1}$) central closures of prime quotient algebras over 
Martindale centroids, ${LS(R)=T(R)}$ is the intersection of prime ideals of 
$R$, in the quotient algebras by which there are no non-zero solvable ideals 
(and the smallest of the ideals of $R$, the quotient algebras do not contain 
non-zero solvable ideals (of degree at most $n$)). An example is the 
equality of $LS$, $T$ and the prime radical $\prr$ of generalized special Lie 
algebras \cite{BP}, \cite[Sec.~3]{Gol6}. In \cite{Gol7}, analogues of 
the Amitsur --- Levitzki theorem are presented for subgroups of multiplicative 
groups of associative $PI$--algebras, which include the results of 
B.~I.~Plotkin and S.~A.~Pikhtil'kov. 

Everywhere below (in the absence of additional conditions) $F$ is any 
associative commutative ring with 1, all $F$--modules are unitary, 
simultaneously left and right with identical action of $F$ and the 
action of their endomorphisms will be written on the right, all 
$F$--algebras are\linebreak linear (quadratic Jordan algebras are mentioned in 
Addition 1), classes of $F$--algebras (groups) contain a zero algebra 
(unit group) and are closed under taking isomorphic copies. 

For any $F$--algebra $R$ and set ${B\subseteq R}$, $F B$, $\langle B\rangle$ 
and $(B)_R$ are the $F$--submodule, subalgebra and ideal of $R$ generated by 
$B$, $\End_F(R)$ is the endomorphism $F$--algebra of the $F$--module $R$, 
$\Aut_F(R)$ is the automorphism group of the $F$--algebra $R$, 
${M^R(B)=\langle t_x\mid t=l, r,\ x\in B\rangle}$ and, in particular, 
${M(R)=M^R(R)}$ (${M(R)'=F \Id_R+M(R)}$) is the \emph{algebra of multiplications} 
(\emph{with unity}) of $R$, ${\mathcal{E}(R)=\{I\lhd R\mid I\cap J\ne \{0\}\ \forall 
\{0\}\ne J\lhd R\}}$ is the set of essential ideals of $R$, $N(R)$, $K(R)$ and 
$Z(R)$ are the associative center, commutative center and center of $R$,  
\begin{gather*}
N(R)\ =\ \{x\in R\mid (x, R, R)=(R, x, R)=(R, R, x)=\{0\}\}\,,
\\
K(R)\ =\ \{x\in R\mid [x, R]=\{0\}\}\,, 
Z(R)\ =\ N(R)\cap K(R)\ =\ \{x\in R\mid l_x=r_x\in Z(M(R))\}\,,
\end{gather*}
${\Ann_t B=\{x\in R\mid (B)t_x=\{0\}\}, t=l, r}$, where 
${l_x: y\longmapsto x y}$ and ${r_x: y\longmapsto y x, x, y\in R}$, are the 
operators of left and right nultiplication by $x$, $\Id_M$ is the identical 
isomorphism of the module $M$, ${(x, y, z)=(x y) z-x (y z)}$ and 
${[x, y]=x y-y x}$ are the associator and ring commutator of ${x, y, z\in R}$ 
(the group commutator is ${[x, y]=x y x^{-1} y^{-1}}$). We denote the algebra 
obtained from $R$ by replacing the multiplication operation with ${[\ ,\ ]}$ 
by $R^{(-)}$, the $m$-th commutant and the $n$-th power 
($n$-th member of the lower central series) of the algebra (group) $R$ by 
$R^{(m)}$ and $R^n$, ${m\geq 0}$, ${n\geq 1}$, where, let us remind, to the 
associative $R$ there corresponds the Lie algebra $R^{(-)}$. If $R$ is a Lie algebra, then 
${\ad_x=r_x=-l_x}$, ${x\in R}$, ${\Ad(R)=M(R)}$, ${\Ad(R)'=M(R)'}$ (the choice 
of ${\ad=r}$ is related to the right action of $\End_F(R)$ on $R$ in the paper).

An algebra (a group) $R$ is \emph{prime} if ${I J\ne \{0\}}$ for all 
${\{0\}\ne I, J\lhd R}$ (${[I, J]\ne \{e\}}$ for all 
${\{e\}\ne I, J\lhd R}$, $e$ is the unity of $R$), and \emph{semiprime} if 
${I^2\ne \{0\}}$ for all ${\{0\}\ne I\lhd R}$ (${I^2=I^{(1)}\ne \{e\}}$ 
for all ${\{e\}\ne I\lhd R}$), ${P\lhd R}$ is a \emph{prime ideal} 
(\emph{normal subgroup}) of $R$ if $R/P$ is prime. The set $\Spec(R)$ of 
all prime ${P\lhd R}$ is the \emph{prime spectrum of $R$}, 
${\prr(R)=\bigcap\limits_{P\in \Spec(R)} P}$ is the \emph{prime radical of 
$R$}, $\prr(R)$ is the smallest of ${I\lhd R}$, $R/I$ is semiprime,
\[
\prr(R)\ =\ \bigcup_{\alpha\geq 0} B_{\alpha}(R)\ =\ 
\{x=x_0\in R\mid \forall\, x_{i+1}\in (x_i)_R^2,\ i\geq 0,\ 
0\in \{x_i\}_{i=0}^{\infty}\, (e\in \{x_i\}_{i=0}^{\infty})\}\,,
\]
where ${\{B_{\alpha}(R)\mid \alpha\geq 0\}}$ is the Baer chain for the mapping 
${B=B_1: R\longmapsto B(R)}$, $B(R)$ is the sum (product) of all ${I\lhd R}$, 
${I^2=\{0\}}$ (${[I, I]=\{e\}}$), $B_{\alpha}(R)$ is the preimage of 
$B(R/B_{\alpha-1}(R))$ in $R$ for the non-limit $\alpha$, 
${B_{\alpha}(R)=\bigcup\limits_{\beta< \alpha} B_{\beta}(R)}$ for the limit 
$\alpha$, and in the group case $(B)_R$ is a normal subgroup of $R$ generated 
by ${B\subseteq R}$ (normal closure of $B$ in $R$) \cite[Proposition~4,~Theorem~6,~p.~192,~193]{ZShS}; \cite{Baer, Sch, Gol7}. 

Semiprime groups do not contain non-unitary nilpotent and solvable normal 
subgroups. It is possible to speak about the absence of non-zero nilpotent 
or (and) solvable ideals in semiprime algebras only in relation to 
individual classes of algebras. At the same time, in any algebra $R$ one can 
single out the sum $N(R)$ ($S(R)$) of all nilpotent (solvable) ideals and the 
smallest ideal $\prr_N(R)$ ($\prr_S(R)$) among ${I\lhd R}$, 
${N(R/I)=\{0\}}$ (${S(R/I)=\{0\}}$), which is equal to the union of a chain 
of ideals constructed similarly to the Baer chain for the mapping
${N: R\longmapsto N(R)}$ (${S: R\longmapsto S(R)}$), 
${\prr(R)\subseteq \prr_N(R)\subseteq \prr_S(R)}$ and for $R$ from the class 
of algebras closed under taking quotient algebras, whose non-zero nilpotent 
(solvable) ideals contain non-zero ideals with zero multiplication, 
${\prr(R)=\prr_N(R)}$ (${\prr(R)=\prr_S(R)}$).

The prime radical ${\prr: R\longmapsto \prr(R)}$, ${R\in \mathfrak{M}}$, 
on a class of $F$--algebras (groups) $\mathfrak{M}$ that is closed under 
taking ideals and homomorphic images is a radical in the Kurosh --- Amitsur 
sense on $\mathfrak{M}$ if and only if ${\prr=RN}$ on $\mathfrak{M}$, 
where the \emph{lower nil-radical} $RN$ is the lower radical defined by 
the class of $F$--algebras with zero multiplication (Abelian group) on 
the class of all $F$--algebras (groups). The examples in Addition 2 show 
that ${\prr\ne RN}$ on the class of Lie $F$--algebras if $F$ is not a 
$\mathbb{Q}$--algebra, and on the class of all groups. However, 
${\prr=RN}$ on the classes of alternative algebras and generalized 
special Lie algebras (Lie algebras with $PI$--algebras of multiplications) 
over any $F$ (Addition 1, \cite{BP}), the class of finite groups.

In a free non-associative $F$--algebra $F\langle X\rangle$ (free group 
$F(X)$) with a set of free generators ${X=\{x_i\}_{i=1}^{\infty}}$, we 
choose elements ${\{g_k(x_1, \ldots, x_{2^k})\}_{k=0}^{\infty}}$, where 
${g_0(x_1)=x_1}$ and further by induction 
${g_{k+1}(x_1, \ldots, x_{2^{k+1}})=g_k(x_1, \ldots, x_{2^k}) 
g_k(x_{2^k+1}, \ldots, x_{2^{k+1}})}$ 
($g_{k+1}(x_1, \ldots, x_{2^{k+1}})=[g_k(x_1, \ldots, x_{2^k}), 
g_k(x_{2^k+1}, \ldots, x_{2^{k+1}})]$) for any ${k\geq 0}$. 
Following \cite{Parf}, we call an $F$--algebra (group) $R$ \emph{weakly 
solvable} if for any set ${B\subseteq R}$, ${|B|< \infty}$, there is 
${k=k(B)\geq 0}$, ${g_k(B)=\{0\}}$ ($\{e\}$), where 
${g_k(B)=\{g_k(b_1, \ldots, b_{2^k})\mid b_i\in B\}}$. Further in the 
text for groups ${g_k=g^{gr}_k, k\geq 0}$. The class of weakly 
solvable $F$--algebras (groups) is a radical subclass of the class of 
all $F$--algebras (groups), contains all locally solvable 
$F$--algebras (groups) and $F$--algebras (groups) that have normal 
series of subalgebras (subgroups) with weakly solvable quotients. 
The lower radical it defines on the class of all $F$--algebras 
(groups), the \emph{weakly solvable radical} $T$, is special  
\cite{Parf, Gol7}, \cite[Theorems~2.8,~3.8,~Remark~3.10]{Gol1}. 
On the class of alternative $F$--algebras $T$ is equal to the locally 
nilpotent radical $LN$ (Addition 1).

An element of an algebra is \emph{nilpotent} if it generates a nilpotent 
subalgebra; \emph{nil-algebras} are algebras consisting of nilpotent 
elements. In the class of $F$--algebras with associative powers 
(associative one-generated subalgebras), nil-algebras form a radical 
subclass, the lower radical defined by it, the \emph{upper nil-radical} 
$\mathcal{N}$, is special on this class (extensions of $\mathcal{N}$ to 
the class of all $F$--algebras can be found in \cite{Gol1}). 

For right modules over an associative ring (algebra) $A$, the usual notations 
will be used: $\End(M)_A$ is the endomorphism ring of $M_A$, 
${\Hom(M, N)_A}$ is the homomorphism group of $A$--modules $M_A$ into $N_A$, 
${\Ann_A L=\{x\in A\mid L x=\{0\}\}}$ is the annihilator of the set 
${L\subseteq M_A}$. The submodule ${K_A\subseteq M_A}$ is \emph{essential} 
(\emph{large}) if ${K\cap K'\ne \{0\}}$ for all ${\{0\}\ne K'_A\subseteq M_A}$, 
and \emph{fully invariant} if ${\phi(K)\subseteq K}$ for all 
${\phi\in \End(M)_A}$. For any ${B\subseteq A}$, ${L\subseteq M_A}$, $L$ has 
no $B$--torsion if ${z B\ne \{0\}}$ (${B\cap \Ann_A z\ne B}$) for all 
${0\ne z\in L}$. In particular, for ${A=\mathbb{Z}}$, ${B=\{k\}}$, ${k\geq 1}$, 
\emph{algebras without $k$--torsion} are algebras without $k$--torsion in 
additive groups.   

In \cite{Gol10} we collect the necessary information about the constructions 
of the Martindale centroid $\CM(R)$, the extended centroid ${}_R C$, 
${\CM(R)\cong {}_R C}$, and the central closures $P(R)$, ${}_R Q$, 
${P(R)=\CM(R) R\cong {}_R Q={}_R C R}$, of a semiprime algebra $R$, 
the orthogonal completion of semiprime algebras and the right rings 
of quotients of a left exact (${\Ann_l A=\{0\}}$) associative ring $A$: its 
complete (maximal) right ring of quotients $Q(A)$, the right and symmetric 
Martindale rings of quotients $Q^r_m(A)$ and $Q^s_m(A)$, where up to 
isomorphism ${A\subseteq Q^s_m(A)\subseteq Q^r_m(A)\subseteq Q(A)=Q(Q(A))}$, 
$Q^s_m(A)$ and $Q^r_m(A)$ in this inclusion of subrings of $Q(A)$ have a 
common unity with $Q(A)$ and the center
\[
Z(Q^s_m(A))\ =\ Z(Q^r_m(A))\ =\ Z(Q(A))\ =\ \{q\in Q(A)\mid [q, A]=\{0\}\}\,.
\]
Like all rings of quotients of $A$, $Q(A)$, $Q^r_m(A)$, $Q^s_m(A)$ inherit 
its semiprimeness (primeness). If $A$ is a $F$--algebra, then $Q(A)$, 
$Q^r_m(A)$, $Q^s_m(A)$ are $F$--algebras and can be defined for $A$ as an 
$F$--algebra. If $A$ is semiprime, then ${Z(Q^s_m(A))={}_A C}$ and 
there is an isomorphism of the ${}_A C$--algebras 
${{}_A C A\subseteq Q^s_m(A)}$ and ${}_A Q$ that is identical on $A$. We 
will denote $Z(Q^s_m(A))$ by ${}_A C$ for the left exact, not necessarily 
semiprime $A$, and call ${}_A C$ the \emph{extended centroid of} $A$, 
the ${}_A C$--subalgebras ${{}_A C A, 
\hat{A}={}_A C+{}_A C A\subseteq Q^s_m(A)}$ the \emph{central closures of} 
$A$. This is consistent with the definitions of the extended centroid and 
central closure of a semiprime algebra and is justified by the fact that 
\[
Q(A)\ =\ Q({}_A C A)\ =\ Q(\hat{A})\,,\quad 
\hat{A}\ =\ {}_{\hat{A}} C \hat{A}\ =\ \hat{\hat A}\,,\quad 
{}_A C A\ =\ {}_{{}_A C A} C {}_A C A
\]
\cite[(1.14)]{Ut}. In particular, ${A=\hat{A}=Q^s_m(A)=Q^r_m(A)}$ for 
the simple $A$ with 1. 

The \emph{right classical ring of fractions} $Q_{cl}(A)$ of an associative 
ring $A$ is uniquely determined up to an isomorphism identical on $A$ 
by the conditions: $A$ is a subring of the ring $Q_{cl}(A)$ with 1, the 
non-zero-divisors of $A$ are invertible in $Q_{cl}(A)$, all elements of 
$Q_{cl}(A)$ have the form $x y^{-1}$, ${x, y\in A}$, $y$ is a 
non-zero-divisor of $A$. The existence of $Q_{cl}(A)$ is equivalent to 
the satisfaction in $A$ of the \emph{right Ore condition}: for any 
${x, y\in A}$, $y$ is a non-zero-divisor of $A$, there are 
${x', y'\in A}$, $y'$ is a non-zero-divisor of $A$, ${x y'=y x'}$ 
\cite[Proposition~1,~p.~180]{Lamb}, \cite[Proposition~9.1,~p.~477]{Fe}. 

If in $A$ with 1 and the maximality condition for direct sums of right ideals 
the essential right ideals contain non-zero-divisors, then $A$ is semiprime 
and ${Q_{cl}(A)=Q(A)}$ is classically semiprime. A ring with maximal 
conditions for direct sums of right ideals and right annihilator ideals is a 
\emph{right Goldie ring}. A ring $A$ with 1 has a right Artinian simple 
(classically semisimple) ring ${Q_{cl}(A)=Q(A)}$ if and only if $A$ is a 
prime (semiprime) right Goldie ring 
\cite[Theorem~9.9,~Corollary~9.10,~Exercise~9.4.3,~p.~483,~484,~481]{Fe}, \cite[Theorems~7.2.1--7.2.3,~p.~168--171]{Her}, 
\cite[Proposition~2,~p.~180]{Lamb}. If $A$ is a semiprime 
right Goldie ring with 1, then the decomposition into a finite direct sum 
${Q_{cl}(A)=Q_1\oplus \ldots \oplus Q_k}$, ${Q_i\cong M_{n_i}(D_i)}$, $D_i$ 
is a division ring, determines the decomposition of $A$ into a subdirect 
product of prime right Goldie rings ${e_i A\cong A/\Ann_A e_i}$, 
${a\longmapsto (e_1 a, \ldots, e_k a)}$, ${a\in A}$, where $M_n(B)$ is the 
ring of ${n\times n}$ matrices with coefficients from the ring $B$, ${n\geq 1}$, 
$e_i$ is the unity of ${Q_i=Q_{cl}(e_i A)}$, ${1=e_1+\ldots+e_k}$. 

An element $f$ of a free associative $F$--algebra $F_{Ass}\langle X\rangle$ 
with a set of free generators ${X=\{x_i\}_{i=1}^{\infty}}$ is called 
\emph{proper} if the coefficients of the irreducible record of $f$ generate 
in $F$ an ideal equal to $F$. An associative $F$--algebra $A$ is a 
\emph{$PI$--algebra} if the identity ${f=0}$ holds on $A$ for the proper 
${f\in F_{Ass}\langle X\rangle}$. Any $PI$--algebra is a $PI$--ring with the
identity ${\st_n^k=0}$ for some ${n, k\geq 1}$, where 
\[
\st_n\ =\ \st_n(x_1, \ldots, x_n)\ =\ \sum_{\sigma\in \mathfrak{S}_n} 
(-1)^{\sigma} x_{\sigma(1)}\cdots x_{\sigma(n)}
\]
is the standard polynomial of degree $n$, $\mathfrak{S}_n$ is the symmetric 
group of degree $n$, $(-1)^{\sigma}$ is the sign of ${\sigma\in \mathfrak{S}_n}$. 
On the alternative (right-alternative) $PI$--algebra in Addition 1 the identity 
${f=0}$ must hold, $f$ is an element of the free alternative (right-alternative) 
$F$--algebra $F_{Alt}\langle X\rangle$ ($F_{RAlt}\langle X\rangle$) with a 
proper image in $F_{Ass}\langle X\rangle$ under the action of the canonical 
epimorphism ${F_{Alt}\langle X\rangle (F_{RAlt}\langle X\rangle)\longrightarrow 
F_{Ass}\langle X\rangle}$, extending the identification 
${x_i\longmapsto x_i}$, ${i\geq 1}$.

The \emph{polynomial degree} $\pideg A$ of an associative $PI$--algebra $A$ is 
the integer part $[\frac{n}2]$, where $n$ is the smallest of all ${m\geq 1}$, 
${\st_m=0}$ is the identity of $A/\prr(A)$, ${\pideg A=0}$ for ${A=\prr(A)}$ and  
${\pideg A=\sqrt{\max\limits_{A\ne P\in \Spec(A)} \dim_{\CM(A/P)} P(A/P)}}$
for ${A\ne \prr(A)}$.

If $A$ is a prime $PI$--algebra, then $A$ is a left and right Goldie ring, 
${Z(A)\ne \{0\}}$, its ring of quotients 
${A S^{-1}=Q_{cl}(A)=Q(A)=\hat{A}\cong P(A)}$ relative to 
${S=Z(A)\setminus \{0\}}$ is a finite-dimensional central simple algebra 
over ${Z(A S^{-1})=Z(A) S^{-1}={}_A C\cong \CM(A)}$ \cite{Pos, Mar, Row1}, 
\cite[Lemma~7.3.2,~p.~173]{Her}, \cite[Theorem~1.7]{Cab4}, 
\cite[(1.15),~(5.2)~or~(5.1)]{Ut} 
(\cite[Proposition~7,~p.~163]{Lamb}), or 
\cite[Proposition~2,~p.~180]{Lamb} (${A S^{-1}\cong M_n(D)}$, where $D$ is a 
division ring and ${n \sqrt{\dim_{Z(D)} D}=\pideg A}$, is equal to its socle); 
a generalization of these conclusions to semi\-prime $PI$--algebras was obtained 
in \cite{Fish, Mart1, Mar3, ArmS}. A semiprime ring $A$ with 1 and an exact 
module with Krull dimension is a finite subdirect product of prime rings, and 
if such an $A$ is a $PI$--ring, then $A$ is a left and right Goldie ring, 
${Q_{cl}(A)=Q(A)}$ is a $PI$--ring and satisfies the same homogeneous 
identities as $A$ \cite{BMar, Mar1, Mar2}. The right classical ring of 
quotients of a $PI$--algebra relative to a multiplicative subsemigroup (if it 
exists) is a $PI$--algebra \cite{B5}.
If an associative ring $A$ is semiprime, $Q(A)\langle X\rangle_{{}_A C}$ is a 
free product of ${}_A C$--algebras $Q(A)$ and ${}_A C_{Ass}\langle X\rangle$, 
${f\in Q(A)\langle X\rangle_{{}_A C}}$, ${\phi(f)=0}$ for any ring homomorphism 
${\phi: Q(A)\langle X\rangle_{{}_A C}\longrightarrow Q(A)}$, 
${\phi(X)\subseteq A}$, ${\phi|_{Q(A)}=\Id_{Q(A)}}$, then ${f=0}$ is a 
\emph{generalized identity of $A$}, $A$ is a \emph{$GPI$--ring}. If ${c f\ne 0}$ 
for all ${0\ne c\in {}_A C}$, then ${f=0}$ is a \emph{strict generalized 
identity}, $A$ is a $SGPI$--ring \cite{B4}. For the prime $A$ all 
generalized identities are strict (${}_A C$ is a field). If $A$ is semiprime, 
then $Q(A)$ and any of its dense $A$--submodules (in particular, $A$) satisfy the 
same generalized and differential identities with coefficients from $Q(A)$ 
\cite{B3, Chua, Lee}. A prime $GPI$--ring $A$ satisfies generalized identities 
with coefficients from $A$ \cite{Chua}.

\section{Algebraic derivations of semiprime algebras}

This section is devoted to the connections between algebraic derivations of 
semiprime algebras and inner derivations of their rings of quotients. 
Derivations of the $F$--algebra $R$ form a subalgebra $\Der(R)$ ($\Der_F(R)$, 
if $F$ is required) of the Lie algebra $\End_F(R)^{(-)}$. In addition to 
$\Der(R)$, our constructions will involve its ideals  
${\Mder(R)=\Der(R)\cap M(R)^{(-)}}$ and ${\Inder(R)=\Der(R)\cap \Lie(R)}$, 
where ${\Lie(R)=\langle t_x\mid t=l, r,\ x\in R\rangle\subseteq M(R)^{(-)}}$ 
($\Inder(R)$ and $\Lie(R)$ are Lie algebras of inner derivations and Lie 
multiplications of $R$). 

Since ${\phi\in \Der(R)}$ for any ${\phi\in \End_F(R)}$, 
${R^2\subseteq \Ker \phi}$, in case ${R^n=\{0\}}$, ${n\geq 3}$, 
\[
\sum_{i=1}^k t^{(i)}_{x_{i1}}\cdots t^{(i)}_{x_{i n-2}}\in \Mder(R)\quad 
(x_{ij}\in R,\ t^{(i)}=l, r,\ k\geq 1)\,,
\]
and, as a consequence, 
${\Inder(R)=\Mder(R)=M(R)^{(-)}=\{l_x+r_y\mid x, y\in R\}}$ for ${n=3}$. If $R$ 
is a Lie algebra, then ${\Inder(R)=\ad(R)=\{\ad_x\mid x\in R\}}$ and in the 
presence of a prime ${p> 0}$, 
${p R=\{0\}}$, ${\ad_x^p\in \Mder(R)\setminus \Inder(R)}$ (${D^p\in \Der(R)}$ 
for any $R$, ${p R=\{0\}}$, ${D\in \Der(R)}$).

If $A$ is an associative algebra without ${0\ne x\in A}$, ${A x A=\{0\}}$ (in 
particular, $A$ with 1 or (and) semiprime), then  
${\Inder(A)=\{\ad_x=r_x-l_x\mid x\in A\}}$, since ${l_x+r_y\in \Der(A)}$, 
${x, y\in A}$, implies ${l_z\in \Der(A)}$, ${z=x+y}$, 
${z a b=z a b+a z b}$, ${a z b=0}$ for all ${a, b\in A}$. If $A$ with 1 is 
prime, then ${\Mder(A)\subseteq \{D\in \Inder(\hat{A})\mid (A)D\subseteq A\}}$.
To do this, it is sufficient to note that for any 
${D=\sum\limits_{i=1}^k l_{x_i} r_{y_i}\in \Mder(A)}$, ${x_i, y_i\in A}$, 
${k\geq 1}$, ${1 D=\sum\limits_{i=1}^k x_i y_i=0}$ and for all ${x, y, z\in A}$
\begin{gather*}
x D\ =\ x D-x \sum_{i=1}^k x_i y_i\ =\ \sum_{i=1}^k [x_i, x] y_i\ =\ 
\sum_{i=1}^k x_i [x, y_i]\,,
\\
(x y) D\ =\ \sum_{i=1}^k x_i ([x, y_i] y+x [y, y_i])\ =\ 
\sum_{i=1}^k (x_i [x, y_i] y+x x_i [y, y_i])\,,
\\
0\ =\ \sum_{i=1}^k [x_i, x] [y, y_i]\ =\ \sum_{i=1}^k [x_i, x z] [y, y_i]\ =\ 
\sum_{i=1}^k [x_i, x] z [y, y_i]\,.
\end{gather*}
Without loss of generality, we will assume that $\{y_i+{}_A C\}$ are linearly 
independent over the field ${}_A C$, since, if necessary, $\{y_i\}$ can be 
replaced by the preimages in $A$ of the elements of the maximal linearly 
independent over ${}_A C$ subsystem 
${\{y_i+{}_A C\}\subseteq (A+{}_A C)/{}_A C}$, and we can go from $\{x_i\}$ to 
their ${}_A C$--linear combinations in $\hat{A}$. We fix ${x\in A}$ and the 
maximal ${}_A C$--linearly independent subsystem 
${\{[x_{i_j}, x]\}_{j=1}^n\subseteq \{[x_i, x]\}}$, 
${[x_i, x]=\sum\limits_{j=1}^n \alpha_{i j} [x_{i_j}, x]}$, 
${\alpha_{i j}\in {}_A C, \alpha_{i_k j}=\delta_{k j}}$, where 
${\delta_{ij}=1}$ for ${i=j}$, ${\delta_{ij}=0}$, otherwise. Then 
${\sum\limits_{j=1}^n [x_{i_j}, x] z 
\Bigl[y, \sum\limits_{i=1}^k \alpha_{i j} y_i\Bigr]=0}$ for all ${y, z\in A}$ 
and, in view of \cite[Theorem~2(a)]{BMM} (\cite[Theorem~2.7]{ADN}), 
${x_1, \ldots, x_k\in {}_A C}$, ${D\in \Inder(\hat{A})}$.

The invariance of the ideal $I$ of the algebra $R$ with respect to the action 
of $\Der(R)$ ($\Aut_F(R)$) will be denoted by ${I\lhd_{\Der} R}$ 
(${I\lhd_{\Aut} R}$). In particular, ${I\lhd_{\Der} R}$ for all 
${I=I^2\lhd R}$. If ${I\lhd_{\Der} R}$ and ${\Ann_t I=\{0\}}$ for ${t=l}$ 
or (and) ${t=r}$, then ${D\longmapsto D|_I}$, ${D\in \Der(R)}$, is an 
embedding ${\Der(R)\hookrightarrow \Der(I)}$, since ${D|_I=0}$ implies 
\[
0\ =\ y t_x D\ =\ y t_{x D}+y D t_x\ =\ y t_{x D}\ =\ x D\quad 
(x\in R,\ y\in I)\,.
\]
If $F$ is a field, ${\Ch F=0}$, then $S(R)$, $\prr_S(R)$, ${T(R)\lhd_{\Der} R}$ 
\cite[the~observations~after~Remark 1.1]{Gol9}.

An element with associative powers $x$ of an $F$--algebra $R$ is called 
\emph{integral over $F$} (\emph{algebraic} for the field $F$) if there is 
${{}_x f(t)=t^n+{}_x f_{n-1} t^{n-1}+\ldots+{}_x f_1 t\in F[t]}$, 
${{}_x f(x)=0}$, where $F[t]$ ($F[X]$) is the algebra of polynomials with 
coefficients from $F$ in the variable $t$ (set of commuting variables $X$). 
An endomorphism $\phi$ of a $F$--module $M$ is called \emph{locally finite} 
(\emph{of degree at most ${n\geq 1}$}) if for any ${x\in M}$ there is 
${{}_{x, \phi} f(t)=t^{n_{x, \phi}}+{}_{x, \phi} f_{n_{x, \phi}-1} 
t^{n_{x, \phi}-1}+\ldots+{}_{x, \phi} f_1 t\in F[t]}$, 
${x {}_{x, \phi} f(\phi)=0}$ (${n_{x, \phi}\leq n}$), and  
\emph{locally nilpotent} if for any ${x\in M}$ there is 
${n_{x, \phi}\geq 1, x \phi^{n_{x, \phi}}=0}$. In particular, for a finitely 
generated $M$ (or (and) ${|F|< \infty}$), the local finiteness  
(of degree at most $n$) of $\phi$ is equivalent to the integrity of $\phi$ in 
$\End_F(M)$. Integrity, local finiteness and local nilpotency of $\phi$ on 
any set ${B\subseteq M}$ are defined similarly. We call 
${D\in \Der(R)}$ \emph{algebraic} (\emph{of degree at most $n$}), 
\emph{strongly algebraic} and a \emph{nil-derivation} if $D$ is a 
locally finite (of degree at most $n$), integral and locally nilpotent 
endomorphism of the $F$--module $R$, respectively. 

Denote by $\Der_n(R)$, $\Der_{nil}(R)$, $\Der_{sa}(R)$ and $\Der_a(R)$ 
($\Der_{a, n}(R)$) the sets of all nilpotent, nil-, strongly algebraic 
and algebraic (of degree at most $n$) ${D\in \Der(R)}$, by 
$\Der_{a, bd}(R)=\bigcup\limits_{n> 1} \Der_{a, n}(R)$ the set of all 
algebraic of bounded degree ${D\in \Der(R)}$, ${\Der_{a, 1}(R)=\{0\}}$, by 
$\Der_{l, a, bd}(R)$ and $\Der_{1, a, bd}(R)$ the sets of all algebraic of 
bounded degree on finitely gene\-rated and one-generated subalgebras 
${D\in \Der(R)}$, by $\Der_{1, n}(R)$ and $\Der_{1, sa}(R)$ the 
sets of all nilpotent and strongly algebraic on one-generated subalgebras 
${D\in \Der(R)}$.

If $R$ is finite over $F$ (finitely generated as a $F$--module), then 
${\Der_a(R)=\Der_{sa}(R)}$. If $R$ is locally finite over $F$ 
(the finitely generated subalgebras of $R$ are finite over $F$), then  
${\Der_a(R)=\Der_{l, a, bd}(R)=\Der_{l, sa}(R)}$ is the set of all 
strongly algebraic on finitely generated subalgebras ${D\in \Der(R)}$. 
In each $\Der_*(R)$ we select sets ${\Inder_*(R)=\Inder(R)\cap \Der_*(R)}$, 
${\Mder_*(R)=\Mder(R)\cap \Der_*(R)}$ ($\Inder_{*, F}(R)$, $\Mder_{*, F}(R)$, 
if you need to specify $F$).

\begin{lemma}
If ${n\geq 1}$, $R$ is an algebra over a ring $F$ with $1/n!$, ${I\lhd R}$ and 
in the algebra $R/I$ there are no non-zero nilpotent (solvable) ideals of 
index (degree) $k$ for some ${k\geq n}$ (${2^k\geq n}$), then 
${(I)D\subseteq I}$ for all ${D\in \Der_{a, n}(R)}$.
\end{lemma}
  
\begin{proof}
Due to ${w(x_1, \ldots, x_n, x_{n+1} D, \ldots, x_{n+l} D) D^n\equiv 
n! w(x_1 D, \ldots, x_{n+l} D)\equiv 0\pod I}$ for any\linebreak ${x_i\in I}$, 
${l\geq 0}$ and regular non-associative words 
${w\in F\langle X\rangle}$ of length ${n+l}$, 
\[
((I+(I)D)/I)^k\ =\ ((I+(I)D)/I)^{(k)}\ =\ \{0\}\,,\quad (I)D\subseteq I\,.
\]
\end{proof}

\begin{co} 
If $F$ is an algebra over a field $\mathbb{F}$, ${\Ch \mathbb{F}=0}$, then the 
ideals of the $F$--algebra $R$, in the quotient algebras by which there are no 
non-zero nilpotent (solvable) ideals, are invariant under the action of 
$\Der_{a, bd}(R)$.
\end{co}

\begin{co}
If $F$ is an algebra over a field $\mathbb{F}$, ${\Ch \mathbb{F}=0}$, then the 
ideals of the $F$--algebra $R$, in the quotient algebras by which there are no 
non-zero locally nilpotent (solvable) ideals, are invariant under the action 
of $\Der_{l, a, bd}(R)$, the ideals of $R$, in the quotient algebras by which 
there are non-zero nil-ideals, are invariant under the action of 
$\Der_{1, a, bd}(R)$.
\end{co} 

\begin{proof} 
If ${D\in \Der_{l, a, bd}(R)}$, ${I\lhd R}$, there are no non-zero locally 
nilpotent (solvable) ideals in the algebra $R/I$, then for any finitely 
generated subalgebra ${A\subseteq I}$ ($A$ can be conside\-red $D$--invariant, 
${D\in \Der_a(R)}$) $D$ is algebraic on $A$ of degree not higher than some 
${n\geq 1}$,\linebreak
${w(x_1, \dots, x_l) D^l\equiv l! w(x_1 D, \ldots, x_l D)\equiv 0\pod I}$ for 
all ${x_i\in \langle A\rangle}$, ${l\geq n}$, regular non-associative 
words ${w\in F\langle X\rangle}$ of length $l$, 
${\langle (A)D\rangle^n\subseteq I}$, ${(I)D\subseteq I}$. For 
${D\in \Der_{1, a, bd}(R)}$ and $R/I$ without non-zero nil-ideals the 
argument is similar for the one-generated $A$. 
\end{proof}

By definition, the center, associative and commutative centers of any algebra 
are invariant under the action of its derivations. 

\begin{rem}
For any $F$--algebra $R$ ${[\Der(R), \End(R)_{M(R)'}]\subseteq \End(R)_{M(R)'}}$.
\end{rem}

\begin{proof}
If ${D\in \Der(R)}$, ${\phi\in \End(R)_{M(R)'}}$, ${t=l, r}$, ${x\in R}$, then 
\[
0\ =\ [\phi, t_x] \ad_D\ =\ [[\phi, t_x], D]\ =\ 
[[\phi, D], t_x]+[\phi, t_{x D}]\ =\ [[\phi, D], t_x]\,, 
\]
${\ad_D\in \ad(\End_F(R)^{(-)})}$ and hence 
${[\Der_F(P(R)), \CM(R)]\subseteq \CM(R)=\End(P(R))_{M(R)'}}$ 
for the semiprime $R$.
\end{proof}

\begin{rem} 
If an $F$--algebra $R$ has no non-zero nilpotent elements and $c_n$--torsion, 
$c_n=\prod\limits_{k=1}^{n-1} \binom{2 k}k$, for some ${n\geq 2}$, then 
${(\langle x\rangle)D=0}$ for all ${x\in R}$, ${D\in \Der(R)}$, 
${(\langle x\rangle)D^n=0}$, and hence ${\Der_n(R)=\Der_{1, n}(R)=\{0\}}$ 
for $R$ without torsion.
\end{rem} 

\begin{proof}
An algebra has no non-zero nilpotent elements if and only if it has no non-zero 
elements with zero square. If ${2\leq k\leq n, 
(\langle x\rangle) D^k=\{0\},}$ ${y\in \langle x\rangle}$, ${y D^{k-1}\ne 0}$, 
then ${2 (k-1)\geq k}$, ${0=y^2 D^{2 (k-1)}=\binom{2 (k-1)}{k-1} (y D^{k-1})^2=y D^{k-1}=0}$, a contradiction. Hence, 
${(\langle x\rangle)D=\{0\}}$. 
\end{proof}

\begin{rem}
If an $F$--algebra $R$ is semiprime, then the injective 
${\phi\in \End(R)_{M(R)'}}$ extend to invertible 
${\overline{\phi}\in \CM(R)}$. For a weakly Artinian $R$ 
(or (and) ${R=P(R)}$) the following conditions are equivalent: 
\begin{enumerate}[label=(2.6.\arabic*), ref=2.6.1, leftmargin=*]

\item all ${0\ne \phi\in \End(R)_{M(R)'}}$ are injective; 

\item $\End(R)_{M(R)'}$ is a field; 

\item $R$ is not decomposable into a direct sum of ideals 
(prime).

\end{enumerate}
\end{rem}

\begin{proof} 
Each ${\phi\in \End(R)_{M(R)'}}$ can be uniquely extended to
${\overline{\phi}\in \CM(R)}$, $\Ker \phi=R\cap \Ker \overline{\phi}=\{0\}$ 
is equivalent to ${\Ker \overline{\phi}=\{0\}}$ and there is ${\psi\in \CM(R)}$, 
${\overline{\phi}=\psi \overline{\phi}^2=\overline{\phi}^2 \psi}$ 
($R$ is an essential $M(R)'$--submodule of $P(R)$ and $\CM(R)$ is regular and 
commutative). If ${\Ker \phi=\{0\}}$, then 
$\Ker \psi=\Ker \overline{\phi} \psi=\{0\}$, 
${\overline{\phi} \psi (\Id_{P(R)}-\overline{\phi} \psi)=0}$, 
${\psi=\overline{\phi}^{-1}}$. Therefore, $\phi$ is injective if and only if 
$\overline{\phi}$ is invertible. From here and the invertibility of 
injective endomorphisms of Artinian modules it follows that 
${(2.6.1)\Longleftrightarrow (2.6.2)}$ for the weakly Artinian (with the minimality 
condition for ideals) $R$ or (and) ${R=P(R)}$. 
In any case ${(2.6.1)\Longrightarrow (2.6.3)}$, since for ${R=I_1\oplus I_2}$, 
${\{0\}\ne I_i\lhd R}$, ${\pi_i\in \End(R)_{M(R)'}}$, 
${\pi_i: x_1+x_2\longmapsto x_i, x_i\in I_i,}$ ${\Ker \pi_i=I_j}$, 
${j\in \{1, 2\}\setminus i}$. 

If $R$ is prime, then $P(R)$ is prime, ${\CM(R)=\CM(P(R))}$ is a field and 
$(1)$ is satisfied \cite[Proposition~3.2,~p.~44]{Raz}. In the non-prime 
$R$ there are ${\{0\}\ne I, J\lhd R}$, ${I\oplus J\subseteq R}$, 
where,\linebreak without loss of generality, ${I\oplus J\in \mathcal{E}(R)}$, 
${\pi: x+y\longmapsto x}$, ${x\in I}$, ${y\in J}$, can be extended to 
${\overline{\pi}=\overline{\pi}^2\in \CM(R)}$,  
${\CM(R) I\oplus \CM(R) J\subseteq P(R)=\overline{\pi}(P(R))\oplus 
(\Id_{P(R)}-\overline{\pi})(P(R))}$ for $\CM(R) I$, ${\CM(R) J\lhd P(R)}$, 
${\CM(R) I\subseteq \overline{\pi}(P(R))}$, 
${\CM(R) J\subseteq (\Id_{P(R)}-\overline{\pi})(P(R))}$. So, the primeness 
of $R$, the primeness of $P(R)$, and conditions $(2.6.i)$, 
${i=1, 2, 3}$ for $P(R)$ are 
equivalent.

For any ${\phi\in \End(R)_{M(R)'}}$, ${\Ker \phi, \phi(R)\lhd R}$ and  
\begin{gather*}
\{0\}\ =\ \Ker \phi\ \phi(R)\ =\ \phi(\Ker \phi\ R)\ =\ 
\phi(R)\, \Ker \phi\ =\ \phi(R\, \Ker \phi)\ =\ 
\Ker \phi\cap \phi(R)\,,
\\
\Ker \phi^2\ =\ \{x\in R\mid \phi x\in \Ker \phi\cap \phi(R)=\{0\}\}\ =\ 
\Ker \phi\,,\quad \phi\,:\ \phi(R)\hookrightarrow \phi(R)\,.
\end{gather*}
If ${\phi(R)=\phi^2(R)}$, then for any ${x\in R}$ there exists ${y\in R}$, 
${\phi x=\phi^2 y}$, ${x-\phi y\in \Ker \phi}$, and  
${R=\Ker \phi\oplus \phi(R)}$. The latter is true for an Artinian 
$M(R)'$--module $\phi(R)$ and therefore for a weakly Artinian $R$
(invertibility of injective endomorphisms of Artinian modules). Thus, 
${(2.6.1)\Longleftrightarrow (2.6.2)\Longleftrightarrow (2.6.3)}$ for the weakly 
Artinian $R$.
\end{proof}

If an $F$--algebra $R$ is semiprime, then the orders of the non-zero elements 
of the periodic part of its additive group 
${\Tor(R)=\{x\in R\mid \exists k> 0,\ k x=0\}}$ are square-free, 
$\Tor(R)=\bigoplus\limits_{p\in \mathcal{P}} R_p$,  
${R_p=\{x\in R\mid p x=\{0\}\}\lhd_{\Der} R}$, $\mathcal{P}$ is the set 
of all primes ${p> 0}$. The satisfaction of\linebreak (2.6.1) (for 
${\pi_p: x\longmapsto p x}$, ${x\in R}$, ${p\in \mathcal{P}}$) in $R$ and, 
in particular, the primeness of $R$ guarantee that $R$ has characteristic 
${p\geq 0}$, ${R=R_p}$, ${p\in \mathcal{P}}$, or ${\Tor(R)=\{0\}}$, ${p=0}$. 
We complement $\Tor(P(R))$ and all $P(R)_p$, ${p\in \mathcal{P}}$, 
in $P(R)$ to an essential $M(R)'$--submodule 
\[
\Tor(P(R))\oplus M\ =\ P(R)_p\oplus M_p\,,\quad 
M_p\ =\ M\oplus \bigoplus_{p\ne q\in \mathcal{P}} P(R)_q\,,
\]
extend the projectors 
\[
\Tor(P(R))\oplus M\ \longrightarrow\ \Tor(P(R))\,,\quad 
P(R)_p\oplus M_p\ \longrightarrow\ P(R)_p
\] 
uniquely to ${e=e^2, e_p=e_p^2\in \CM(R)}$, 
${p e_p=0}$, ${e_p(P(R))=P(R)_p}$, ${e_p e_q=\delta_{p q} e_p, 
p, q\in \mathcal{P}}$, and put ${e_0=e_0^2=\Id_{P(R)}-e}$, ${P(R)_0=e_0(P(R))}$, 
${e_p e=e_p}$, ${e_p e_0=0}$, ${p\in \mathcal{P}}$, ${\Tor(P(R)_0)=\{0\}}$. 
If ${0\ne x, e x\in P(R)}$, then there exists 
${\phi\in M(R)', 0\ne e x \phi\in \Tor(P(R))\oplus M}$, either 
$e x \phi\in \Tor(P(R))$ and there is 
${p\in \mathcal{P}}$, ${e_p x \phi=e_p (e x \phi)\ne 0}$, ${e_p x\ne 0}$, 
or one can choose ${k> 0}$, $0\ne e k x \phi\in M$, 
${e (e k x \phi)=e k x \phi=0}$, a contradiction. Therefore $\Tor(P(R))$ is an essential 
$M(R)'$--submodule of $e P(R)$, for ${x\in P(R)}$ ${e_q x=0}$ for all 
${q\in \mathcal{P}_0=\mathcal{P}\cup \{0\}}$ is equivalent to ${x=0}$, 
\begin{multline*}
\iota\,:\ P(R)\hookrightarrow \prod_{q\in \mathcal{P}_0} P(R)_q\cong 
O_E(P(R))\ =
\\ 
\{x\in O(P(R))\mid \exists \{x_q\}_{q\in \mathcal{P}_0}\subseteq P(R),\ 
e_q x=e_q x_q\ \forall\, q\in \mathcal{P}_0\}\,,
\end{multline*}
where ${\iota(x): q\longmapsto e_q x}$, ${x\in P(R)}$, ${q\in \mathcal{P}_0}$, 
$P(R)_q$ is a semiprime algebra of characteristic ${q\in \mathcal{P}_0}$, 
$O(P(R))$ is the orthogonal completion of $P(R)$ \cite{Gol10}. If the 
orders of elements of $\Tor(P(R))$ are bounded, then 
${|\{p\in \mathcal{P}\mid P(R)_p\ne \{0\}\}|< \infty}$, 
${e=\sum\limits_{p\in \mathcal{P}} e_p}$, 
${P(R)=\bigoplus\limits_{q\in \mathcal{P}_0} P(R)_q}$. Conditions 
${\Tor(R)=\{0\}}$, ${\Tor(P(R))=\{0\}}$ and ${\Tor(\CM(R))=\{0\}}$ 
are equivalent to each other, since ${\Tor(R)=R\cap \Tor(P(R))}$, 
${R_p=R\cap P(R)_p}$, ${p\in \mathcal{P}}$. 

For a semiprime $R$ with (2.6.1) the local finiteness 
of ${\phi\in \End(R)_{M(R)'}}$ is equivalent to the integrity (algebraicity) 
of $\phi$ over $F$. If $F$ is an algebraically closed field, then, in view 
of the primeness of $\End(R)_{M(R)'}$, such $\phi$ is included in $F \Id_R$. 

If a semiprime $R$ is weakly Noetherian (with the maximality condition for 
ideals) and ${\phi\in \End(R)_{M(R)'}}$ is locally finite over $F$, then 
among the ideals ${\{\Ker f(\phi)\mid f(t)\in F_1[t]\}}$ one can choose the 
maximal ideal $\Ker h(\phi)$, where $F_1[t]$ is the set of all 
${f(t)\in F[t]}$ with leading coefficient 1, for any ${x\in R}$ there is 
${{}_{x, \phi} f(t)\in F_1[t]}$, ${{}_{x, \phi} f(\phi) x=0}$, 
$\Ker {}_{x, \phi} f(\phi)\cup \Ker h(\phi)\subseteq 
\Ker {}_{x, \phi} f h(\phi)=\Ker h(\phi)$, 
${h(\phi) x=0}$, ${h(\phi)=0}$, $\phi$ is integral over $F$.

\begin{lemma}
If $R$ is a semiprime $F$--algebra and ${D\in \Der(R)}$ with one of the 
conditions: 
\begin{enumerate}[label=(2.7.\arabic*), ref=2.7.1, leftmargin=*]

\item $\End(R)_{M(R)'}$ without torsion ($c_n$--torsion for some ${n\geq 1}$, 
${c_1=1}$),\\
${\ad_D\in \Der_{1, n}(\End(R)_{M(R)'})}$ (${(\End(R)_{M(R)'})\ad_D^n=\{0\}}$); 

\item ${1\in R}$, $Z(R)$ without torsion ($c_n$--torsion for some ${n\geq 1}$), 
${D\in \Der_{1, n}(Z(R))}$\\
(${(Z(R))D^n=\{0\}}$); 

\item $R$ is torsion-free and satisfies (2.6.1) or (and) is weakly 
Noetherian, $\End(R)_{M(R)'}$ is integral over $F$, ${D\in \Der_{nil}(R)}$,

\end{enumerate}
then ${[D, \End(R)_{M(R)'}]=\{0\}}$.
\end{lemma}

\begin{proof} 
The inclusions of (2.7.1), (2.7.2) refers to the restrictions of $\ad_D$ to 
$\End(R)_{M(R)'}$ and $D$ to $Z(R)$. In view of the commutativity and 
semiprimeness (primeness for the prime $R$) of $\End(R)_{M(R)'}$ 
\cite[Propositions~3.1,~3.2,~p.~43,~44]{Raz} 
(${\End(R)_{M(R)'}\hookrightarrow \CM(R)}$), (2.7.1) follows immediately 
from Remark 2.5. If ${1\in R}$, then for any ${\psi\in \End(R)_{M(R)'}}$
\[
\psi x\ =\ \psi (1 t_x)\ =\ (\psi 1) t_x\ =\ x t_{\psi 1}\quad 
(x\in R,\ t=l, r)\,,
\]
${\psi 1\in Z(R)}$, ${\psi=t_{\psi 1}}$ (${\phi\longmapsto \phi 1}$, 
${\phi\in \End(R)_{M(R)'}}$, is an isomorphism of $\End(R)_{M(R)'}$ and 
$Z(R)$), ${\psi \ad_D=t_{(\psi 1) D}}$, ${t=l, r}$, and hence (2.7.2) 
follows from the semiprimeness $Z(R)$ and Remark 2.5 (here the semiprimeness 
of $R$ with 1 can be replaced by the semiprimeness of $Z(R)$). 
If ${D\in \Der_{nil}(R)}$, then for any ${\psi\in \End(R)_{M(R)'}}$, 
${x\in R}$ there are ${n, m\geq 1}$, $x D^n=(\psi (x D^i)) D^m=0$, 
${i=0, \ldots, n-1}$, and for all ${l\geq n+m-1}$ 
\[
(\psi \ad_D^l) x\ =\ \biggl(\sum_{i=0}^l \binom{l}i (-1)^i D^i 
\psi D^{l-i}\biggr) x\ =\ \sum_{i=0}^l \binom{l}i (-1)^i (\psi (x D^i)) D^{l-i}
\ =\ 0\,,
\]
where $\ad_D$ acts on $\End(R)_{M(R)'}$ as a $F$--subalgebra of $\End_F(R)$, 
and the action on the element is written in the previously adopted order. 
For ${(\psi \ad_D) x\ne 0}$ there exists ${2\leq k< m+n-1}$, 
${0=(\psi \ad_D^l) x\ne (\psi \ad_D^{k-1}) x}$ for all ${l\geq k}$,
\begin{multline*}
(\psi^p \ad_D^{p (k-1)}) x\ =\ \frac{(p (k-1))!}{((k-1)!)^p} 
(\psi \ad_D^{k-1})^p x\ =
\\
\sum_{i=0}^{\max\{n-1,\ p (k-1)\}} 
\binom{p (k-1)}i (-1)^i (\psi^p(x D^i)) D^{p (k-1)-i}\quad
(p\geq 1)\,,
\end{multline*}
so for a locally finite $\psi$ one can choose ${p\geq 1}$, 
${(\psi \ad_D^{k-1})^p x=0}$. If $R$ with (2.6.1), 
${x\ne 0}$, then ${(\psi \ad_D^{k-1})^p=\psi \ad_D^{k-1}=0}$, a contradiction. For the weakly Noetherian $R$ there is 
${s\geq 1}$, 
${R=\bigcup\limits_{q\geq 1} I_q=I_s}$, 
${I_q=\bigcap\limits_{i\geq q} \Ker \psi \ad_D^i\lhd R}$, 
${I_q\subseteq I_{q+1}}$. For ${s\geq 2}$, ${I_{s-1}\ne I_s}$, the 
previous reasoning with ${k=s}$ implies the presence of 
${t\geq 1}$, ${R=\bigcup\limits_{l\geq 1} J_l=J_t}$, 
${J_l=\Ker (\psi \ad_D^{s-1})^l\lhd R}$, ${J_l\subseteq J_{l+1}}$, ${(\psi \ad_D^{s-1})^t=\ad_D^{s-1} \psi=0}$, but 
${I_{s-1}\ne R}$, a contradiction. So, ${R=I_1}$, ${\psi \ad_D=0}$, for (2.7.3)
${[D, \End(R)_{M(R)'}]=\{0\}}$. 
The conclusion for (2.7.3) is also possible from (2.7.1), since in this case ${\ad_D\in \Der_{1, n}(\End(R)_{M(R)'})}$. Indeed, for any locally 
finite ${\psi\in \End(R)_{M(R)'}}$ and ${x\in R}$ one can find 
${p, q\geq 1}$, ${(\langle \psi\rangle (x D^i)) D^p=
((\langle \psi\rangle)\ad_D^t) x=\{0\}}$ for all 
${i\geq 0, t\geq q}$ (see above). For $R$ with (2.6.1) 
this implies ${(\langle \psi\rangle)\ad_D^q=\{0\}}$. For 
the weakly Noetherian $R$ there is ${j\geq 1}$, 
${R=\bigcup\limits_{k\geq 1} H_k=H_j}$, 
${H_k=\bigcap\limits_{l\geq 1,\ m\geq k} \Ker \psi^l \ad_D^m\lhd R}$, 
${H_k\subseteq H_{k+1}}$, and ${(\langle \psi\rangle)\ad_D^j=\{0\}}$. 
\end{proof}
                    
\begin{lemma} 
If a field $\mathbb{F}$ is algebraically closed, ${\Ch \mathbb{F}=0}$ and 
an $\mathbb{F}$--algebra $R$ has no non-zero nilpotent elements, then 
${\Der_{1, sa}(R)=\Der_{1, n}(R)=\{0\}}$. As a consequence, if $R$ is a 
semi-\linebreak prime $\mathbb{F}$--algebra and ${D\in \Der(R)}$ with one 
of the conditions: 
\begin{enumerate}[label=(2.8.\arabic*), ref=2.8.1, leftmargin=*]

\item ${\ad_D\in \Der_{1, sa}(\End(R)_{M(R)'})}$; 

\item ${1\in R}$ and ${D\in \Der_{1, sa}(Z(R))}$; 

\item $R$ is weakly Noetherian, $\End(R)_{M(R)'}$ is algebraic over 
$\mathbb{F}$, ${D\in \Der_a(R)}$, 

\end{enumerate}
then ${[D, \End(R)_{M(R)'}]=\{0\}}$. 
\end{lemma}

\begin{proof}
Let $R$ has no non-zero nilpotent elements and 
${D\in \Der_{1, sa}(R)}$. For any ${0\ne x\in R}$ we choose 
${f(t)=\prod\limits_{i=1}^n (t-\alpha_i)\in \mathbb{F}[t]}$ of the 
minimal degree ${n\geq 1}$ among all ${0\ne g(t)\in \mathbb{F}[t]}$, 
${(\langle x\rangle)g(D)=\{0\}}$, ${\alpha\in \{\alpha_i\}\subset \mathbb{F}}$, 
${0\ne y\in (\langle x\rangle)f_{\alpha}(D)}$, 
${f(t)=(t-\alpha) f_{\alpha}(t)}$, ${y D=\alpha y}$, a set of regular 
non-associative words 
${\{w_k\}_{k=1}^{\infty}\subset \mathbb{F}\langle X\rangle}$ of lengths 
$n_k$, ${n_{k+1}> n_k}$, $y_k=w_k(\underbrace{y, \ldots, y}_{n_k})\ne 0$ 
(${\langle y\rangle^l\ne \{0\}}$, ${l\geq 1}$). 
Then ${y_k D=n_k \alpha y_k}$ and ${y_k h(D)=h(n_k \alpha) y_k=0}$ for any 
${0\ne h(t)\in \mathbb{F}[t]}$, ${(\langle y\rangle)h(D)=\{0\}}$, 
${k\geq 1}$. Therefore ${\alpha=0, (\langle x\rangle)D^n=\{0\}, D=0}$ 
(Remark 2.5). This, together with Lemma 2.7, implies the desired conclusion for (2.8.1), (2.8.2). 

For any $\mathbb{F}$--algebra $R$, ${D\in \Der_a(R)}$, locally finite 
${\psi\in \End(R)_{M(R)'}}$ and ${x\in R}$ 
\[
W\ =\ \sum\limits_{i, j\geq 0} \mathbb{F} (\psi^j \ad_D^i)\subseteq 
\End(R)_{M(R)'}\,,\quad 
V\ =\ \sum\limits_{i, j, k\geq 0} \mathbb{F} D^i \psi^j D^k
\subseteq \End_{\mathbb{F}}(R)\,, 
\]
there is ${s\geq 0}$,
\begin{gather*}
W'\ =\ W\cap \bigcap_{l\geq 0} \Ann_{\End(R)_{M(R)'}} x D^l\ =\ 
W\cap \bigcap_{l=0}^s \Ann_{\End(R)_{M(R)'}} x D^l\,,\quad 
(W')\ad_D\subseteq W'\,,
\\
W/(W\cap \Ann_{\End(R)_{M(R)'}} x D^l)\cong W(x D^l)\ =\ 
\{w (x D^l)\mid w\in W\}\subseteq V(x D^l)\,,
\\
W/W'\hookrightarrow \bigoplus\limits_{l=0}^s 
W/(W\cap \Ann_{\End(R)_{M(R)'}} x D^l)\,,\quad 
\dim_{\mathbb{F}} W/W'\leq 
\sum\limits_{l=0}^s \dim_{\mathbb{F}} V(x D^l)< \infty\,,
\end{gather*}
and for ${A\in \End_{\mathbb{F}}(W/W')}$, 
${A: w+W'\longmapsto w \ad_D+W'}$, ${w\in W}$, there exists 
${0\ne p(t)\in \mathbb{F}[t]}$, 
${q(A)=0, ((\langle \psi\rangle)q(\ad_D))(x)=\{0\}}$ for all 
${q(t)\in p(t) \mathbb{F}[t]}$. Therefore, under (2.8.3), 
${R=T_z}$ for the maximal $T_z$ among all 
${T_u=\bigcap\limits_{l\geq 1,\ v(t)\in u(t) \mathbb{F}[t]} 
\Ker \psi^lv(\ad_D)\lhd R}$, ${0\ne u(t)\in \mathbb{F}[t]}$ 
($T_u\cup T_v\subseteq T_{u v}$, ${0\ne u(t), v(t)\in \mathbb{F}[t]}$), 
(${\langle \psi\rangle)z(\ad_D)=\{0\}}$ and the conclusion for (2.8.3) reduces to that for (2.8.1).
\end{proof}

\begin{rem}
If $A$ is an associative $F$--algebra and ${x\in A}$ is algebraic over $F$, 
then $\ad_x\in \Der_{sa}(A)$.
\end{rem}

\begin{proof}
Without loss of generality, we can assume that $A$ with $1$. We select 
in $F$ a subring $H$ with 1 generated by the coefficients of 
${f(t)=t^n+f_{n-1} t^{n-1}+\ldots+f_0\in F[t]}$, ${f(x)=0}$. Since, by 
Hilbert's basis theorem, $H$ is Noetherian, the finitely generated 
$H$--module 
\[
M\ =\ {}_H\langle x\rangle^1\mathbin{\otimes_H} {}_H\langle x\rangle^1=
\sum\limits_{0\leq i, j\leq n-1} H x^i\otimes x^j\,,
\]
${{}_H\langle x\rangle^1=\sum\limits_{i\geq 0} H x^i}$, ${x^0=1}$, 
is Noetherian and all its submodules, including 
${V=\sum\limits_{k\geq 0} H x_k}$ for  
${x_k=\sum\limits_{i=0}^k \binom{k}i (-1)^i x^i\otimes x^{k-i}}$, 
are finitely generated,
${x_m+g_{m-1} x_{m-1}+\ldots+g_0=0}$ for some ${m\geq 1}$, 
${g_i\in H}$. For any ${y\in A}$ 
\[
\phi_y\,:\ \sum_{0\leq i, j\leq n-1} h_{ij} x^i\otimes x^j\longmapsto 
\sum_{0\leq i, j\leq n-1} h_{ij} x^i y x^j\quad (h_{ij}\in H) 
\]
is an epimorphism of $H$--modules $M$ and  
${\sum\limits_{0\leq i, j\leq n-1} H x^i y x^j\subseteq A}$ 
\cite[Sec.~11.1,~p.~513]{Fe}. So, 
\[
\phi_y(x_m+g_{m-1} x_{m-1}+\ldots+g_0)\ =\ y g(\ad_x)\ =\ 0
\quad (y\in A)\,,
\]
${g(\ad_x)=0}$, where ${g(t)=t^m+g_{m-1} t^{m-1}+\ldots+g_0\in F[t]}$.
\end{proof}

\begin{rem}
If $A$ is a semiprime associative ring, ${x\in A}$ and $\ad_x$ is 
strongly algebraic over ${}_A C$ of degree at most ${n\geq 1}$, 
then $x$ is algebraic over ${}_A C$ of degree at most $n$. 
\end{rem}

\begin{proof}
Without loss of generality, we can assume that ${A=\hat{A}}$, ${Z(A)={}_A C}$, 
${f(\ad_x)=0}$, ${f(t)=t^n+f_{n-1} t^{n-1}+\ldots+f_0\in {}_A C[t]}$. Then for 
some integer combinations ${c_{ij}\in {}_A C}$ of the coefficients of $f(t)$ 
\[
0\ =\ (-1)^n y f(\ad_x)\ =\ 
x^n y+\sum_{0\leq i< n,\ 0\leq j\leq n-i} c_{ij} x^i y x^j
\quad (y\in A)
\]
and, due to \cite[Corollary~2.9]{ADN}, there is 
${g(t)=t^n+c_{n-1} t^{n-1}+\ldots+c_0\in {}_A C[t]}$, ${g(x)=0}$.
\end{proof}
           
Remarks 2.9, 2.10 contain the equivalence for a semiprime associative ring 
$A$ of the algebraicity of ${x\in A}$ and the strong algebraicity of $\ad_x$ 
over ${}_A C$ from \cite[Proposition~1]{BGr}, \cite[Theorem~1.4]{ChuaL}. If $A$ is 
torsion-free, then the niplotency of $\ad_x$ is equivalent to the nilpotency\linebreak 
of ${x-\alpha}$ for some ${\alpha\in {}_A C}$ \cite[Theorem~1.3]{LeeA}, 
\cite[Corollary~4.5]{ADN} (torsion-freeness in $A$ is used only in deriving 
the nilpotency of ${x-\alpha}$ for nilpotent $\ad_x$ and is replaced 
by $n!$--torsion-freeness, $n$ is the nilpotency index of $\ad_x$; 
${\ad_x^{2 n-1}=0}$ for ${x^n=0}$). If $R$ is an $F$--algebra with 
${\Ann R=\Ann_l R\cap \Ann_r R=\{0\}}$, ${D\in \Der(R)}$, then from 
${\ad_D\in \Der_{sa}(M(R))}$, ${f(\ad_D)=0}$ for some ${f(t)\in F_1[t]}$, 
it follows that ${D\in \Der_{sa}(R)}$, ${t_x f(\ad_D)=t_{x f(D)}=0}$ for all 
${x\in R}$, ${f(D)=0}$. In particular, if ${\ad_D^n=0}$ for some 
${n\geq 1}$, then ${D^n=0}$.

Each derivation of a left exact associative ring $A$ can be uniquely extended 
to a derivation of its rings of quotients $Q(A)$, $Q^r_m(A)$ and $Q^s_m(A)$ 
\cite[Lemma~2]{Lee}, \cite[Preliminaries]{Har1}. Following 
\cite{Har2, Har3}, we will call ${D\in \Der(A)}$ 
\emph{strictly algebraic over a subring ${H\subseteq Q(A)}$} if there exist 
${n\geq 1}$, ${h_1, \ldots, h_n\in H}$ such that 
${\sum\limits_{i=0}^n D^i l_{h_i}=0}$ on $A$ and the ideal 
$(h_1, \ldots, h_n)_H$ has a zero annihilator in $A$. Due to the coincidence 
of the differential identities of the semiprime $A$ and $Q(A)$ with 
coefficients from $Q(A)$ \cite{Lee}, the condition of strict 
algebraicity of ${D\in \Der(A)}$ over $H$ is inherited (with the same $\{h_i\}$) 
by the extension of $D$ to $Q(A)$. The extension to $Q(A)$ of any nilpotent 
(strongly algebraic over ${}_A C$) ${D\in \Der(A)}$ is nilpotent (strongly 
algebraic over ${}_A C$) with the same nilpotency index 
(annihilating normalized polynomial without free term over ${}_A C$) 
\cite[Preliminaries]{ChuaL}. For the torsion-free semiprime $A$, the 
extension to $Q^s_m(A)$ of any strictly algebraic ${D\in \Der(A)}$ over 
$Q^s_m(A)$ is inner for $Q^s_m(A)$, ${D=\ad_q}$ on $Q^s_m(A)$ for some 
${q\in Q^s_m(A)}$ (algebraic over ${}_A C$ for strongly algebraic $D$ over 
${}_A C$) \cite[Corollary~1]{Har2}, \cite[Theorem~20]{Har3}. The 
realization of nilpotent ${D\in \Der(A)}$ as inner for nilpotent elements 
of $Q^s_m(A)$ is considered in \cite[Corollary~8]{Grz}, 
\cite[Theorem~3.6]{ChuaL}. The study of the properties of elements of semiprime rings 
corresponding to nilpotent and strongly algebraic inner derivations is the 
subject of \cite{ADN, ChuaL}. 

A semiprime algebra $R$ with a semiprime (prime) algebra $M(R)$ is 
called \emph{multiplicatively semiprime} (\emph{prime}), hereinafter a 
\emph{$m.s.p.$--algebra} (\emph{$m.p.$--algebra}). Each 
$m.p.$--algebra is prime. From the proof of \cite[Theorem~4.3]{Cab3}, 
it follows that 
\[
P(M(R))\cong M(P(R))\,,\quad \CM(M(R))\ =\ 
\CM(R) \Id_{P(M(R))}\cong \CM(R)
\]
for any $m.s.p.$--algebra $R$ (in the original version 
${P(M(R)')\cong M(P(R))'}$ (for the $\CM(R)$--algebra $P(R)$), 
${\CM(M(R)')\cong \CM(R)}$ for $R$ with semiprime $M(R)'$ (and $M(R)$)).

If $R$ is a $m.s.p.$--algebra, then any ${D\in \Der_F(R)}$ extends to 
${\overline{D}\in \Der_F(P(R))}$ by the rule:\linebreak 
${(\phi x) \overline{D}=\phi (x D)+(\phi \ad_D^{\tau}) x}$, ${x\in R}$, 
${\phi\in \CM(R)}$, where ${\tau: \CM(R)\longrightarrow {}_{M(R)} C}$ is a 
$F$--isomor\-phism of ${\CM(R)\cong \CM(M(R))}$ and ${}_{M(R)} C$, 
$\ad_D$ is the extension of ${\ad_D\in \Der_F(M(R))}$ to an element of 
$\Der_F(Q(M(R)))$ ($\Der_F(Q^s_m(M(R)))$), 
${\ad_D^{\tau}: \phi\longmapsto \tau^{-1}((\tau \phi)\ad_D)}$, 
${\phi\in \CM(R)}$, ${\ad_D^{\tau}\in \Der_F(\CM(R))}$ (correctness follows 
from ${t_y \ad_D=t_{y \overline{D}}}$, ${t=l, r}$, ${y\in P(R)}$). The 
linearity of $\overline{D}$ over $\CM(R)$ is equivalent to the linearity of 
$\ad_D$ over ${}_{M(R)} C$.

If $R$ is an algebra with ${\Ann R=\{0\}}$, then ${x\longmapsto (l_x, r_x)}$, 
${x\in R}$, is a bijection, for any ${D\in \Der(R)}$, ${x\in R}$, $x D$ is 
completely determined by ${(l_{x D}=l_x \ad_D, r_{x D}=r_x \ad_D)}$. 

\begin{co}
If $R$ is a torsion-free $m.s.p.$--algebra, ${D\in \Der(R)}$, $\ad_D$ 
is strictly algebraic over $Q^s_m(M(R))$, then ${\ad_D=\ad_Q}$ on 
$Q^s_m(M(R))$ for some ${Q\in Q^s_m(M(R))}$. To a strongly algebraic 
$D$ over $\CM(R)$ there corresponds an algebraic $Q$ over ${}_{M(R)} C$. 
\end{co}

Let $A$ be a semiprime associative ring, ${B={}_A C A}$. Any 
${\phi\in \Hom(I, A)_A}$, ${I\lhd A}$, can be extended by ${}_A C$--linearity to  
${\overline{\phi}\in \Hom({}_A C I, B)_B}$, since for 
${\sum\limits_i c_i x_i=0}$,  ${c_i\in {}_A C}$, ${x_i\in I}$, there exist 
${J\in \mathcal{E}(A)}$, ${\sum\limits_i c_i J\subseteq A}$, and  
${\Bigl(\sum\limits_i c_i (\phi x_i)\Bigr) J=
\phi\Bigl(\Bigl(\sum\limits_i c_i x_i\Bigr) J\Bigr)=\{0\}}$, 
${\sum\limits_i c_i \phi x_i=0}$. Here  
${\Bigl(\sum\limits_i c_i I\Bigr) J+J \Bigl(\sum\limits_i c_i I\Bigr)
\subseteq I}$ and ${\Bigl(\sum\limits_i c_i y_i\Bigr) J, 
J \Bigl(\sum\limits_i c_i y_i\Bigr)\ne \{0\}}$ for all 
${\sum\limits_i c_i y_i\ne 0}$, ${y_i\in I}$. So,  
$Q^r_m(A)$ can be identified with a ${}_A C$--subalgebra of $Q^r_m(B)$ using 
${\iota: [(\phi, I)]\longmapsto [(\overline{\phi}, {}_A C I)]}$, 
${I\in \mathcal{E}(A)}$, ${\phi\in \Hom(I, A)_A}$,
\[
\{{}_A C I\mid I\in \mathcal{E}(A)\}\subseteq \mathcal{E}(B)\,,\quad 
\{A\cap I=A\cap {}_A C (A\cap I)\mid I\in \mathcal{E}(B)\}\subseteq 
\mathcal{E}(A)\,. 
\]
Since for each ${q\in Q^s_m(A)}$ there is ${I\in \mathcal{E}(A)}$, 
${q I+I q\subseteq A}$, ${q {}_A C I+{}_A C I q\subseteq B}$, 
${q\in Q^s_m(B)}$, ${Q^s_m(A)\subseteq Q^s_m(B)}$ 
\cite[Introduction]{Gol10}. 

If ${I\lhd A}$, ${J\in \mathcal{E}(I)}$, then 
${J^3\subseteq (J)_A^3\subseteq J}$, ${J^3, (J)_A^3\in \mathcal{E}(I)}$, 
for any ${\phi\in \Hom(J, I)_I}$, ${z\in J}$, ${y\in I}$, ${x\in A}$ holds  
${\phi(z y x)=(\phi z) y x=\phi(z y) x}$, 
${\phi|_{J I}\in \Hom(J I, I)_A}$, ${\phi|_{(J)_A^3}\in \Hom((J)_A^3, A)_A}$, 
${[(\phi, J)]=[(\phi, (J)_A^3)]}$ in $Q^r_m(I)$, 
${\iota_{I, A}: [(\phi, J)]\longmapsto [(\phi', (J)_A^3\oplus K)]}$ is 
an embedding of $Q^r_m(I)$ into $Q^r_m(A)$ for  
${K\lhd A, (J)_A^3\oplus K\in \mathcal{E}(A)}$ (equivalent to 
${I\oplus K\in \mathcal{E}(A)}$ ($\{0\}=I\cap K\cap (J)_A^3=I\cap K$)), 
${\phi'|_{(J)_A^3}=\phi|_{(J)_A^3}}$, ${\phi'(K)=\{0\}}$.
If ${y, z\in I}$,
${[(l_y, I)] [(\phi, J)]=[(l_y \phi, J I)]=[(l_z, I)]}$ and so 
${(l_y \phi-l_z)(J I)=\{0\}}$, then 
${[(l_{x+y}, A)] [(\phi', (J)_A^3\oplus K)]=
[(l_y \phi', (J)_A^3 A\oplus K A)]=[(l_z, A)]}$ for all ${x\in A}$, 
${x I=\{0\}}$
(${(J)_A^3 A\subseteq (J)_A^3\subseteq I J I}$, ${z K=\{0\}}$). 
For ${q=[(\phi, J)]\in Q^s_m(I)}$, ${H\in \mathcal{E}(I)}$, ${H q+q H\subseteq I, q'=\iota_{I, A} q}$ 
it follows that ${((H)_A^3\oplus K) q'+q' ((H)_A^3\oplus K)
\subseteq A}$, ${q'\in Q^s_m(A)}$ and  
${\iota_{I, A}: Q^s_m(I)\hookrightarrow Q^s_m(A)}$. 
                                      
If ${0, 1\ne e=e^2\in {}_A C}$, then for all ${I\in \mathcal{E}(B)}$, 
${\phi\in \Hom(I, B)_B}$, 
\[
e [(\phi, I)]\ =\ [(\phi, I)] e\ =\ [(\phi l_e, I)]\ =\ 
[(\phi l_e, e I\oplus (1-e) B)]\ =\ \iota_{e I, B} [(\phi, e I)]\,,
\]
${[(\phi, e I)]\in Q^r_m(e B), e I\in \mathcal{E}(e B)}$. For any 
${\psi\in \Hom(J, e B)_{e B}, J=e J\lhd e B, y\in J, x\in B}$ holds 
${\psi(y x)=\psi(y e x)=(\psi y) e x=(\psi y) x, J\lhd B, 
\Hom(J, e B)_{e B}=\Hom(J, e B)_B}$. For ${J\in \mathcal{E}(e B)}$ 
we have ${J\oplus (1-e) B\in \mathcal{E}(B)}$ and
${\iota_{e B, B} [(\psi, J)]=[(\psi l_e, J\oplus (1-e) B)]}$. For 
${x, y\in B}$, $[(\phi, I)]\in Q^r_m(B)$ from  
${[(l_x, B)] [(\phi, I)]=[(l_y, B)]}$ follows 
${[(l_{e x}, e B)] [(\phi, e I)]=[(l_{e y}, e B)]}$, the latter is equivalent 
to ${[(l_{e x+(1-e) z}, B)] [(\phi l_e, e I\oplus (1-e) B)]=
[(l_{e y}, B)]}$ for all ${z\in B}$. Therefore in this case 
${e Q^*_m(B)=\iota_{e B, B}(Q^*_m(e B))}$, ${*=r, s}$. 

Since for any ${0\ne q\in Q^r_m(A)}$ there is ${I\in \mathcal{E}(A)}$, 
${\{0\}\ne q I\subseteq A}$, ${A\cap e A\in \mathcal{E}(e A)}$, 
for any ${I\lhd A}$ either ${e I\ne \{0\}, e I\cap A\ne \{0\}}$ or 
${I=(1-e) I.}$ If ${I\in \mathcal{E}(A),}$ then ${e I\ne \{0\}}$ 
($I (A\cap e A)\ne \{0\}$), ${e I, I\cap e I\in \mathcal{E}(e A)}$, 
${I\cap e I\in \mathcal{E}(A\cap e A)}$, since  
\begin{gather*}
\{0\}\ne J\cap A\ =\ J\cap A\cap e A\lhd A, e A\,,\quad  
\{0\}\ne J\cap A\cap I\ =\ J\cap I\cap e I\quad (\{0\}\ne J\lhd e A)\,,
\\
\{0\}\ne (H)_A^3\cap I\cap e I\subseteq H\cap I\cap e I\quad 
(\{0\}\ne H\lhd A\cap e A)
\end{gather*}
(${A J+J A\subseteq J}$, ${(H)_A=(H)_{e A}}$). For  
${H\in \mathcal{E}(A\cap e A)}$, 
${(H)_A\oplus (A\cap (1-e) A)\in \mathcal{E}(A)}$, due to 
${(H)_A\in \mathcal{E}(e A)}$, 
${A\cap (1-e) A\in \mathcal{E}((1-e) A)}$. For ${\{0\}\ne J\lhd B}$, 
${\{0\}\ne H\lhd {}_A C (A\cap e A)}$ similar reasoning gives 
${{}_A C e I, {}_A C (I\cap e I)\in \mathcal{E}(e B)}$, 
${{}_A C (I\cap e I)\in \mathcal{E}({}_A C (A\cap e A))}$.

For any ${I\in \mathcal{E}(A)}$, ${\phi\in \Hom(I, A)_A}$, 
${H\in \mathcal{E}(A\cap e A)}$, ${\psi\in \Hom(H, A\cap e A)_{A\cap e A}}$, 
in view of ${\phi(I\cap e I)=\overline{\phi}(I\cap e I)=e \phi(I\cap e I)
\subseteq A\cap e A}$, 
\begin{multline*}
e [(\phi, I)]\ =\ [(\phi, I)] e\ =\ 
[(\phi, I)] [(l_e, (A\cap e A)\oplus (A\cap (1-e) A))]\ =
\\ 
[(\phi l_e, (A\cap e A) I\oplus (A\cap (1-e) A) I)]\ =\ 
[(\phi l_e, (I\cap e I)\oplus (A\cap (1-e) A))]\ =\ 
\\
\iota_{A\cap e A, A} [(\phi, I\cap e I)]\,,
\end{multline*}
${[(\phi, I\cap e I)]\in Q^r_m(A\cap e A)}$, 
${\iota_{A\cap e A, A}[(\psi, H)]=
[(\psi l_e, (H)_A^3\oplus (A\cap (1-e) A))]}$. If ${x, y\in A}$ and 
${[(l_x, A)] [(\phi, I)]=[(l_y, A)]}$, then 
${[(l_x, A\cap e A)] [(\phi l_e, I\cap e I)]=[(l_y, A\cap e A)]}$, where for 
${x=e x}$ holds ${[(l_x, A)]=e [(l_x, A)]}$, 
${[(l_y, A)]=e [(l_y, A)]=[(l_{e y}, (A\cap e A)\oplus (A\cap (1-e) A))]}$, 
${y=e y}$. If ${x, y\in A\cap e A}$, 
${[(l_x, A\cap e A)] [(\psi, H)]=[(l_y, A\cap e A)]}$, then 
\[
[(l_{x+z}, A)] [(\psi l_e, (H)_A^3\oplus (A\cap (1-e) A))]\ =\ 
[(l_y, A)]\quad (z\in A\cap e A)\,.
\]
Thus, ${e Q^*_m(A)=\iota_{A\cap e A, A}(Q^*_m(A\cap e A))}$, ${*=r, s}$
         
If $A$ is prime and contains a non-zero minimal (smallest non-zero) ideal 
$I=\Soc(A_{M(A)'})=\bigcap\limits_{J\in \mathcal{E}(A)} J\ne \{0\}$, then 
\begin{enumerate}

\item $I$ is a simple ring (by Andrunakievich's lemma,  
${(J)_A^3\subseteq J}$ for any ${J\lhd I\lhd A}$); 

\item ${\Hom(I, A)_I=\Hom(I, A)_A=\Hom(I, I)_I=\Hom(I, I)_A}$, 
${\End(I)_I=\End(I)_A}$ (${I=I^2}$\\ and 
${\phi(a a' b)=(\phi a) a' b=\phi(a a') b}$ 
for all ${\phi\in \Hom(I, A)_I}$, ${a, a'\in I}$, ${b\in A^1=F\oplus A}$); 

\item ${Q^r_m(A)\cong Q^r_m(I)\cong \End(I)_I}$ 
(${[(\phi, I)]\longmapsto \phi}$, ${\phi\in \Hom(I, A)_A}$); 

\item ${Q^s_m(I)\cong Q^s_m(A)}$, since 
\[
Q^s_m(I)\cong \{\phi\in \End(I)_I\mid L(I) \phi\subseteq L(I)\}\ =\ 
\{\phi\in \End(I)_I\mid L(I) \phi\subseteq L(A)|_I\}\cong Q^s_m(A)\,,
\]
${I=I^2}$ and for any ${a\in I}$, ${\phi\in Q^s_m(A)}$ there is ${b\in A}$, 
${l_a \phi=l_b|_I}$, ${l_{a' a} \phi=l_{a' b}}$ for all ${a'\in I}$, where 
${L(A)=\langle l_x\mid x\in A\rangle=\{l_x\mid x\in A\}}$ is the ring of 
left multiplications of $A$.

\end{enumerate}
Since $I$ is the smallest non-zero ideal of $B$ (${I=I\cap c I}$ for all 
${0\ne c\in {}_A C}$, ${I={}_A C I}$), here $A$ can be replaced by $B$.

By Martindale's theorem, \cite[Theorem~3]{Mart} (together with \cite{Chua}), 
any non-zero prime associative $GPI$--ring $A$ has a primitive central 
closure ${B={}_A C A}$ whose socle ${I=\Soc(B_B)}$ is the smallest non-zero 
ideal and the smallest dense right ideal of $B$, 
\[
Q(A)\ =\ Q(B)\cong \End(I)_B\cong \End(B e)_{e B e}\,,\quad 
e B e\cong \End_B(B e)\,,
\] 
where ${0\ne e=e^2}$ is a generator of the minimal right (left) ideal $B e$ 
($e B$), $e B e$ is a division ring that is finite-dimensional over its 
center ${Z(e B e)=e {}_A C\cong {}_A C}$ 
\cite[Propositions~4,~7,~p.~109,\linebreak 163]{Lamb}, 
\cite[Theorems~2.1,~3.1]{Jain} ($B$ and the simple ring $I$ 
have common minimal left and right ideals, ${I=\Soc(I_I)}$). In the proof of \cite[Theorem~2]{Har2} (taking into account the presence of characteristics 
in the prime rings) it was established that in this case 
\[
\{D\in {}_A C \Der(A)\mid ({}_A C)D=\{0\}\}\ =\ 
{}_A C \Der(A)\cap \Inder(Q^s_m(A))\,,
\]
where $\Der(A)$ is identified with the subring of $\Der(Q^s_m(A))$ that is 
formed by uniquely defined extensions of elements of $\Der(A)$ to 
$Q^s_m(A)$. Using \cite{B4}, the latter can be transferred to semi\-prime 
$SGPI$--rings.

\begin{prop}
If $A$ is a semiprime associative $SGPI$--ring, then 
\[
\{D\in {}_A C \Der(A)\mid ({}_A C)D=\{0\}\}\ =\ {}_A C \Der(A)\cap 
\Inder(Q^s_m(A))\,.
\]
\end{prop}

\begin{proof}
In view of the observations after Remark 2.6 (the formulation of \cite[Theorem~1]{Har2}) and Corollary 2.11, the ring ${B={}_A C A}$ contains ideals 
${B_p=e_p B\lhd_{\Der} B}$ of characteristics ${p\in \mathcal{P}_0}$ that 
are semiprime as rings, ${0\ne e_p=e_p^2\in {}_A C}$, for any ${x\in B}$ 
${e_p x=0}$ for all ${p\in \mathcal{P}_0}$ is equivalent to ${x=0}$ and up 
to isomorphism 
\[
Q^*_m(A)\subseteq Q^*_m(B)\,,\quad 
e_p Q^*_m(A)\ =\ Q^*_m(A_p)\subseteq e_p Q^*_m(B)\ =\ Q^*_m(B_p)
\quad (*=r, s,\ p\in \mathcal{P}_0)\,,
\]
where ${A_p=A\cap e_p A=A\cap B_p}$ and for ${p\in \mathcal{P}}$ 
${e_p Q^*_m(A)=Q^*(A)_p\subseteq e_p Q^*_m(B)=Q^*_m(B)_p}$ 
(if ${0\ne q\in Q^r_m(B)_p}$, then there is  
${I\in \mathcal{E}(B)}$, ${\{0\}\ne q I\subseteq B_p}$, 
${(q-e_p q) I=\{0\}}$, ${q=e_p q}$; ${p e_p=0}$).

Let ${D\in {}_A C \Der(A)}$, ${({}_A C)D=\{0\}}$, 
${D_p=D|_{B_p}=e_p D|_{B_p}\in \Der(B_p)}$ for all ${p\in \mathcal{P}_0}$. 
There is a unique ${v_p=v_p^2=e_p v_p\in {}_A C}$, 
${v_p \Ann_{B_p} ((B_p)D_p)_{B_p}=\{0\}}$, ${v_p x=x}$ for all 
${x\in ((B_p)D_p)_{B_p}}$ (see properties of central closures for 
${B\cong P(A)}$, \cite[Remark~2]{Har4}, \cite[Proposition~2.10]{ADN}, 
\cite[\linebreak Theorem~2.3.9(i)]{BMM}; ${1-e_p\in \Ann_{B_p}((B_p)D_p)_{B_p}}$).
If ${D_p\ne 0}$, then, according to \cite[Theorems\linebreak 2,~3]{B4}, 
${\{0\}\ne B'_p=v_p B=v_p B_p\lhd B}$ contains ${u_p=u_p^2=v_p u_p\ne 0}$ 
such that ${{}_p A=u_p B u_p\cap A}$ is a semiprime $PI$--ring, 
\[
Q({}_p A)\ =\ u_p B u_p\ =\ u_p B_p u_p\ =\ u_p B'_p u_p\ =\ u_p Q(A) u_p
\] 
is a centrally separable algebra over ${Z(u_p B_p u_p)=u_p {}_A C}$, finitely 
generated as a $u_p {}_A C$--module. Having chosen for $u_p B u_p$ a central 
polynomial ${f(x_1, \ldots, x_n)}$ multilinear in some $\{x_{i_j}\}$, 
${n\in \{i_j\}}$, for definiteness, and ${y_1, \ldots, y_n\in u_p B u_p}$, 
${f(y_1, \ldots, y_n)\ne 0}$, we can use one of the constructions of the proof 
of \cite[Theorem~2]{Har2} for
\[
l\,:\ x\longmapsto f(y_1, \ldots, y_{n-1}, x)\ =\ \sum_i 
a_i x b_i\in u_p {}_A C\quad (x\in u_p B u_p)\,,
\]
${a_i, b_i\in u_p B u_p}$. As in \cite{Har2}, due to 
${(u_p c)D_p=(u_p D_p) c}$, ${c\in {}_A C}$, for all ${x, y\in B}$ 
\begin{gather*}
l(x) D_p\ =\ \sum_i ((a_i D_p) x b_i+a_i x (b_i D_p))+
\sum_i a_i (x D_p) b_i\ =\ t(x)+l(x D_p)\,,
\\ 
l(x) (l(y) D_p)-l(y) (l(x) D_p)\ =\ 
l(x) (t(y)+l(y D_p))-l(y) (t(x)+l(x D_p))\ =\ 0\,.
\end{gather*}
If ${D_p\not\in \Inder(Q^s_m(A_p))}$, then by \cite[Theorem~2]{Har2}, 
$A_p$ satisfies the generalized identity 
\[
l(x) (t(y)+l(z_1 v_p))-l(y) (t(x)+l(z_2 v_p))\ =\ 0
\]
and, consequently, ${v_p l(x) l(z)=l(x) l(z)=l(x)^2=0}$ for all 
${x, z\in A_p}$ (${y=0}$), but ${l(y_n)^2\ne 0}$ ($u_p {}_A C$ is semiprime), a contradiction. 
Therefore, there exist ${q_p\in Q^s_m(A_p)}$, ${D_p=\ad_{q_p}}$, and ${D=\ad_q}$ 
for a unique ${q={\sum\limits_{p\in \mathcal{P}_0}}^{\perp} e_p q_p\in 
Q^s_m(A)=O(Q^s_m(A))}$, ${e_p q=e_p q_p}$ for all ${p\in \mathcal{P}_0}$,
\begin{multline*}
\{D\in {}_A C \Der(A)\mid ({}_A C)D=\{0\}\}\ =\ 
{}_A C \Der(A)\cap \Inder(Q^s_m(A))\ =
\\ 
({}_A C \Der(A)\cap \Inder(Q(A)))|_{Q^s_m(A)}\,,
\end{multline*}
$\Der(A)$ on the right-hand side of the equality is identified with the subring of 
$\Der(Q(A))$ consisting of uniquely defined extensions of elements 
of $\Der(A)$ to $Q(A)$ \cite[Lemma~4]{Har2}, \cite[Lemma~1]{B6}, \cite[Proposition~8.7]{BMich}, \cite[Lemma~2]{Lee} (\cite[Exercise~10,~p.~167]{Lamb}).
\end{proof}

Taking into account Remark 2.4, we can deduce from here 

\begin{co}
If $R$ is a $m.s.p.$--algebra, $M(R)$ is a $SGPI$--algebra and ${D\in \Der(R)}$, 
then ${[\overline{D}, \CM(R)]=\{0\}}$ is equivalent to 
${\ad_D\in {}_{M(R)} C \Der(M(R))\cap \Inder(Q^s_m(M(R)))}$.
\end{co}

In particular, if ${R=P(R)}$ is a $m.s.p.$--algebra, $M(R)$ is a $SGPI$--algebra, 
${D\in \Der(R)}$ under the conditions of Lemmas 2.7, 2.8, then  
${\ad_D\in {}_{M(R)} C \Der(M(R))\cap \Inder(Q^s_m(M(R)))}$, for a 
$m.p.$--algebra $R$, (2.8.3) is omitted as obvious. 

Anticipating Theorem 2.14, we note that 
\[
x D^n t_y\ =\ \sum_{i=0}^n \binom{n}i (-1)^{n-i} x t_{y D^{n-i}} D^i
\]
for any algebra $R$, ${D\in \Der(R)}$, ${x, y\in R}$, ${t=l, r}$, ${n\geq 1}$, 
since 
\begin{multline*}
x t_y D^n\ =\ \sum_{i=0}^n \binom{n}i (x D^i) t_{y D^{n-i}}\ =\
x D^n t_y+\sum_{i=0}^{n-1} \binom{n}i 
\biggl(\sum_{j=0}^i \binom{i}j (-1)^{i-j} x t_{y D^{n-j}} D^j\biggr)\ =
\\
x D^n t_y+\sum_{j=0}^{n-1} \biggl(\sum_{i=j}^{n-1} \binom{n}i \binom{i}j 
(-1)^{i-j} \biggr) x t_{y D^{n-j}} D^j\ =
\end{multline*}
\begin{multline*}
x D^n t_y+\sum_{j=0}^{n-1} \binom{n}j 
\biggl(\sum_{i=j}^{n-1} \binom{n-j}{i-j} (-1)^{i-j} \biggr) 
x t_{y D^{n-j}} D^j\ =
\\
x D^n t_y+\sum_{j=0}^{n-1} \binom{n}j 
\biggl(\sum_{l=0}^{n-1-j} \binom{n-j}l (-1)^l\biggr) x t_{y D^{n-j}} D^j\ =\ 
x D^n t_y-\sum_{j=0}^{n-1} \binom{n}j (-1)^{n-j} x t_{y D^{n-j}} D^j\,.
\end{multline*}
By analogy with \cite[Proposition~1.3.1,~p.~14]{Row}, for any finitely generated 
$F$--modules $M$, $M'$ we can realize ${\Hom_F(M, M')}$ as a quotient module of 
a submodule of a finitely generated $F$--module. It is sufficient to select in 
$M$ and $M'$ finite systems of generators $\{e_i\}_{i=1}^n$ and 
$\{e'_j\}_{j=1}^m$ and establish ${\Hom_F(M, M')\cong L/N}$,  
\begin{gather*}
L\ =\ \biggl\{(f_{ij})\in F^{n\times m}\biggl| 
\sum_j \biggl(\sum_i h_i f_{ij}\biggr) e'_j=0\ 
\forall\, \{a_i\}\subseteq F,\ \sum_i h_i e_i=0\biggl\}\,,
\\ 
N\ =\ \biggl\{(h_{ij})\in F^{n\times m}\biggl| 
\sum_j h_{ij} e'_j=0,\ i=1, \ldots, n\biggl\}\,, 
\end{gather*}
by associating to each ${\phi\in \Hom_F(M, M')}$ a class ${(f_{ij})+N\in L/N}$ 
for any representations ${e_i \phi=\sum\limits_j f_{ij} e'_j, i=1, \ldots, n, 
(f_{ij})\in L}$ (for ${(f_{ij})\in L,  
\phi_{(f_{ij})}: \sum\limits_i f_i e_i\longmapsto 
\sum\limits_j \Bigl(\sum\limits_i f_i f_{ij}\Bigr) e'_j, f_i\in F}$, 
${\phi_{(f_{ij})}\in \Hom_F(M, M')}$, ${\phi_{(f_{ij})}=\phi_{(f'_{ij})}}$ is 
equivalent to ${(f_{ij})-(f'_{ij})\in N}$).

\begin{teor}
If a left exact associative $F$--algebra $A$ is right Noetherian (with one of 
the conditions: has no ${}_A C$--torsion; weakly Noetherian; finitely 
cogenerated as a $M(A)$--module) and the set $\mathcal{D}_r(A)$ 
(${\mathcal{D}_{ir}(A)=\{I\lhd A\mid I\in \mathcal{D}_r(A)\}}$) of dense 
right ideals (ideals with zero left annihilators) of $A$ is closed under 
countable intersections, then the algebraicity of ${D\in \Der(A)}$ over 
$F$ inherits its extension to $Q(A)$ (${\hat{A}={}_A C+{}_A C A}$).
\end{teor}

\begin{proof}
The uniquely defined extension of ${D\in \Der(A)}$ to $Q(A)$ sets 
the rule: $(q D)(x y)=((q x) D-q (x D)) y$, ${q, y\in Q(A), 
x\in q^{-1} A=\{a\in A\mid q a\in A\}\in \mathcal{D}_r(A)}$ 
\cite[Lemma~2]{Lee}. For any\linebreak ${q\in Q(A)}$, ${k\geq 0}$ we set 
${I_k(q)=\bigcap\limits_{i=0}^k {(q D^i)}^{-1} A\in \mathcal{D}_r(A)}$ 
\cite[(1.7)]{Ut} or \cite[Lemma~3,~p.~161]{Lamb} and\linebreak  
${I(q)=\bigcap\limits_{i=0}^{\infty} {(q D^i)}^{-1} A}$, 
${I(q)\in \mathcal{D}_r(A)}$ 
(${I(q)\in \mathcal{D}_{ir}(A)}$ for ${q\in {}_A C=Z(Q(A))}$, since 
${q D^i\in {}_A C}$, ${(q D^i)^{-1} A\in \mathcal{D}_{ir}(A)}$). 
According to the previous 
\[
q (x D^n)\ =\ \sum_{i=0}^n \binom{n}i (-1)^{n-i} ((q D^{n-i}) x)D^i\in A\quad 
(x\in I_n(q),\ n\geq 0)
\]
and, as a consequence, 
\[
q D^n x\ =\ \sum_{i=0}^n \binom{n}i (-1)^{n-i} (q (x D^{n-i}))D^i\in A\quad 
(x\in I_n(q),\ n\geq 0)\,.
\]
From this we immediately get that for ${D\in \Der_{nil}(A)}$ and any ${x\in I(q)}$ 
there is ${n_{q, x}\geq 1}$, ${0=q D^n x=q D^{n_{q, x}} f(D)}$ for all 
${n\geq n_{q, x}}$, ${f(t)\in F[t]}$. If ${D\in \Der_a(A)}$, then for any 
${x\in I(q)}$ there is a finitely generated Noetherian subring ${H\subseteq F}$ 
with ${1\in F}$ such that the $H$--modules
\[
W\ =\ \sum_{i, j\geq 0} H (q (x D^j))D^i\ =\ 
\sum_{i, j\geq 0} H (q D^i)(x D^j)\,,\quad 
W'\ =\ \sum_{i\geq 0} H (x D^i)
\]
and, consequently, ${\Hom_H(W', W)}$, $\alpha(V)$ for 
${V=\sum\limits_{i\geq 0} H(q D^i)}$, ${\alpha\in 
\Hom_H(V, \Hom_H(W', W))}$, ${\alpha: v\longmapsto l_v|_{W'}}$, ${v\in V}$,
are finitely generated and Noetherian (Hilbert's basis theorem; Noetherian 
property of finitely generated modules over Noetherian rings; the observation 
before Theorem 2.14). Hence there exists $f_{q, x}(t)\in H_1[t]$, 
${(q f_{q, x}(D))(W')=\{0\}}$, and ${q h(D) x=0}$ for all 
${h(t)\in f_{q, x}(t) F[t]}$, ${f_{q, x}(t)=t^{n_{q, x}}}$ for 
${D\in \Der_{nil}(A)}$. If $A$ has no ${}_A C$--torsion (in particular,\linebreak 
it is prime), ${q\in {}_A C}$ and ${0\ne x\in I(q)}$, then this is equivalent 
to ${q f_{q, x}(D)=0}$. If $A$ is right\linebreak Noetherian (weakly Noetherian 
and ${q\in {}_A C}$), then we can choose among the right ideals\linebreak (ideals) 
${J_u=\bigcap\limits_{v(t)\in u(t) F[t]} (I(q)\cap \Ker l_{q v(D)}), 
u(t)\in F_1[t]}$, the maximal ideal $J_w$ and, in view of 
${J_u\cup J_v\subseteq J_{u v}}$, ${u(t), v(t)\in F_1[t]}$, 
${J_w=I(q)\in \mathcal{D}_r(A)}$, ${q w(D)=0}$ 
(for ${D\in \Der_{nil}(A)}$ you should replace $\{J_u\}_{u\in F_1[t]}$ with 
$\{J_i\}_{i\geq 1}$, ${J_i=\bigcap\limits_{j\geq i} (I(q)\cap l_{q D^j})}$, 
${J_i\subseteq J_{i+1}}$, and find the maximal 
${J_n=\bigcup\limits_{i\geq 1} J_i=I(q)}$, ${q D^n=0}$). 

A finite cogeneration of a $M(A)$--module $A$ is the presence in any system 
of ideals of $A$\linebreak with zero intersection of a finite subsystem 
with zero intersection. For $A$ with such a condition, ${q\in {}_A C}$ and 
\[
J(q)\ =\ \bigcap\limits_{x\in I(q)} 
((q f(D)\mid f(t)\in f_{q, x}(t) F[t])_{Q(A)}\cap A)
\] 
from ${J(q) I(q)=\{0\}}$, ${I(q)\in \mathcal{D}_{ir}(A)}$ it follows that for 
some ${m\geq 1}$, ${x_i\in I(q)}$  
\[
\{0\}\ =\ J(q)\ =\ \bigcap\limits_{i=1}^m 
(q h(D)\mid h(t)\in f_{q, x_i}(t) F[t])_{Q(A)}\,,
\]
${q g(D)=0}$, ${g(t)=\prod\limits_{i=1}^m f_{q, x_i}(t)\in F_1[t]}$. It remains 
to note that ${D\in \Der_a(\hat{A})}$ (${D\in \Der_{nil}(\hat{A})}$) is 
equivalent to ${D\in \Der_a(A), \Der_a({}_A C)}$
(${D\in \Der_{nil}(A), \Der_{nil}({}_A C)}$).
\end{proof}

If an algebra ${R\ne \{0\}}$ is subdirectly irreducible (has a smallest non-zero 
ideal) or (and) is weakly Artinian, then $R$ is a finitely cogenerated 
$M(R)$--module. From the finite cogene\-ration of $R$ as a $M(R)$--module it 
immediately follows that 
${\Soc(R_{M(R)'})=\bigcap\limits_{I\in \mathcal{E}(R)} I\in \mathcal{E}(R)}$. 
If $R$ is semiprime, then ${\Soc(R_{M(R)'})\in \mathcal{E}(R)}$ is equivalent to 
the inclusion in any non-zero ideal of $R$ of its minimal ideal 
\cite[Proposition~1,~Exercise~7,~p.~102,~107]{Lamb}. A semiprime algebra 
$R$ without $\CM(R)$--torsion is prime (Remark 2.6). 
If $R$ is semiprime and weakly Noetherian (Artinian) or finitely cogenerated as 
a $M(R)$--module, then the direct sum  
${\bigoplus\limits_{p\in \mathcal{P}_0} P(R)_p\subseteq P(R)}$ is finite. 
Applying to Theorem 2.14 (2.7.2), (2.8.2) and Proposition 2.12, we obtain, 
taking into account ${1\in \hat{A}}$, ${{}_A C={}_{\hat{A}} C=Z(\hat{A})}$,
${\Der_a(R)=\Der_{1, sa}(R)}$, ${\Der_{nil}(R)=\Der_{1, n}(R)}$ for any 
$F$--algebra $R$ with finite over $F$ one-generated subalgebras, 
${\mathcal{D}_{ir}(A)=\mathcal{E}(A)}$ for any semiprime associative algebra $A$,

\begin{co}
If an associative $SGPI$--algebra $A$ over a ring $F$ is semiprime, torsion-free, 
with one of the conditions: prime; weakly Noetherian; finitely cogenerated as 
a $M(A)$--module, and $\mathcal{E}(A)$ is closed under countable intersections, 
${}_A C$ is integral over $F$, ${D\in \Der_{nil}(A)}$ or (and) ${D\in \Der_a(A)}$, 
$F$ is an algebraically closed field of ${\Ch F=0}$, then 
$D\in {}_A C \Der(A)\cap \Inder(Q^s_m(A))$. 
\end{co}

The algebraicity of ${}_A C$ over an algebraically closed field 
${F=\overline{F}}$ is the embeddability of ${}_A C$ in the direct 
product $\prod\limits_{a\in I} F_a$, ${F_a=F}$, with the condition 
${|\{\pi_a c\mid a\in I\}|< \infty}$ for all ${c\in {}_A C}$, 
${\pi_a: \prod\limits_{a\in I} F_a\longrightarrow F_a}$.
For any $F$--algebra $R$, ${B\subseteq R}$, ${D\in \Der(R)}$ from  
${D\in \Der_a(B)}$ it follows that ${D\in \Der_a(\langle B\rangle)}$, 
since for each ${x\in \langle B\rangle}$ there exist a finitely 
generated Noetherian subring ${H\subseteq F}$ with ${1\in F}$ and a 
finitely generated $H$--submodule ${M\subseteq \langle B\rangle}$, 
${(M)D\subseteq M}$ such that the $H$--module  
${\sum\limits_{i\geq 0} H x D^i\subseteq M}$ is finitely generated 
(Hilbert's basis theorem; from ${D\in \Der_{nil}(B)}$ it also 
follows that ${D\in \Der_{nil}(\langle B\rangle)}$).   

\begin{co}
If $R$ is a $m.s.p.$--algebra over a ring $F$, $M(R)$ with one of the 
conditions: prime ($R$ is a $m.p.$--algebra); weakly Noetherian; finitely 
cogenerated as a $M(M(R))$--module, and $\mathcal{E}(M(R))$ is closed under 
countable intersections, then ${\overline{D}\in \Der_F(P(R))}$ inherits the 
condition ${D\in \Der_a(R)}$ (${D\in \Der_{nil}(R)}$). If $R$ is 
torsion-free, $\CM(R)$ is integral over $F$, $M(R)$ is a $SGPI$--algebra, 
${D\in \Der_{nil}(R)}$ (${D\in \Der_a(R)}$, $F$ is an algebraically closed 
field, ${\Ch F=0}$), then ${\ad_D\in {}_{M(R)} C \Der(M(R))\cap \Inder(Q^s_m(M(R)))}$.
\end{co}

\begin{proof}
It is sufficient to notice that ${\ad_D\in \Der_a(M(R))}$ 
(${\ad_D\in \Der_{nil}(M(R))}$) for any $D\in \Der_a(R)$ 
${(D\in \Der_{nil}(R))}$, ${\CM(R)\cong \CM(M(R))\cong {}_{M(R)} C}$ 
(see above and before Corollary 2.11), and apply Theorem 2.14. 

The second statement follows from (2.7.1), (2.8.1) ((2.7.2), (2.8.2)), Proposition 2.12, 
Corollary 2.13 and the observations before Corollary 2.15.
\end{proof}
                  
\begin{co}
If $R$ is a $m.p.$--algebra without torsion over a ring $F$, $M(R)$ 
is a $GPI$--algebra with a non-zero minimal ideal, the 
field $\CM(R)$ is integral over $F$, ${D\in \Der_{nil}(R)}$, then 
${\ad_D\in {}_{M(R)} C \Der(M(R))\cap \Inder(Q^s_m(M(R)))}$. 
\end{co}

\begin{proof}
The minimal non-zero ${I\lhd M(R)}$ is the smallest non-zero ideal 
of $M(R)$, ${}_{M(R)} C M(R)$ and is equal to the socle of 
${}_{M(R)} C M(R)$ \cite[Theorem~3]{Mart}, \cite[Proposition~4,~p.~109]{Lamb} and the observations after Corollary 2.11, 
$M(R)$ is subdirectly irreducible and has no ${}_{M(R)} C$--torsion 
(${{}_{M(R)} C\cong \CM(M(R))\cong \CM(R)}$ is a field). It remains to 
apply Theorem 2.14 (Corollary 2.16), (2.7.3) and Corollary 2.13.
\end{proof}

\begin{co}
If $R$ is a $m.p.$--algebra over a ring $F$, $M(R)$ is a $PI$--algebra, 
${D\in \Der(R)}$, ${[\overline{D}, \CM(R)]=\{0\}}$, then  
${\ad_D=\ad_Q}$ on ${Q(M(R))=\hat{M(R)}}$ for some ${Q\in \hat{M(R)}}$, $D$ 
is strongly algebraic over $\CM(R)$ of degree at most 
$\dim_{\CM(R)} P(R)$ and $Q$ is algebraic over ${}_{M(R)} C$ of degree at most 
${\dim_{{}_{M(R)} C} \hat{M(R)}=(\dim_{\CM(R)} P(R))^2}$. 
\end{co}

\begin{proof}
For the $m.p.$--algebra $R$, the $\CM(R)$--algebra $P(R)$ is simple and finite 
dimensional if and only if $M(R)$ is a $PI$--algebra, in which case 
\[
\hat{M(R)}\cong P(M(R))\cong M(P(R))\ =\ 
\End_{\CM(R)}(P(R))\cong M_{\dim_{\CM(R)} P(R)}(\CM(R))
\]
\cite[Theorem~1.7]{Cab4}, \cite[Proposition~3,~p.~92]{Lamb} (the density 
theorem) taking into account the exactness and irreducibility of the 
$M(P(R))$--module $P(R)$. Application of Corollary 2.13\linebreak and Remark 2.10 
(remarks after it) completes the proof.
\end{proof}

\begin{co}
If $R$ is a $m.p.$--algebra with 1 over a ring $F$, $M(R)$ is a $PI$--algebra 
and ${D\in \Der(R)}$ with one of the conditions: ${D\in \Der_{1, sa}(Z(R))}$, 
$F$ is an algebraically closed field, ${\Ch F=0}$; $R$ is torsion-free, 
${D\in \Der_{1, n}(Z(R))}$; $R$ has no $c_n$--torsion, ${(Z(R))D^n=\{0\}}$ for 
some ${n\geq 1}$, then ${\ad_D=\ad_Q}$ on $\hat{M(R)}$ for some 
${Q\in \hat{M(R)}}$, $D$ is strongly algebraic over $\CM(R)$ of degree at most 
$\dim_{\CM(R)} P(R)$ and $Q$ is algebraic over ${}_{M(R)} C$ of degree at most 
$(\dim_{\CM(R)} P(R))^2$.
\end{co}

\begin{proof}
It is sufficient to note that ${}_{M(R)} C$ is the field of quotients of 
\[
Z(M(R))\ =\ \End(R)_{M(R)'}\ =\ \{l_z=r_z\mid z\in Z(R)\}\cong Z(R)\,,
\]
${\{0\}=[\End(R)_{M(R)'}, D]=({}_{M(R)} C)\ad_D}$ 
(the proof of Lemma 2.8, (2.7.2), (2.8.2);  
$0=1 \ad_D=(c c^{-1})\ad_D=c (c^{-1}\ad_D)=c^{-1}\ad_D$, 
${0\ne c\in Z(M(R))}$), and apply Corollaries 2.13, 2.18.
\end{proof}

Since ${\tau_I: \phi\longmapsto \phi|_I}$, ${\phi\in \CM(R)}$, is an isomorphism 
of $\CM(R)$ and $\End(I)_{M(R)'}$ for any semiprime algebra $R$ and fully 
invariant essential $M(R)'$--submodule ${I\subseteq P(R)}$, 
(2.7.1), (2.7.3), (2.8.1), (2.8.3) can be formulated for ${I=\CM(R) I\in \mathcal{E}(R)}$, 
${(I)D\subseteq I}$. For such $I$ and $D$ with 
${\{z\in P(R)\mid z I=I z=\{0\}\}=\{0\}}$, one can define  
${\tilde{D}\in \Der_F(P(R))}$ by setting 
\[
(\phi x)\tilde{D}\ =\ \phi(x D)+(\phi \ad_{D|_I}^{\tau_I}) x
\quad (\phi\in \CM(R),\ x\in R)\,,
\]  
for ${\ad_{D|_I}^{\tau_I}: \phi\longmapsto \tau_I^{-1}((\tau_I \phi)\ad_{D|_I})}$, 
${\ad_{D|_I}^{\tau_I}\in \Der_F(\CM(R))}$, ${\ad_{D|_I}\in \Der(\End_F(I))}$ 
(correctness follows from ${t_y \ad_{D|_I}|_I=t_{y \tilde{D}}|_I}$, 
${t=l, r}$, ${y\in P(R)}$), ${[D, \End(I)_{M(R)'}]=\{0\}}$ is equivalent to 
${[\tilde{D}, \CM(R)]=\{0\}}$. If $R$ is a $m.s.p.$--algebra, then  
${\overline{D}=\tilde{D}}$, due to 
$(P(R) I+I P(R))D'=(P(R))D' I+I ((P(R)) D')=\{0\}$ for 
${D'=\overline{D}-\tilde{D}}$. 
  
In the context of Lemmas 2.7, 2.8 it is appropriate to highlight a number of 
considerations from \cite{Chung, LeeL} applicable to derivations of any prime 
algebra, not necessarily associative. Without making any changes to the proof 
of \cite[Proposition~2]{Chung} it can be immediately written as 

\begin{lemma}
If $I$ is a right ideal of a $F$--algebra $R$ with ${\Ann_r I=\{0\}}$, 
${D\in \Der(R)}$, ${n\geq 1}$ and ${(I)D^n=\{0\}}$, then ${D^{2 n-1}=0}$.
\end{lemma}

\begin{lemma}
If $R$ is a $F$--algebra without $Z(R)$--torsion, ${D\in \Der(R)}$ and 
${(R)D^n\subseteq Z(R)}$ for some ${n\geq 1}$, then ${R=Z(R)}$ or (and) 
${D^{2 n}=0}$ (${D^{n+1}=0}$ for $R$ without $c_n$--torsion and 
${D^n=0}$ for $R$ without $\binom{2 n}n c_n$--torsion).
\end{lemma}

\begin{proof}
We can assume that ${Z(R)\ne \{0\}}$. Since ${z x\ne 0}$ for all 
${0\ne z\in Z(R), 0\ne x\in R}$, $Z(R)$ has no zero divisors, 
$R$ has characteristic ${p\geq 0}$. Due to 
\[                   
(z (x D^{n-1}))D^n\ \equiv\ (z D^n)(x D^{n-1})\ \equiv\ 0\pod{Z(R)}
\quad (z\in Z(R),\ x\in R)\,,
\]
${(Z(R))D^n=\{0\}}$ or (and) ${(R)D^{n-1}\subseteq Z(R)}$ (from 
${z x\in Z(R)}$, ${0\ne z\in Z(R)}$, ${x\in R}$, it follows that 
${z [x, R]=z (x, R, R)=z (R, x, R)=z (R, R, x)=\{0\}}$ and ${x\in Z(R)}$). 
If ${(Z(R))D^n=\{0\}}$, then ${D^{2 n}=0}$, in $R$ without $c_n$--torsion 
${(Z(R)) D=\{0\}}$ (Remark 2.5), ${D^{n+1}=0}$, and without 
${c_n \binom{2 n}n}$--torsion 
${0=x^2 D^{2 n}=\binom{2 n}n (x D^n)^2=x D^n}$ for all ${x\in R}$,
${D^n=0}$. It remains to apply induction on $n$.
\end{proof}

\begin{lemma}
If $R$ is a $F$--algebra without $Z(R)$--torsion, ${D\in \Der(R)}$, 
${f(t)=\prod\limits_i (t-a_i)}$, ${a_i\in F}$, has the smallest degree among 
${g=\prod\limits_j (t-b_j)\in F[t]}$, ${b_j\in F}$, ${(R)g(D)\subseteq Z(R)}$, 
and ${n=|\{a_i+\Ann_F R\}|}$, ${k_i=|\{j\mid a_j-a_i\in \Ann_F R\}|}$, 
${k=\min\limits_i k_i}$, then for $R$ without $c_k n$--tor\-sion ${R=Z(R)}$ or 
(and) ${(Z(R))D=\{0\}}$, ${D f(D)=0}$.
\end{lemma}

\begin{proof}
We can assume that ${Z(R)\ne \{0\}}$ has no zero divisors, $F$ is an 
integrity domain, $R$ is a $F$--module without torsion (if necessary, one can move 
from $F$ to ${F/\Ann_F R}$, $\Ann_F R=\Ann_F Z(R)$), 
${f(t)=\prod\limits_{i=1}^n (t-a_i)^{k_i}=
t^m+f_{m-1} t^{m-1}+\ldots+f_0}$, ${a_i\ne a_j}$ for ${i\ne j}$, ${f_m=1}$, 
$m=k_1+\ldots+k_n$, ${(R)h_i(D)\not\subseteq Z(R)}$, 
${h_i(t)=(t-a_i)^{k_i-1} \prod\limits_{1\leq i\ne j\leq n} (t-a_j)^{k_j}}$, 
${i=1, \ldots, n}$. 
Then for all ${i=1, \ldots, n}$, ${x\in R}$, ${z\in Z(R)}$ from 
${x D h_i(D)\equiv \alpha_i x h_i(D)\pod {Z(R)}}$ it follows that 
\begin{multline*}
(z (x h_i(D)))f(D)\ =\ 
\sum_{l=0}^m f_l \sum_{j=0}^l \binom{l}j (z D^j) (x h_i(D) D^{l-j})\ \equiv
\\
\Bigl(\sum_{l=0}^m f_l z (D+a_i \Id_R)^l\Bigr) (x h_i(D))\ \equiv\ 
(z f(D+a_i \Id_R)) (x h_i(D))\ \equiv\ 0\pod {Z(R)}\,,
\end{multline*}
${(Z(R))f(D+a_i \Id_R)=\{0\}}$. Note that  
${\bigcap\limits_{i=1}^n A_i=\emptyset}$, 
${A_i=\{a_j-a_i\}_{1\leq i\ne j\leq n}}$, since otherwise one 
can find $i_j$, ${j=1, \ldots, n}$, 
${a_{i_j}-a_j=a}$, ${i_j\ne i_{j'}}$ when ${j\ne j'}$, 
${n a=\sum\limits_{j=1}^n (a_{i_j}-a_j)=0}$, ${a=0}$, a contradiction. 
Therefore, up to associativity over the field of quotients of $F$, the 
greatest common divisor of ${\{f(t+a_i)\}}_{i=1}^n$ is equal to $t^k$, 
${(Z(R))D^k=(Z(R))D=\{0\}}$ (Remark 2.5). 
\end{proof}

In conclusion of this section, we present the necessary information to transfer 
what has been said to the case of a $F$--algebra ${(R, \alpha)}$ with an 
automorphism or anti-automorphism $\alpha$ of order 2. Let's start with the 
fact that $\alpha$ induces an 
automorphism ${\alpha_M\in \Aut_F(\End_F(R))}$ of order 2, 
${\psi \alpha_M=\alpha \psi \alpha}$, ${\psi\in \End_F(R)}$,  
${t_x \alpha_M=t_{x \alpha}}$ for the automorphism $\alpha$ and  
${l_x \alpha_M=r_{x \alpha}}$, ${r_x \alpha_M=l_{x \alpha}}$ for the involution 
$\alpha$, ${t=l, r}$, ${x\in R}$. The semiprimeness $R$ is equivalent to its 
$\alpha$--semiprimeness (the absence of ${\{0\}\ne I=(I)\alpha\lhd R}$, 
${I^2=\{0\}}$ (${I=(I)\alpha\lhd R}$ is an 
$\alpha$--ideal $R$, ${I \lhd_{\alpha} R}$); if ${J\lhd R}$, ${J^2=\{0\}}$, 
then ${((J)\alpha)^2=(J\cap (J)\alpha)^2=\{0\}}$, 
${J\cap (J)\alpha, J+(J)\alpha\lhd_{\alpha} R}$, and in the case of 
$\alpha$--semiprime $R$ ${\{0\}=J\cap (J)\alpha=(J+(J)\alpha)^2=J+(J)\alpha}$). 

The ring $Q^s_m(A)$ of a semiprime associative $F$--algebra ${(A, \alpha)}$ can 
be defined by replacing $\mathcal{E}(A)$ with 
\begin{multline*}
\mathcal{E}(A)_{\alpha}\ =\ \{I\lhd_{\alpha} A\mid I\in \mathcal{E}(A)\}\ =
\\
\{I\lhd_{\alpha} A\mid I\cap J\ne \{0\}\ \forall\, \{0\}\ne J\lhd_{\alpha} A\}\ =\ 
\{I\cap (I)\alpha\mid I\in \mathcal{E}(A)\}
\end{multline*}
(for all ${I\in \mathcal{E}(A)}$ and ${\{0\}\ne J\lhd A}$ 
${\{0\}\ne (I\cap J)\alpha=(I)\alpha\cap (J)\alpha, 
I\cap (I)\alpha\cap (J)\alpha}$ and therefore 
${\{0\}\ne I\cap (I)\alpha\cap J}$, ${I\cap (I)\alpha\in \mathcal{E}(A)}$; 
if ${I\lhd_{\alpha} A}$, ${\{0\}\ne I\cap J}$ for any 
${\{0\}\ne J\lhd_{\alpha} A}$, then ${I\in \mathcal{E}(A)}$, since otherwise 
${\{0\}=I\cap K=I\cap (K)\alpha=I\cap (K+(K)\alpha)\ne \{0\}}$ for some 
$\{0\}\ne K\lhd A$, a contradiction), and $\alpha$ can be extended to an automorphism of 
order 2 (involution) of $Q^s_m(A)$, setting 
${q \alpha=[(\phi, I)], \phi x=(q (x \alpha))\alpha}$ 
(${\phi x=((x \alpha) q)\alpha}$), ${q\in Q^s_m(A)}$, 
${I\in \mathcal{E}(A)_{\alpha}}$, ${q I+I q\subseteq A}$, ${x\in I}$, for the 
automorphism (involution) $\alpha$, the extensions of 
$\Der_{\alpha}(A)=\{D\in \Der(A)\mid D=D \alpha_M\}$ to $Q^s_m(A)$ are 
included in $\Der_{\alpha}(Q^s_m(A))$ for the extension of $\alpha$ to 
$Q^s_m(A)$ and $\alpha_M$ corresponding to it.
If ${(A, \alpha)}$ is a semiprime associative $SGPI$--ring, then by 
Proposition 2.12
\[
\{D\in {}_A C_{\alpha} \Der_{\alpha}(A)\mid ({}_A C)D=\{0\}\}\ =\ 
{}_A C \Der_{\alpha}(A)\cap 
\{\ad_q\mid q\in Q^s_m(A),\ q+\delta_{\alpha} q \alpha\in {}_A C\}
\]
with ${\delta_{\alpha}=-1}$ and ${\delta_{\alpha}=1}$ for the automorphism 
and involution $\alpha$, ${{}_A C_{\alpha}=\{c\in {}_A C\mid c=c \alpha\}}$ 
(taking into account ${{}_A C=\alpha({}_A C)}$), since for ${D=\ad_q}$, 
${q, x\in Q^s_m(A)}$ 
\[
x D \alpha \ =\ x \alpha D\ =\ 
[x, q] \alpha\ =\ [x \alpha, q]\ =\ -\delta_{\alpha} [x \alpha, q \alpha]\,.
\]
For $A$ without 2--torsion, 
${\{D\in {}_A C_{\alpha} \Der_{\alpha}(A)\mid ({}_A C_{\alpha})D=\{0\}\}}$ is 
described in the same way, since 

\begin{rem}
If ${(A, \alpha)}$ is a 2--torsion-free semiprime associative $F$--algebra, 
${D\in \Der_{\alpha}(A)}$, ${({}_A C_{\alpha})D=\{0\}}$, 
then ${({}_A C)D=\{0\}}$.
\end{rem}

\begin{proof}
It is sufficient to note that for all ${x\in {}_A C}$, ${i\geq 1}$ 
\begin{gather*}
0\ =\ (x+x \alpha)D\ =\ x D+x D \alpha\,,
\\ 
0\ =\ (x (x \alpha))D\ =\ (x D)(x \alpha)+x (x D \alpha)\ =\ 
(x D) (x \alpha)-x (x D)\ =\ (x D)(x \alpha - x)\,,
\\
0\ =\ 2 (x D^2) (x D)\ =\ (x D^2) (x D)\ =\ (x D^{i+1}) (x D^i)\,,
\\
0\ =\ ((x D^2)^3 (x D))D\ =\ (x D^2)^4+3 (x D^2)^2 (x D^3) (x D)\ =\ 
(x D^2)^4\ =\ x D^2
\end{gather*}
(semiprimeness of ${}_A C$), and apply Remark 2.5. 
\end{proof}

If ${(R, \alpha)}$ is a $m.s.p.$--algebra, then the auto\-morphism (involution) 
$\alpha$ extends to an automorphism (involution) of $P(R)$ by the rule: 
${(\phi x) \overline{\alpha}=(\phi \alpha_M^{\tau})(x \alpha)}$, ${x\in R}$, 
${\phi\in \CM(R)}$, where ${\tau: \CM(R)\longrightarrow {}_{M(R)} C}$
is a $F$--isomorphism of ${\CM(R)\cong \CM(M(R))}$ and ${}_{M(R)} C$, 
$\alpha_M$ is an\linebreak extension of $\alpha_M$ to ${Q^s_m(M(R)), 
\alpha_M^{\tau}: \phi\longmapsto \tau^{-1}((\tau \phi)\alpha_M), 
\phi\in \CM(R), \alpha_M^{\tau}\in \Aut_F(\CM(R))}$ of order 2, 
${(\psi \alpha_M^{\tau}) x=(\psi (x \alpha))\alpha}$, ${x\in R}$,
${\psi\in \End(R)_{M(R)'}}$ (correctness follows from 
$t_y \alpha_M=t_{y \overline{\alpha}}$, ${t=l, r}$, ${y\in P(R)}$). Each 
${D\in \Der_{\alpha}(R)}$ corresponds to 
${\overline{D}\in \Der_{\overline{\alpha}}(P(R))}$, the linearity of 
$\overline{D}$ over $\CM(R)$ 
(${\CM(R)_{\overline{\alpha}}=\{\phi\in \CM(R)\mid 
\tau \phi\in {}_{M(R)} C_{\alpha_M}\}}$) is equivalent to the linearity of 
the extension ${\ad_D\in \Der_{\alpha_M}(M(R))}$ to  
${\ad_D\in \Der_{\alpha_M}(Q^s_m(M(R)))}$ over ${}_{M(R)} C$ 
(${}_{M(R)} C_{\alpha_M}$). Consequently, if ${(R, \alpha)}$ is a 
$m.s.p.$--algebra, $M(R)$ is a $SGPI$--algebra, ${D\in \Der_{\alpha}(R)}$, 
then ${[\overline{D}, \CM(R)]=\{0\}}$ 
(${[\overline{D}, \CM(R)_{\overline{\alpha}}]=\{0\}}$ for $R$ without 
2--torsion) is equivalent to the including of $\ad_D$ in 
\[
{}_{M(R)} C_{\alpha_M} \Der_{\alpha_M}(M(R))\cap 
\{\ad_Q\mid Q\in Q^s_m(M(R)),\ 
Q-Q \alpha_M\in {}_{M(R)} C\}\,.
\]

Lemmas 2.7, 2.8 can be reformulated for ${D\in \Der_{\alpha}(R)}$, 
\[
\End(R)_{M(R)', \alpha_M}\ =\ 
\{\phi\in \End(R)_{M(R)'}\mid \phi=\phi \alpha_M\}\,,\quad 
Z(R)_{\alpha}\ =\ \{z\in Z(R)\mid z=z \alpha\}
\] 
and ${[D, \End(R)_{M(R)', \alpha_M}]=\{0\}}$ under the appropriate condition, 
the conditions of indecomposability into a direct sum of ideals and weak 
Noetherian property are replaced by their analogues for $\alpha$--ideals. 
Theorem 2.14 gets the following version: 
if ${(A, \alpha)}$ is a semiprime associative $F$--algebra with one the 
conditions: has no ${}_A C_{\alpha}$--torsion (and, in particular, 
is $\alpha$--prime); weakly Noetherian for $\alpha$--ideals; finitely 
cogenerated as a ${\langle M(A), \alpha\rangle}$--module, and  
$\mathcal{E}(A)_{\alpha}$ is closed under countable intersections, then the 
algebraicity ${D\in \Der_{\alpha}(A)}$ over $F$\linebreak inherits its extension to 
${\hat{A}_{\alpha}={}_A C_{\alpha}+{}_A C_{\alpha} A}$, the extension of 
${D\in \Der_{\alpha}(A)\cap \Der_{nil}(A)}$ to $\hat{A}_{\alpha}$ is included 
in ${\Der_{\alpha}(\hat{A}_{\alpha})\cap \Der_{nil}(\hat{A}_{\alpha})}$.

In the context of the prime part of Proposition 2.12, it should be noted that 
there are other realizations of derivations of primitive associative rings 
with non-zero socle in terms of inner derivations \cite[Theorem~14.2,~p.~130]{JacSR}.

\section{Radicals of groups of automorphisms of algebras}

The description of the radicals of linear groups in terms of the radicals 
of coefficient rings from \cite{Gol7} can be supplemented by a scheme for 
constructing radicals of algebras based on radicals of groups, similar to 
\cite[Lemma~4.1]{Gol6}.

\begin{lemma}
Let $\mathfrak{M}$ and $\mathfrak{N}$ be classes of $F$--algebras and groups 
closed under taking homomorphic images, ideals and, respectively, normal 
subgroups, $\mathcal{T}$ be a radical in the sense of Kurosh --- Amitsur on 
$\mathfrak{N}$, $\alpha$ be a mapping that associates the algebras 
${R\in \mathfrak{M}}$ with the groups ${\alpha(R)\in \mathfrak{N}}$ and has 
the properties: 
\begin{itemize}

\item[(1)] there is a correspondence ${I\longmapsto \alpha(I)_R\lhd \alpha(R)}$, 
${I\lhd R}$, such that ${\alpha(I)_R=\mathcal{T}(\alpha(I)_R)}$ if  
${\alpha(I)=\mathcal{T}(\alpha(I))}$,  
${\alpha(R)=\prod\limits_{a\in A} \alpha(I_a)_R}$ for 
${R=\sum\limits_{a\in A} I_a}$, ${I_a\lhd R}$, 
${\alpha(I_a)=\mathcal{T}(\alpha(I_a))}$; 
    
\item[(2)] to any homomorphism of algebras ${\psi: R\longrightarrow R'}$ there 
corresponds an epimorphism (a homomorphism for the hereditary $\mathcal{T}$) of 
groups ${\psi_{\alpha}: \alpha(R)\longrightarrow \alpha(\psi(R))}$ such that 
\[
(\Ker \psi_{\alpha}\, \alpha(\Ker \psi)_R)/\alpha(\Ker \psi)_R\ =\ 
\mathcal{T}((\Ker \psi_{\alpha}\, \alpha(\Ker \psi)_R)/\alpha(\Ker \psi)_R)
\]
(and $\psi_{\alpha}$ is an epimorphism if 
${\alpha(Ker \psi)_R=\mathcal{T}(\alpha(\Ker \psi)_R)}$).

\end{itemize}
Then ${\mathfrak{R}_{\mathcal{T}, \alpha}=\{R\in \mathfrak{M}\mid 
\alpha(R)=\mathcal{T}(\alpha(R))\}}$ is a radical subclass of 
$\mathfrak{M}$.
\end{lemma}

\begin{proof} The only difference from the proof of \cite[Lemma~4.1]{Gol6} 
is the need to check the closedness of the class 
$\mathfrak{R}_{\mathcal{T}, \alpha}$ under taking homomorphic images for the 
hereditary radical $\mathcal{T}$, but in this case for any epimorphism 
${\psi: R\longrightarrow R'}$, ${R\in \mathfrak{R}_{\mathcal{T}, \alpha}}$, 
$\alpha(\Ker \psi)_R=\mathcal{T}(\alpha(\Ker \psi)_R)\lhd 
\alpha(R)=\mathcal{T}(\alpha(R))$ and so, $\psi_{\alpha}$ is an epimorphism 
and ${\alpha(R')=\psi_{\alpha}(\alpha(R))=\mathcal{T}(\alpha(R'))}$, 
${R'\in \mathfrak{R}_{\mathcal{T}, \alpha}}$. 

In (2), we can limit ourselves to requirement of surjectivity of 
$\psi_{\alpha}$ for ${R\in \mathfrak{R}_{\mathcal{T}, \alpha}}$ and 
canonical epimorphisms ${\psi: R\longrightarrow R/I}$, ${I\lhd R}$, 
${I, R/I\in \mathfrak{R}_{\mathcal{T}, \alpha}}$, with equality for 
$\psi_{\alpha}$ from (2) only in the latter case, without using 
the heredity of $\mathcal{T}$.
\end{proof}

The version of \cite[Lemma~4.1]{Gol6} given here is applicable to Lemma 4.1 
itself, both in (1) (${\gamma(R)=\sum\limits_{a\in A} \gamma(I_a)_R}$ 
for ${R=\sum\limits_{a\in A} I_a}$, ${I_a\lhd R}$, 
${\gamma(I_a)=\mathcal{T}(\gamma(I_a))}$) and (2).

For example, if $\mathfrak{M}$ and $\mathfrak{N}$ are the classes of all 
associative $F$--algebras and groups, 
\[
\alpha(A)\ =\ U(A^1)\cap (1+A)\ =\ 1+J(A)\,,\quad 
\alpha(I)_A\ =\ \alpha(I)\ =\ 1+I\cap J(A)\quad (I\lhd A\in \mathfrak{M})\,,
\]
$U(A^1)$ is the group of invertible elements of the algebra ${A^1=F\oplus A}$ 
obtained by the standard adjunction of 1 to $A$, $J(A)$ is the set 
of quasiregular elements of $A$, then any homomorphism of algebras 
${\psi: A\longrightarrow A'}$, ${A, A'\in \mathfrak{M}}$, induces a 
homomorphism of groups 
\[
\psi_{\alpha}\ =\ \psi^1|_{\alpha(A)}\,:\ 
\alpha(A)\longrightarrow \alpha(A')\,,\quad 
\Ker \psi_{\alpha}\ =\ 1+J(\Ker \psi)\ =\ \alpha(\Ker \psi)\,,
\]
${\psi^1(f+x)=f+\psi x, f\in F, x\in A, \psi_{\alpha}(\alpha(A))=\alpha(\psi(A))}$ 
for ${\Ker \psi=\Ker \psi^1=J(\Ker \psi)}$, for  
${A=\sum\limits_a I_a}$, ${I_a=J(I_a)\lhd A}$, one has
${\alpha(A)=\prod\limits_a \alpha(I_a)}$, ${A=J(A)}$. Therefore, to any radical 
of groups $\mathcal{T}$ and a subclass ${\mathfrak{M}'\subseteq \mathfrak{M}}$ 
that is closed under taking ideals and homomorphic images, such that all 
${A\in \mathfrak{M}'}$, ${\alpha(A)=\mathcal{T}(\alpha(A))}$, are quasiregular 
(${A=J(A)}$), there corresponds a radical $\mathcal{T}'$ on $\mathfrak{M}'$ 
with the class of $\mathcal{T}'$--radical algebras 
${\mathfrak{R}_{\mathcal{T}, \alpha}\cap \mathfrak{M}'}$. The latter is 
true, in particular, for the class of all quasiregular associative 
$F$--algebras and any radical of groups $\mathcal{T}$. These conclusions are 
generalized to the mapping ${}_n\alpha$ (${{}_1\alpha=\alpha}$), ${n\geq 1}$, 
\[
{}_n\alpha(A)\ =\ GL_n(A^1)\cap (E+M_n(A))\ =\ E+J(M_n(A))\,,\quad 
{}_n\alpha(I)_A\ =\ {}_n\alpha(I)\quad (I\lhd A\in \mathfrak{M})\,,
\]
where ${GL_n(A^1)=U(M_n(A^1)), 
E=(\delta_{ij})_{i, j=1}^n=\sum\limits_{i=1}^n E_{ii}, 
E_{ij}=(\delta_{si}\delta_{tj})_{s, t=1}^n\in M_n(A^1), I=J(I)}$ 
for ${I\lhd A}$ is equivalent to ${M_n(I)=J(M_n(I))}$.
For all ${n\geq 3}$, ${k\geq 0}$, ${B\subseteq A}$ 
\[
B_n\ =\ \{t_{ij}(b)=E+b E_{ij}\mid 1\leq i\ne j\leq n,\ b\in B\}
\subseteq {}_n\alpha(A)\,,\quad g_k(B)_n\subseteq g^{gr}_k(B_n)\,,
\]
due to ${[t_{ij}(x), t_{jk}(y)]=t_{ik}(xy)}$, ${i\ne j\ne k}$, ${x, y\in A}$. 
If ${n\geq 3}$, ${{}_n\alpha(A)=T({}_n\alpha(A))}$, then for any 
${B\subseteq A}$, ${|B|< \infty}$, there is ${k=k(B)\geq 0}$, 
${g^{gr}_k(B_n)=\{E\}}$, ${g_k(B)=\{0\}}$ and, consequently, ${A=T(A)=LN(A)}$. 
If ${A=LN(A)}$, then ${M_n(A)=LN(M_n(A))}$, 
${g^{gr}_l(E+C)\subseteq E+\langle C\rangle^{2^l}}$ for all 
${C\subseteq M_n(A)}$, ${E+C\subseteq {}_n\alpha(A)}$, ${l\geq 0}$, for 
${|C|< \infty}$ there exists ${m\geq 0}$, ${\langle C\rangle^{2^m}=\{0\}}$, 
${g^{gr}_m(E+C)=\{E\}}$, ${{}_n\alpha(A)=E+M_n(A)=T({}_n\alpha(A))}$. 
As a consequence, ${}_n\alpha$, ${n\geq 3}$, connects the weakly solvable 
radical of groups $T$ and the locally nilpotent radical $LN$ of associative 
algebras (the prime radicals (lower nil-radicals) of groups and associative 
$PI$--algebras \cite[Theorem~2.6]{Gol7}), ${LN=T'}$. In a similar way, 
the prime radicals of groups and associative algebras $\prr$ are related by 
means of ${}_n\alpha$, ${n\geq 3}$, ${\prr=\prr'}$, although $\prr$ is not a 
radical in the sense of Kurosh --- Amitsur on the class of all groups 
(algebras) (Addition 2). If ${{}_n\alpha(A)=\prr({}_n\alpha(A))}$, then 
for any ${x_0=x, a_l, b_l, c_l\in A, l\geq 0}$ and 
${x_{l+1}=x_l a_l x_l b_l c_l}$,
\[
t_{ij}(x_{l+1})\ =\ 
[[t_{ij}(x_l), [t_{ji}(a_l), [t_{ij}(x_l), t_{jk}(b_l)]]], t_{kj}(c_l)]\in 
(t_{ij}(x_l))_{{}_n\alpha(A)}^2\quad (i\ne j\ne k)\,,
\]
${E\in \{t_{ij}(x_l)\}_{l=0}^{\infty}}$, ${0\in \{x_l\}_{l=0}^{\infty}}$, 
${x\in \prr(A)}$, ${A=\prr(A)}$. If ${A=\prr(A)}$, then  
\begin{gather*}
M_n(A)\ =\ \prr(M_n(A))\ =\ 
M_n(A)\cap \prr(M_n(A^1))\ =\ 
M_n(A)\cap M_n(\prr(A^1))\,,
\\
(E+(a_{ij}))_{{}_n\alpha(A)}\subseteq E+((a_{ij}))_{M_n(A)}\,,\  
(E+(a_{ij}))_{{}_n\alpha(A)}^2\subseteq E+((a_{ij}))_{M_n(A)}^2\ 
((a_{ij})\in M_n(A))\,,
\end{gather*}
${{}_n\alpha(A)=E+M_n(A)=\prr({}_n\alpha(A))}$ \cite[Proposition~3.6]{Gol7}.

For the mapping ${}_n\beta$, ${n\geq 2}$, 
\[
{}_n\beta(A)\ =\ E(A^1, A)\,,\quad {}_n\beta(I)_A\ =\ E(A^1, I)\ =\ 
(E(I))_{E(A^1)}\quad (I\lhd A\in \mathfrak{M})\,,
\]
${E(B)=\langle B_n\rangle}$, ${B\subseteq A^1}$, 
${{}_n\beta(A)\subseteq {}_n\alpha(A)}$, Lemma 3.1 allows us to construct 
radicals on the class ${\mathfrak{M}'\subseteq \mathfrak{M}}$ (see above) 
based on the radicals of groups $\mathcal{T}$ with the conditions: for any 
${I\lhd A\in \mathfrak{M}'}$
\begin{itemize}

\item[(1)] ${{}_n\beta_A(I)=\mathcal{T}({}_n\beta_A(I))}$ if 
${{}_n\beta(I)=\mathcal{T}({}_n\beta(I))}$;

\item[(2)] ${({}_n\beta(A)\cap (E+M_n(I)))/{}_n\beta_A(I)=
\mathcal{T}(({}_n\beta(A)\cap (E+M_n(I)))/{}_n\beta_A(I))}$\\ 
(this can be limited to the case 
${{}_n\beta(A/I)=\mathcal{T}({}_n\beta(A/I))}$, 
${{}_n\beta(I)=\mathcal{T}({}_n\beta(I))}$),

\end{itemize}
since the canonical epimorphism ${\psi: A\longrightarrow A/I}$, 
${I\lhd A\in \mathfrak{M}}$, induces an epimorphism 
\[
\psi_{{}_n\beta}\ =\ \psi^1_n|_{{}_n\beta(A)}\,:\ 
{}_n\beta(A)\longrightarrow {}_n\beta(A/I)\,,\quad 
\Ker \psi_{{}_n\beta}\ =\ {}_n\beta(A)\cap (E+M_n(I))\,,
\]
where ${\psi^1_n: M_n(A^1)\longrightarrow M_n((A/I)^1)}$, 
${\psi^1_n((x_{ij}))=(\psi^1 x_{ij})}$, ${\psi^1(f+x)=f+x+I}$, ${f\in F}$, 
$x\in A$, ${(x_{ij})\in M_n(A^1)}$. In particular, this 
is true for $\mathcal{T}$ such that for any ${I\lhd A\in \mathfrak{M}'}$
\begin{itemize}

\item[(1)] ${A=J(A)}$ if ${{}_n\beta(A)=\mathcal{T}({}_n\beta(A))}$; 

\item[(2)] ${{}_n\beta(I)_A\cap D_n(A, I)=
\mathcal{T}({}_n\beta(I)_A\cap D_n(A, I))}$ if 
${{}_n\beta(I)=\mathcal{T}({}_n\beta(I))}$;

\item[(3)] ${D_n(A, I)=\mathcal{T}(D_n(A, I))}$ if 
${{}_n\beta(I)=\mathcal{T}({}_n\beta(I))}$, 
${{}_n\beta(A/I)=\mathcal{T}({}_n\beta(A/I))}$\\ (and so, ${A=J(A)}$), 

\end{itemize}
where ${D_n(A, I)=\Bigl\{\sum\limits_{i=1}^n x_i E_{ii}\in 
{}_n\beta(A)\Bigl| x_i\in 1+I\Bigr\}}$ and for ${I=J(I)\lhd A}$
\[
{}_n\beta(I)_A\ =\ ({}_n\beta(A)\cap D_n(A, I)) E(I)\,,\quad 
{}_n\beta(A)\cap (E+M_n(I))\ =\ D_n(A, I) E(I)
\] 
are the semidirect products of ${{}_n\beta(A)\cap D_n(A, I)}$ and $E(I)$, 
${D_n(A, I)}$ and $E(I)$, respectively, 
${{}_n\beta(I)_A, E(I)\lhd {}_n\beta(A)\cap (E+M_n(I)), 
{}_n\beta(I)_A\cap D_n(A, I)\lhd D_n(A, I)}$ (reducing matrices from 
${E+M_n(I)}$ to diagonal form by multiplication by matrices from $I_n$).
Like ${}_n\alpha$, ${}_n\beta$, ${n\geq 3}$, relates the radicals $T$ 
and $\prr$ of groups and associative algebras.

For any $F$--algebra $R$ and ${I\lhd R}$, we select in the group $\Aut_F(R)$ 
the submonoid ${C(R, I)}$ and its subgroup ${U(R, I)}$,
\[
C(R, I)\ =\ \{\phi\in \Aut_F(R)\mid (R)(\phi-\Id_R)\subseteq I\}\,,\, 
U(R, I)\ =\ \{\phi\in C(R, I)\mid \phi^{-1}\in C(R, I)\}\,,
\] 
${C(R, I) C(R, J)\subseteq C(R, I+J)}$ for all ${I, J\lhd R}$, due to 
\[
\phi \phi'-\Id_R\ =\ \phi (\phi'-\Id_R)+\phi-\Id_R\in C(R, I+J)\quad 
(\phi\in C(R, I),\ \phi'\in C(R, J))\,, 
\]
${C(R, I)=U(R, I)}$ is the kernel of the 
homomorphism ${\pi_I: \Aut_F(R)\longrightarrow \Aut_F(R/I)}$ induced 
by the canonical epimorphism ${R\longrightarrow R/I}$, 
${(x+I)(\pi_I \phi)=x \phi+I}$, ${x\in R}$, ${\phi\in C(R, I)}$, 
for ${I\lhd_{\Aut} R}$.
If $F$ is an algebra over a field $\mathbb{F}$, ${\Ch \mathbb{F}=0}$, 
then in $\Aut_F(R)$ we can distinguish the subgroups 
\begin{gather*}
E(R)\ =\ \langle \exp(D)\mid D\in \Der_{nil}(R)\rangle\,,\quad 
\overline{E}(R)\ =\ \langle \exp(D)\mid D\in \Der_n(R)\rangle\,,
\\ 
G(R)\ =\ \langle \exp(D)\mid D\in \Mder_{nil}(R)\rangle\,,\quad 
\overline{G}(R)\ =\ \langle \exp(D)\mid D\in \Mder_n(R)\rangle\,,
\end{gather*}
where for any ${D\in \Der_{nil}(R)}$, ${x\in R}$, ${x D^{n-1}\ne x D^n=0}$ at 
some ${n=n(D, x)\geq 1}$,
$x \exp(D)=\sum\limits_{i=0}^{n-1} \frac{x D^i}{i!}$. 
To any ${D\in \Der_{nil}(R)}$ (${D\in \Der_n(R)}$) there corresponds 
${\ad_D\in \Der_{nil}(M(R))}$ ($\ad_D\in \Der_n(M(R))$), 
for all ${I\lhd R}$, ${(I)D\subseteq I}$, ${\psi\in M(I)}$, 
${\phi\in \Mder_*(I)}$, ${*=a, sa, nil, n}$,
\[
\exp(-D|_I) \psi \exp(D|_I)\ =\ \psi \exp(\ad_{D|_I})\in M(I)\,,\quad 
\phi \exp(\ad_{D|_I})\in \Mder_*(I)\,,
\]
${\exp(-D) \exp(\phi) \exp(D)=\exp(\phi \exp(\ad_D))\in G(I)}$ 
($\overline{G}(I)$) for ${\phi\in \Mder_{nil}(I)}$ ($\Mder_n(I)$) 
and so, ${G(R), \overline{G}(R)\lhd E(R)}$, 
${G(I), \overline{G}(I)\lhd G(R)_I}$ (${G(I), \overline{G}(I)\lhd E(R)_I}$) 
for any ${I\lhd R}$ (${I\lhd_{Der} R}$), $G(R)_I$ ($E(R)_I$) is a subgroup 
of $\Aut_F(I)$ consisting of restrictions of elements $G(R)$ ($E(R)$) to $I$. 

If ${I\lhd_{\Der} R}$, ${D\in \Der(R)}$ or (and) ${I\lhd R}$, 
${D\in \Mder(R)}$, then for any ${x\in R}$, ${k\geq 0}$,  
$\phi\in C(R, I)$ there is ${x_k\in I}$, ${x D^k \phi=x D^k+x_k}$,  
\[
x \phi^{-1} D^k \phi\ =\ (x-x_0 \phi^{-1}) D^k \phi\ =\ 
x D^k+x_k-x_0 (\phi^{-1} D \phi)^k\in x D^k+I
\]
(${\Der(R)=\psi \Der(R) \psi^{-1}}$, ${\Mder(R)=\psi \Mder(R) \psi^{-1}}$, 
${\psi\in \Aut_F(R)}$), and for  
\[
D\ =\ \sum_i t_{i1}\cdots t_{i n_i}\in \Mder_{nil}(R)\,,\quad 
C\ =\ \phi^{-1} D \phi-D\in \Mder(R)\cap (M^R(S))_{M(R)}\,,
\]
${t_{ij}\in \{t_{x_{ij}}\mid t=l, r,\ x_{ij}\in R\}}$, ${n_i\geq 1}$, 
${S=\{x_{ij} \phi-x_{ij}\}_{i, j}\subseteq I}$, 
\[
[\phi^{-1}, \exp(D)]\ =\ \exp(\phi^{-1} D \phi) \exp(-D)\ =\ 
\exp(D+C) \exp(-D)\in V(R, S)\subseteq G(R, (S)_R)\,,
\]
where for all ${B\subseteq R}$, ${J\lhd R}$ 
\begin{gather*}
V(R, B)\ =\ \{\phi\in G(R)\mid x \phi, x \phi^{-1}\in x+x (M^R(B))_{M(R)}\
\forall\, x\in R\}\,,
\\
G(R, J)\ =\ G(R)\cap U(R, J)\ =\ G(R)\cap C(R, J)\,,
\end{gather*}
${G(R, J), \overline{G}(R, J)=\overline{G}(R)\cap G(R, J), 
V(R, B), \overline{V}(R, B)=\overline{G}(R)\cap V(R, B)\lhd G(R),}$ in 
view of ${x G(R) M G(R)\subseteq x M\lhd R}$, ${x\in R}$, ${M\lhd M(R)}$,
${V(R, B) V(R, B')\subseteq V(R, B\cup B')}$, ${B, B'\subseteq R}$. Hence 
for any ${I\lhd R}$
\begin{multline*}
[U(R, I), G(R)]\ =\ 
\langle [\phi, \exp(D)]\mid \phi\in U(R, I),\ D\in \Mder_{nil}(R)\rangle
\subseteq
\\ 
\prod_{S\subseteq I,\ |S|< \infty} V(R, S)\subseteq V(R, I)\,,
\end{multline*} 
${[U(R, I), \overline{G}(R)]\subseteq \overline{V}(R, I)}$. 
Since for any associative ring $A$ with 1, ${a, b\in U(A)}$,
\begin{gather*}
[a, b]\ =\ 1+((a-1)(b-1)+(b-1) (a^{-1}-1)+(a-1) (b-1) (a^{-1}-1)) b^{-1}\,,
\\
V(R, B)^k\subseteq \{\phi\in G(R)\mid 
x \phi, x \phi^{-1}\in x+x (M^R(B))_{M(R)}^k\ \forall\, x\in R\}\quad 
(B\subseteq R,\ k\geq 1)\,.
\end{gather*}

Remark 2.1 in \cite{Gol9} requires clarification, which does not affect the 
remaining conclusions of \cite{Gol9} (they concern algebras in which powers 
of ideals are ideals). If $R$ is an algebra and ${k\geq 1}$ such that 
${J^l\lhd R}$ for all ${J\lhd R}$, ${l=1, \ldots, k}$, then to any 
${I\lhd R}$, ${I^{k+1}=\{0\}}$, there corresponds ${(M^R(I))_{M(R)}\lhd M(R)}$, 
${(M^R(I))_{M(R)}^{k+1}=\{0\}}$ (${A(I, R)^{k+1}=A(I, R, D)^{k+1}=\{0\}}$ in 
the notations of \cite{Gol9}; ${(M^R(I))_{M(R)}^n=\{0\}}$ implies 
${I^{n+1}=\{0\}}$, ${n\geq 1}$, without additional conditions). 

We call an ideal $I$ of an algebra $R$ \emph{$M$--nilpotent} if the ideal 
${(M^R(I))_{M(R)}\lhd M(R)}$ is nilpotent. We choose the sum $L(R)$ of 
all $M$--nilpotent ideals of $R$ and the smallest ideal $\prr_M(R)$ among 
all ${J\lhd R}$, ${L(R/J)=\{0\}}$, 
${\prr_M(R)=\bigcup\limits_{\alpha\geq 0} L_{\alpha}(R)}$, where 
${\{L_{\alpha}(R)\mid \alpha\geq 0\}}$ is a series of ideals of $R$, 
constructed for the mapping ${L: R\longmapsto L(R)}$ similarly to the Baer 
chain, ${\prr(R)\subseteq \prr_M(R)\subseteq \prr_N(R)}$ 
(${\prr_M(R)=\prr_N(R)}$ if in $R/\prr_M(R)$ the powers of ideals are ideals).
Following the elementwise description of $\prr$ \cite[Proposition~4,~Theorem~6,~p.~192,~193]{ZShS}, it can be shown that  
\begin{multline*}
\prr_M(R)\ =\ 
\bigl\{x=x_0\in R\bigl| \forall\, 
x_{i+1}\in R (M^R(\{x_i\}))_{M(R)}^2,\ i\geq 0,\ 
0\in \{x_i\}_{i=0}^{\infty}\bigr\}\ =
\\
\bigl\{x=x_0\in R\bigl| \forall\, 
x_{i+1}\in R (M^R((x_i)_R))_{M(R)}^2,\ i\geq 0,\ 
0\in \{x_i\}_{i=0}^{\infty}\bigr\}\ =\ 
\bigcap\limits_{P\in \Spec_{\prr_M}(R)} P\,,
\end{multline*}
where ${\Spec_{\mathcal{T}}(R)=\{P\in \Spec(R)\mid \mathcal{T}(R/P)=\{0\}\}, 
\mathcal{T}=\prr_*, *=N, S, M}$, or any radical of $F$--algebras. Only 
${\prr_M(R/P)=\{0\}}$ for ${\{x_i\}_{i=0}^{\infty}\subseteq R\setminus \{0\}, 
x_{i+1}\in R (M^R((x_i)_R))_{M(R)}^2, i\geq 0,}$ and the maximal 
ideal $P$ among ${I\lhd R}$, ${I\cap \{x_i\}_{i=0}^{\infty}=\emptyset}$, 
needs explanation. It is sufficient to note that for all ${J\lhd R}$, 
${k\geq 1}$ 
\[
R (M^R(R (M^R(J))_{M(R)}^k))_{M(R)}^2\subseteq 
J (M^R(R (M^R(J))_{M(R)}^k))_{M(R)}\subseteq 
R (M^R(J))_{M(R)}^{k+1}
\]
and, as a consequence, if ${y=y_0\in J}$, 
${y_{i+1}\in R (M^R((y_i)_R))_{M(R)}^2}$, ${i\geq 0}$, 
${(M^R(J))_{M(R)}^n=\{0\}}$ for some ${n\geq 1}$, then ${y_l=0}$ for all 
${l\geq \max\{1, n-1\}}$. If $F$ is a field, ${\Ch F=0}$, ${I\lhd R}$ is 
$M$--nilpotent and ${D\in \Der(R)}$, then ${I+(I)D\lhd R}$ is $M$--nilpotent, 
due to  
\begin{multline*}
(M^R(I+(I)D))_{M(R)}\ =\ (M^R(I))_{M(R)}+(M^R((I)D))_{M(R)}\ =
\\ 
(M^R(I))_{M(R)}+((M^R(I))_{M(R)})\ad_D
\end{multline*}
and the nilpotency of ${J+(J)D}$ for any associative $F$--algebra $A$, 
nilpotent ${J\lhd A}$, ${D\in \Der(A)}$ \cite[the~observations~after~Remark~1.1]{Gol9}. Using transfinite induction, we get from here that 
${L_{\alpha}(R)\lhd_{\Der} R}$ for all ${\alpha\geq 0}$, 
${\prr_M(R)\lhd_{\Der} R}$. Since 
${(M^R(J))_{M(R)}=\sum\limits_a (M^R(J_a))_{M(R)}}$ for any
${J=\sum\limits_a J_a, J_a\lhd R}$, the finite sums of $M$--nilpotent 
ideals are $M$--nilpotent. 

\begin{lemma} 
If $F$ is an algebra over a field $\mathbb{F}$, ${\Ch \mathbb{F}=0}$, 
$R$ is a $F$--algebra and $I$ is an ideal of $R$ such that the ideal 
${(M^R(I))_{M(R)}\lhd M(R)}$ is locally nilpotent (nilpotent), then 
the group ${\overline{G}(R, I)^{(1)}}$ (${G(R, I)^{(1)}}$) is locally 
nilpotent (nilpotent) and, as a consequence, 
$G(R, \prr(R))\subseteq G(R, \prr_M(R))\subseteq \prr(G(R))$.
\end{lemma} 

\begin{proof}
The group ${\Id_R+(M^R(I))_{M(R)}}$ is locally nilpotent (nilpotent) and 
\[
\overline{G}(R, I)^{(1)}\subseteq [U(R, I), \overline{G}(R)]\subseteq 
\overline{G}(R)\cap (\Id_R+(M^R(I))_{M(R)})\subseteq 
\overline{V}(R, I)\,.
\]
If there is ${n\geq 1}$, ${(M^R(I))_{M(R)}^n=\{0\}}$, then  
${(G(R, I)^{(1)})^n=V(R, I)^n=\{\Id_R\}}$ (see above). If $(M^R(I))_{M(R)}$ is 
a sum (union) of nilpotent ideals $M(R)$, then the 
groups ${V(R, S)\lhd G(R)}$, ${S\subseteq I}$, ${|S|< \infty}$, are nilpotent,
\[
G(R, I)^{(1)}\subseteq [U(R, I), G(R)]\subseteq 
\prod_{S\subseteq I,\ |S|< \infty} V(R, S) 
\subseteq V(R, I)\subseteq G(R, I)\subseteq \prr(G(R))\,.
\]
Each canonical epimorphism ${R\longrightarrow R/J}$ induces a homomorphism 
${\pi_J: G(R)\longrightarrow G(R/J)}$, ${\Ker \pi_J=G(R, J)}$, 
${\pi_J(G(R, J'))\subseteq G(R/J, J'/J)}$ for any ${J, J'\lhd R}$, 
${J\subseteq J'}$. Therefore to the series of ideals 
${\{L_{\alpha}(R)\mid \alpha\geq 0\}}$ there 
corresponds a non-decreasing chain of normal subgroups 
${\{G(R, L_{\alpha}(R))\mid \alpha\geq 0\}}$ of the group $G(R)$ from 
${G(R, \{0\})=\{\Id_R\}}$ to ${G(R, \prr_M(R))}$,  
${\prr_M(R)=L_{\gamma}(R)=\bigcup\limits_{\alpha\geq 0} L_{\alpha}(R)}$ 
for some ${\gamma\geq 0}$, where for any non-limit $\alpha$  
\begin{multline*}
G(R, L_{\alpha}(R))/G(R, L_{\alpha-1}(R))
\cong 
\pi_{L_{\alpha-1}(R)}(G(R, L_{\alpha}(R)))\subseteq
\\
G(R/L_{\alpha-1}(R), N_{\alpha}(R)/L_{\alpha-1}(R))\ =\  
G(R/L_{\alpha-1}(R), L(R/L_{\alpha-1}(R)))\subseteq
\\ 
\prr(G(R/L_{\alpha-1}(R)))
\end{multline*}
and limit $\alpha$ 
${L_{\alpha}(R)=\bigcup\limits_{\beta< \alpha} L_{\beta}(R)}$,
\begin{multline*}
G(R, L_{\alpha}(R))^{(1)}\subseteq 
\prod_{S\subseteq L_{\alpha}(R),\ |S|< \infty} V(R, S)\subseteq 
\prod_{\beta< \alpha} V(R, L_{\beta}(R))\subseteq
\\ 
\prod_{\beta< \alpha} G(R, L_{\beta}(R))\subseteq G(R, L_{\alpha}(R))\,.
\end{multline*}
Using transfinite induction, we obtain that 
${G(R, L_{\alpha}(R))\subseteq \prr(G(R))}$ for all ${\alpha\geq 0}$ and 
so, ${G(R, \prr_M(R))\subseteq \prr(G(R))}$. Similar reasoning applies 
when replacing $\prr_M(R)$ and ${\{L_{\alpha}(R)\mid \alpha\geq 0\}}$ 
with $\prr(R)$ and the Baer chain ${\{B_{\alpha}(R)\mid \alpha\geq 0\}}$ 
(${(M^R(I))_{M(R)}^2=\{0\}}$ for all ${\{0\}=I^2\ne I\lhd R}$).
\end{proof}

\begin{prop}
If $R$ is an associative $PI$--ring with $1$, $G$ is a subgroup of $U(R)$, 
then ${\prr(G)=T(G)}$ is a solvable extension of the product 
of all nilpotent normal subgroups of $G$ \cite[Theorem~2.6]{Gol7}.
\end{prop}

For any $F$--algebra $R$ there exists a homomorphism 
${{}_A I: \Aut_F(R)\longrightarrow \Aut_F(A)}$, 
\[
{}_A I \phi\,:\ \psi\longmapsto \phi \psi \phi^{-1}\quad 
(\phi\in \Aut_F(R),\ \psi\in A=\End_F(R), M(R), M(R)')\,.
\]
If $F$ is an algebra over a field $\mathbb{F}$, ${\Ch \mathbb{F}=0}$, then 
${{}_A I \exp(D)=\exp(-\ad_D)}$ for ${D\in \Der_{nil}(R)}$ with 
${\ad_D\in \Der_{nil}(A)}$ and ${{}_A I: E(R)\longrightarrow E(A), 
G(R)\longrightarrow \Inn(A), A=M(R), M(R)'}$, and for 
${D\in \Der_n(R)}$ with ${\ad_D\in \Inder_n(A)}$ and 
${{}_A I: \overline{E}(R)\longrightarrow \overline{\Inn}(A)}$, ${A=\End_F(R)}$, 
where for any $F$--algebra $R$ 
\[
\Inn(R)\ =\ \langle \exp(D)\mid D\in \Inder_{nil}(R)\rangle\,,\quad 
\overline{\Inn}(R)\ =\ \langle \exp(D)\mid D\in \Inder_n(R)\rangle\,.
\]
If ${\phi\in \Aut_F(R)}$, ${{}_A I \phi(t_x)=t_{x \phi^{-1}}=t_x}$ for any 
${t=l, r}$, ${x\in R}$, then 
\[
R (R)(\phi^{-1}-\Id_R)\ =\ (R)(\phi^{-1}-\Id_R) R\ =\ \{0\}\,,\quad 
\phi^{-1}|_{R^2}\ =\ \Id_{R^2}
\]
(${(x y) \phi^{-1}=(x \phi^{-1})(y \phi^{-1})=x y}$, ${x, y\in R}$). Therefore 
if ${\Ann R=\{0\}}$ or (and) ${R=R^2}$, then ${\Ker {}_A I=\{\Id_R\}}$, 
${}_A I$ is an embedding, ${A=\End_F(R), M(R), M(R)'}$. Everywhere below until 
the Additions, ${F=\mathbb{F}}$ is a field, ${\Ch \mathbb{F}=0}$, 
$R$ is a $\mathbb{F}$--algebra. 

If $A$ is a semiprime associative $PI$--algebra, then $M(A)$ is a semiprime 
$PI$--algebra and ${\pideg M(A)=(\pideg A)^2}$, since $M(A)$ is a subdirect 
product of $M(A/P)$, ${P\in \Spec(A)}$ 
(epimorphisms ${M(A)\longrightarrow M(A/P)}$ are induced by canonical 
epimorphisms ${A\longrightarrow A/P}$, $P\in \Spec(A)$), $P(A/P)$ for 
${A\ne P\in \Spec(A)}$ is a central simple finite-dimensional algebra over 
a field $\CM(A/P)$, $M(P(P/A))$ is a finite-dimensional primitive 
$\CM(A/P)$--algebra with an exact irreducible module $P(A/P)$ and so, 
\begin{multline*}
M(P(P/A))\ =\ \CM(A/P) M^{P(A/P)}(A/P)\ =
\\ 
\End_{\End(P(A/P))_{M(P(A/P))}}(P(A/P))\cong 
M_{\dim_{\CM(A/P)} P(A/P)}(\CM(A/P))\,, 
\end{multline*}
${M(A/P)\cong M^{P(A/P)}(A/P)}$ is prime and has homogeneous identities with 
integer coefficient in common with $M(P(A/P))$, 
${\pideg M(A/P)=\pideg M(P(A/P))=\dim_{\CM(A/P)} P(A/P)}$ 
\cite[Introduction]{Gol10}, \cite[Proposition~3~p.~92]{Lamb}. Recall 
that by a corollary of Regev's theorem on the tensor product of $PI$--rings, 
the algebras of multiplications of associative $PI$--algebras are 
$PI$--algebras \cite[Theorem~6.1.1,~p.~239]{Row}.

\begin{teor}
Let $R$ be a $\mathbb{F}$--algebra, ${B_0=\prr_M(R)}$, ${A_0=M(R/B_0)}$, 
${B_{i+1}=\prr(A_i)}$, ${A_{i+1}=M(A_i/B_{i+1})}$ for all ${i\geq 0}$. Then 
\begin{multline*}
G_0\ =\ G(R, B_0)\subseteq 
G_1\ =\ \pi_{B_0}^{-1} {}_{A_0} I^{-1}(G(A_0, B_1))
\subseteq \ldots \subseteq
\\ 
G_{i+1}\ =\ 
\pi_{B_0}^{-1} {}_{A_0} I^{-1} \pi_{B_1}^{-1} {}_{A_1} I^{-1}\cdots 
\pi_{B_i}^{-1} {}_{A_i} I^{-1}(G(A_i, B_{i+1}))\subseteq \ldots \subseteq 
\prr(G(R))\,.
\end{multline*}
If ${k\geq 0}$, $A_k/B_{k+1}$ is a $PI$--algebra 
(${\sup\limits_{R\ne P\in \Spec_{\prr_M}(R)} \dim_{\CM(R/P)} P(R/P)< \infty}$), 
$G$ is a subgroup of $G(R)$, then ${\prr(G)=T(G)}$ is a solvable extension of 
${G\cap G_{k+1}}$ (${G\cap G_0}$).
\end{teor}

\begin{proof}
Since for all ${i\geq 0}$ with ${A_{-1}=R}$
\[
\pi_{B_i}\,:\ G(A_{i-1})\longrightarrow G(A_{i-1}/B_i)\,,\quad 
\Ker \pi_{B_i}\ =\ G(A_{i-1}, B_i)\subseteq \prr(G(A_{i-1}))
\]
and ${{}_{A_i} I: G(A_{i-1}/B_i)\hookrightarrow \Inn(A_i)}$ 
(Lemma 3.2; ${\prr(A_i)=\prr_M(A_i)}$), 
\begin{multline*}
{}_{A_i} I^{-1}(G(A_i, B_{i+1}))\subseteq 
{}_{A_i} I^{-1}(\prr(G(A_i)))\ =\ 
{}_{A_i} I^{-1}({}_{A_i} I(G(A_{i-1}/B_i))\cap \prr(G(A_i)))
\subseteq
\\
\shoveright{ 
{}_{A_i} I^{-1}(\prr({}_{A_i} I(G(A_{i-1}/B_i))))\ =\ 
\prr(G(A_{i-1}/B_i))\,,}
\\
\shoveleft{
\pi_{B_i}^{-1}(\prr(G(A_{i-1}/B_i)))\ =\ 
\pi_{B_i}^{-1}(\pi_{B_i}(G(A_{i-1}))\cap \prr(G(A_{i-1}/B_i)))\subseteq}
\\
\shoveright{ 
\pi_{B_i}^{-1}(\prr(\pi_{B_i}(G(A_{i-1}))))\ =\ \prr(G(A_{i-1}))\,,}
\\
\shoveleft{
\pi_{B_i}^{-1} {}_{A_i} I^{-1}(G(A_i, B_{i+1}))\subseteq 
\pi_{B_i}^{-1} {}_{A_i} I^{-1}(\prr(G(A_i)))\subseteq \prr(G(A_{i-1}))\,,}
\\
G_{i+1}\ =\ \pi_{B_0}^{-1} {}_{A_0} I^{-1}\cdots 
\pi_{B_i}^{-1} {}_{A_i} I^{-1}(G(A_i, B_{i+1}))\subseteq 
\pi_{B_0}^{-1} {}_{A_0} I^{-1}\cdots 
\pi_{B_i}^{-1} {}_{A_i} I^{-1}(\prr(G(A_i)))\subseteq
\end{multline*}
\begin{multline*}
\pi_{B_0}^{-1} {}_{A_0} I^{-1}\cdots 
\pi_{B_{i-1}}^{-1} {}_{A_{i-1}} I^{-1}(\prr(G(A_{i-1})))\subseteq 
\ldots \subseteq \pi_{B_0}^{-1} {}_{A_0} I^{-1}(\prr(G(A_0)))\subseteq
\\ 
\prr(G(A_{-1}))\ =\ \prr(G(R))\,.
\end{multline*}
If $B$ is a semiprime algebra with 
${\sup\limits_{B\ne P\in \Spec(B)} \dim_{\CM(B/P)} P(B/P)< \infty}$ 
(in particular, a semi\-prime associative $PI$--algebra), then 
${\Mder_{nil}(B)=\Mder_n(B)}$ \cite{Mar, Row1}, \cite[Theorem~4.1,~p.~47]{Raz}. If $A_k/B_{k+1}$ is a $PI$--algebra for some 
${k\geq 0}$, $G$ is a subgroup of $G(R)$, then ${B_{k+j}=0}$, 
${G_{k+1}=G_{k+j}}$ for all ${j\geq 2}$, 
${G(A_k/B_{k+1})=\overline{G}(A_k/B_{k+1})\subseteq U(A_{k+1}')}$ 
for ${A_{k+1}'=M(A_k/B_{k+1})'}$ and ${\prr(G')=T(G')}$
for ${G'=\pi_{B_{k+1}} {}_{A_k} I\cdots \pi_{B_1} {}_{A_0} I \pi_{B_0}(G)}$, 
therefore ${\prr(G)=T(G)}$ (the observation before Theorem 3.4, 
Proposition 3.3). Since for any field $\Bbbk$ weakly solvable subgroups of 
the group $GL_n(\Bbbk)$, ${n\geq 1}$, are solvable of degree at most some 
${l(n)\geq 1}$ \cite[Proposition~1.33,~Theorem~1.35]{Gol7} and the 
group $G(A_k/P)$ is embeddable in the group $GL_{n_P}(\CM(A_k/P))$ for all 
${A_k\ne P\in \Spec(A_k)}$, 
\[
n_P\ =\ \dim_{\CM(A_k/P)} P(A_k/P)\leq 
m\ =\ (\pideg A_k)^2\,,
\]
$T(G_P)$ is a solvable radical of the group 
${G_P=\pi_P {}_{A_k} I \pi_{B_k}\cdots {}_{A_0} I \pi_{B_0}(G)
\subseteq G(A_k/P)}$ of solvability degree at most ${l(n_P)\leq l(m)}$ and 
\begin{multline*}
{}_{A_k} I \pi_{B_k} \cdots {}_{A_0} I \pi_{B_0}(T(G)^{(l(m))})\subseteq 
\bigcap_{A_k\ne P\in \Spec(A_k)} \pi_P^{-1}(T(G_P)^{(l(m))})
\ = 
\\
\bigcap_{A_k\ne P\in \Spec(A_k)} G(A_k, P)\ =\ G(A_k, B_{k+1})\,,
\end{multline*}
${T(G)^{(l(m))}\subseteq G\cap G_{k+1}\subseteq \prr(G)}$. For 
${\sup\limits_{R\ne P\in \Spec_{\prr_M}(R)} \dim_{\CM(R/P)} P(R/P)=n< \infty}$, 
the transition to $G(A_0)$ is not required, $A_0'$ satisfies ${\st_{2 n}=0}$, 
${\pi_{B_0}(G)\subseteq G(A_0)=\overline{G}(A_0)\subseteq U(A_0')}$,
${T(G)^{(l(n))}\subseteq G\cap G_0}$, ${G_1=G_j}$ for all ${j\geq 2}$ 
and in the presence of ${l\geq 1}$, ${(M^R(B_0))_{M(R)}^l=\{0\}}$, 
$T(G)$ is a solvable radical of $G$ of solvability degree at most ${l+l(n)+1}$ 
(the proof of Lemma 3.2).
\end{proof} 

If $R$ is a $m.s.p.$--algebra, then  
${{}_{M(R)} I: \overline{E}(R)\hookrightarrow 
\overline{\Inn}(Q^s_m(M(R)))_{M(R)}}$ and for $R$ with a $SGPI$--algebra 
$M(R)$ ${{}_{M(R)} I: E_c(R)\hookrightarrow \Inn(Q^s_m(M(R)))_{M(R)}}$, where
\begin{multline*}
E_c(R)\ =\ 
\langle \exp(D)\mid D\in \Der_{nil}(R),\ [\overline{D}, \CM(R)]=\{0\}\rangle\,,
\\
\shoveleft{
\Inn(Q^s_m(M(R)))_{M(R)}\ =}
\\
\shoveright{ 
\langle \exp(D|_{M(R)})\mid D\in \Inder(Q^s_m(M(R))),\ 
(M(R))D\subseteq M(R),\ D|_{M(R)}\in \Der_{nil}(M(R))\rangle\,,}
\\
\shoveleft{
\overline{\Inn}(Q^s_m(M(R)))_{M(R)}\ =}
\\ 
\{\phi|_{M(R)}\mid \phi\in \langle \exp(D)\mid D\in \Inder_n(Q^s_m(M(R))),\ 
(M(R))D\subseteq M(R)\rangle\}
\end{multline*}
(Corollaries 2.11, 2.13). For $R$ with a $PI$--algebra $M(R)$, 
${Q(M(R))=\hat{M(R)}}$, from $D\in \Mder(\hat{M(R)})$, 
${(M(R))D\subseteq M(R)}$, ${D|_{M(R)}\in \Der_{nil}(M(R))}$ it follows that 
${D\in \Mder_n(\hat{M(R)})}$,
\begin{multline*}
\Inn(\hat{M(R)})_{M(R)}\ =\ 
\overline{\Inn}(\hat{M(R)})_{M(R)}\cong 
\\
\langle \exp(D)\mid D\in \Inder_n(\hat{M(R)}),\ (M(R))D\subseteq M(R)\rangle
\subseteq U(M(\hat{M(R)}))\,, 
\end{multline*}
since ${\sup\limits_{\hat{M(R)}\ne P\in \Spec(\hat{M(R)})} 
\dim_{{}_{{\hat M(R)}/P} C} \widehat{\hat{M(R)}/P}=n< \infty}$, 
${((\hat{M(R)}+P)/P)D_P^n=\{0\}}$ for $D_P\in \Mder(\hat{M(R)}/P)$ 
induced by $D$, ${P\in \Spec(\hat{M(R)}), (\hat{M(R)})D^n=\{0\}}$.
Using $\pideg M(R)=\pideg \hat{M(R)}$, Lemma 2.7, 
Corollary 2.13 and the proof of Theorem 3.4, we can deduce from here 

\begin{co}
If $R$ is a $m.s.p.$--algebra with a $PI$--algebra $M(R)$, $G$ is a subgroup 
of $E_c(R)$, then ${E_c(R)=\overline{E}(R)}$, ${\prr(G)=T(G)}$ is a solvable 
radical of $G$ of solvability degree at most 
${n(\hat{M(R)})\leq l((\pideg M(R))^2)}$, where  $n(\hat{M(R)})$ is the 
maximal among all solvability degrees of subgroups of 
$U(M(\hat{M(R)}/P))$, ${\hat{M(R)}\ne P\in \Spec(\hat{M(R)})}$.
\end{co} 

Therefore, if $R$ is a $\mathbb{F}$--algebra, $G$ is a subgroup of $E(R)$, $I$ 
is a $G$--invariant ideal of $R$ such that the algebra $R/I$ under the 
conditions of Corollary 3.5 and ${\pi_I(G)\subseteq E_c(R/I)}$, then $T(G)$ is 
a solvable extension of ${T(G)\cap C(R, I)}$ of degree at most 
$n(\hat{M(R/I)})$. In particular, this can be formulated for 
${I=\prr(R), \prr_N(R), \prr_S(R), \prr_M(R)\lhd_{\Aut} R}$. In general, 
${C(R, I)\ne T(C(R, I))}$ even for such $I$. Indeed, if ${R^2=\{0\}}$, then 
\[
C(R, R)\ =\ \Aut_{\mathbb{F}}(R)\ =\ U(\End_{\mathbb{F}}(R))
\] 
and for ${\dim_{\mathbb{F}} R=n< \infty}$ ${C(R, R)\cong GL_n(\mathbb{F})}$, 
\[
E(R)\ =\ \overline{E}(R)\ =\  
\langle \exp(D)\mid D\in \Der(R)=\End_{\mathbb{F}}(R),\ D^n=0 \rangle\cong 
SL_n(\mathbb{F})\,,
\]
${T(GL_n(\mathbb{F}))=C(GL_n(\mathbb{F}))\cong U(\mathbb{F})=\mathbb{F}^*}$ 
(can be deduced from the Jordan --- Dickson theorem on the simplicity of 
$PSL_m(\Bbbk)$ over any field $\Bbbk$ with ${|\Bbbk|> 3}$ for ${m=2}$, the 
unsolvability of $SL_m(\Bbbk)$ under the same restrictions and 
\[
C(GL_n(A))\ =\ U(Z(A)) E\ =\ \{(x_{ij})\in GL_n(A)\mid 
[(x_{ij}), t_{kl}(x)]\in U(Z(A)) E\ \forall\, k\ne l,\ x\in A\}
\]
for any associative ring $A$ with 1; \cite[Theorem~1.35,~Proposition~4.10]{Gol7}).

If $A$ is a semiprime associative $\mathbb{F}$--algebra, then by analogy with 
the previous we have ${\overline{E}(A)\subseteq \overline{\Inn}(Q^s_m(A))_A}$, 
${E_c(A)\subseteq \Inn(Q^s_m(A))_A}$ for the $SGPI$--algebra $A$, where  
\begin{gather*}
\Inn(Q^s_m(A))_A\ =\ 
\langle \exp(D|_A)\mid D\in \Inder(Q^s_m(A)),\ (A)D\subseteq A,\ 
D|_A\in \Der_{nil}(A)\rangle\,,
\\
\overline{\Inn}(Q^s_m(A))_A\ =\ 
\{\phi|_A\mid \phi\in \langle \exp(D)\mid D\in \Inder_n(Q^s_m(A)),\ 
(A)D\subseteq A\rangle\}\,.
\end{gather*}
For the $PI$--algebra $A$ one has 
${E_c(A)=\overline{E}(A)\hookrightarrow U(M(\hat{A}))}$, any subgroup 
${G\subseteq \overline{E}(A)}$ possess a solvable radical ${\prr(G)=T(G)}$ 
of solvability degree at most 
${n(\hat{A})\leq l((\pideg A)^2)}$, ${\Der_{1, n}(A)=\Der_n(A)}$, 
${E_c(A)=\langle \exp(D)\mid D\in \Der_{1, n}(A)\rangle}$. The latter 
follows from the fact that for any ${D\in \Der_{1, n}(A)}$ and 
${A\ne P\in \Spec_{\mathcal{N}}(A)=\Spec(A) (\prr(A)=\mathcal{N}(A)), 
(P)D\subseteq P}$, ${S=Z(A/P)\setminus \{0\}\ne \emptyset}$, 
${P(A/P)=\hat{A/P}=S^{-1} A/P}$, 
\[
0\ =\ 1 D_P\ =\ s D_P\ =\ (s^{-1} D_P) s+s^{-1} (s D_P)\ =\ s^{-1} D_P\quad 
(s\in S)\,,
\]
${(Z(\hat{A/P}))D_P=(S^{-1} Z(A/P))D_P=\{0\}}$, where $D_P$ is the extension 
of ${D_P\in \Der(A/P)}$ indu\-ced by $D$ to $\hat{A/P}$ (Corollary 2.3, (2.7.2)), and therefore ${D^m=0}$ for ${m=1}$ if ${A=\{0\}}$, and  
${m=\sup\limits_{A\ne P\in \Spec(A)} \dim_{Z(\hat{A/P})} \hat{A/P}< \infty}$, 
otherwise.

We will consider analogues of the obtained conclusions for algebraic (locally 
finite) deri\-vations of algebras. Let $\mathbb{F}((t))$ and $R((t))$ be a 
field and a $\mathbb{F}((t))$--algebra of Laurent series over a field 
$\mathbb{F}$ and a $\mathbb{F}$--algebra $R$ in a variable $t$, 
${\overline{R}((t))=\mathbb{F}((t)) R}$ be a $\mathbb{F}((t))$--subalgebra 
of $R((t))$ consisting of series with coefficients from finite-dimensional 
$\mathbb{F}$--subspaces of $R$, 
${M(\overline{R}((t)))=M^{\overline{R}((t))}(R)\cong \overline{M(R)}((t))}$ 
for ${\sum\limits_i \phi_i \psi_i\longmapsto \sum\limits_i \phi_i \psi_i|_R}$, 
${\phi_i\in \mathbb{F}((t))}$, ${\psi_i\in M^{\overline{R}((t))}(R)}$. We 
select in the group $\Aut_{\mathbb{F}((t))}(\overline{R}((t)))$ the subgroups 
\begin{gather*}
E_a(R, t)\ =\ \langle \exp(t D)\mid D\in \Der_a(R)\rangle\,,\quad 
\overline{E}_a(R, t)\ =\ \langle \exp(t D)\mid D\in \Der_{sa}(R)\rangle\,,
\\
G_a(R, t)\ =\ \langle \exp(t D)\mid D\in \Mder_a(R)\rangle\,,\quad
\overline{G}_a(R, t)\ =\ \langle \exp(t D)\mid D\in \Mder_{sa}(R)\rangle\,,
\\
G_a(R, I, t)\ =\ G_a(R, t)\cap C(\overline{R}((t)), \overline{I}((t)))\ =\ 
\Ker \pi_{\overline{I}((t))}\quad (I\lhd R)\,,
\end{gather*}
where the homomorphism 
${\pi_{\overline{I}((t))}: G_a(R, t)\longrightarrow G_a(R/I, t)}$ 
is induced by the canonical epimor\-phism 
${\overline{R}((t))\longrightarrow \overline{R/I}((t))\cong 
\overline{R}((t))/\overline{I}((t))}$, 
${\overline{I}((t))=I((t))\cap \overline{R}((t))}$,
\[
D\,:\ \sum_{i\geq m} t^i x_i\longmapsto \sum_{i\geq m} t^i (x_i D)\quad 
(x_i\in R,\ m\in \mathbb{Z},\ D\in \Der(R)\subseteq \Der(R((t))), 
\Der(\overline{R}((t))))\,, 
\]
${\Der_a(R)\subseteq 
\Der_a(\overline{R}((t)))\cap \Der_{a, \mathbb{F}}(\overline{R}((t)))}$
(for any ${D\in \Der_a(R)}$, 
${\sum\limits_{i=1}^k \phi_i x_i\in \overline{R}((t))}$, 
${\phi_i\in \mathbb{F}((t))}$, ${x_i\in R}$, ${k\geq 1}$ there are 
${f_i(t)\in \mathbb{F}[t]}$, ${x_i f_i(D)=0}$, 
${\Bigl(\sum\limits_{i=1}^k \phi_i x_i\Bigr) f(D)=0}$, 
${f(t)=\prod\limits_{i=1}^k f_i(t)}$).

For any ${B\subseteq R((t))}$, we denote by $B_0$ the set of the first 
non-zero coefficients of the series from $B$, supplemented by ${0\in R}$. 
If ${J\lhd R'}$, ${R'=\overline{R}((t)), R((t))}$, then ${J_0\lhd R}$, 
\[
J_0^n\subseteq (J^n)_0\,,\ J_0^{(m)}\subseteq (J^{(m)})_0\,,\ 
R(M^R(J_0))_{M(R)}^n\subseteq ((R'(M^{R'}(J))_{R'})^n)_0\quad 
(n\geq 1,\ m\geq 0)\,,
\]
the nilpotency (solvability, $M$--nilpotency) of $J$ is inherited by 
$J_0$. If ${J_0\subseteq J}$, then ${J\subseteq J_0((t))}$, 
${J=\overline{J_0}((t))}$ for ${J\subseteq \overline{R}((t))}$, the 
nilpotency (solvability, $M$--nilpotency) conditions for $J$ and $J_0$ 
are equivalent, nilpotent (solvable) $J$ and $J_0$ have equal nilpotency 
indices (solvability degrees), since for all ${I\lhd R}$ 
\begin{multline*}
\overline{I^n}((t))\ =\ \overline{I}((t))^n\subseteq 
I((t))^n\subseteq I^n((t))\,,\quad 
\overline{I^{(m)}}((t))\ =\ \overline{I}((t))^{(m)}\subseteq 
I((t))^{(m)}\subseteq I^{(m)}((t))\,,
\\
\shoveleft{
\overline{R}((t))(M^{\overline{R}((t))}(\overline{I}((t))))_{M(\overline{R}((t)))}^n
\ =\ 
\overline{R(M^R(I))_{M(R)}^n}((t))\subseteq}
\\ 
R((t))(M^{R((t))}(I((t))))_{M(R((t)))}^n\subseteq R (M^R(I))_{M(R)}^n((t))
\quad (n\geq 1,\ m\geq 0)\,.
\end{multline*}
Hence ${\mathcal{T}(\overline{R}((t))/\overline{\mathcal{T}(R)}((t)))=\{0\}}$, 
${\mathcal{T}(\overline{R}((t)))\subseteq \overline{\mathcal{T}(R)}((t))}$, 
${\mathcal{T}=\prr, \prr_*, *=N, S, M}$ and similarly for $R((t))$. 
If ${J=J(0)\lhd \overline{R}((t))}$ and 
${J(i+1)=J(i)+\overline{J(i)_0}((t))}$, ${i\geq 0}$, then for any  
${x=\sum\limits_{j\geq m} t^j x_j\in I}$, ${m\in \mathbb{Z}}$, there are 
${k\geq m}$, ${n\geq 0}$, 
${\sum\limits_{j\geq m} \mathbb{F} x_j=\sum\limits_{i=m}^k \mathbb{F} x_i, 
\{x_i\}_{i=m}^k\subseteq J(n)_0}$ and 
${x\in \overline{J(n)_0}((t))\subseteq J(n+1)}$. Since finite sums of solvable 
($M$--nilpotent) ideals of an algebra are solvable ($M$--nilpotent), it follows 
that ${\overline{S(R)}((t))=S(\overline{R}((t)))}$, 
${\overline{L(R)}((t))=L(\overline{R}((t)))}$ and by induction 
${\overline{\prr_*(R)}((t))=\prr_*(\overline{R}((t))), *=S, M}$ 
(for ${*=N}$ this is true for $R$ and $\overline{R}((t))$ from classes of 
$\mathbb{F}$--algebras and $\mathbb{F}((t))$--algebras closed under taking 
homomorphic images in which finite sums of nilpotent ideals are nilpotent; see 
also \cite[Remark~3.9]{Gol8}).  

If ${P\in \Spec(R)}$, ${I, J\lhd \overline{R}((t))}$ and 
${I J\subseteq \overline{P}((t))\subseteq I\cap J}$, then  
${I_0 J_0\subseteq P}$, and hence ${I_0\subseteq P}$, 
${I=\overline{I_0}((t))\subseteq \overline{P}((t))}$ or (and) 
${J_0\subseteq P}$, ${J=\overline{J_0}((t))\subseteq \overline{P}((t))}$. 
So, ${\overline{P}((t))\in \Spec_{\mathcal{T}}(\overline{R}((t)))}$ 
for all ${P\in \Spec_{\mathcal{T}}(R)}$, ${\mathcal{T}=\prr, \prr_*}$, 
${*=N, S, M}$ (${\Spec_{\prr}(R)=\Spec(R)}$).

If ${I\lhd R}$, ${\phi\in C(\overline{R}((t)), \overline{I}((t)))}$,  
${D=\sum\limits_i t_{i1}\cdots t_{i n_i}\in \Mder_a(R)}$, ${C_0=0}$, ${C_1=C}$, 
\[
C_n\ =\ \phi^{-1} D^n\phi-D^n\ =\ (D+C)^n-D^n\in \Mder(\overline{R}((t)))\cap 
(M^{\overline{R}((t))}(S))_{M(\overline{R}((t)))}\quad (n\geq 0)\,,
\]
${t_{ij}\in \{t_{x_{ij}}\mid t=l, r,\ x_{ij}\in R\}}$, ${n_i\geq 1}$, $S$ is a 
basis of a finite-dimensional $\mathbb{F}$--space $I$ generated by the 
coefficients of the series 
${\{x_{ij} \phi-x_{ij}\}_{i, j}\subseteq \overline{I}((t))}$, then 
in the notation below 
\[
[\phi^{-1}, \exp(t D)]\ =\ 
\exp(t (D+C)) \exp(-t D)\in V_a(R, S, t)\subseteq G_a(R, (S)_R, t)\,,
\]
since ${\phi^{-1} D \phi\in 
\Mder_a(\overline{R}((t)))\cap \Mder_{a, \mathbb{F}}(\overline{R}((t)))}$ and 
for any ${x\in \overline{R}((t))}$ there exist ${k=k(x)}$, ${l=l(x)\geq 0}$, 
${x (D^n+C_n)\in \sum\limits_{i=0}^k \mathbb{F} x (D^i+C_i)}$, 
${x C_n\in \sum\limits_{i=0}^k \mathbb{F} x (D^i+C_i)+
\sum\limits_{i=0}^k \mathbb{F} x D^i}$ for all ${n> k}$, 
${\sum\limits_{n\geq 0} \mathbb{F} x C_n=\sum\limits_{n=0}^l \mathbb{F} x C_n}$, 
\begin{gather*}
x [\phi^{-1}, \exp(t D)]\ =\ 
x+x \sum_{n\geq 0} \frac{t^n}{n!} C_n \exp(-t D)\in 
x+x \sum_{n=0}^l \mathbb{F}((t)) C_n \exp(-t D)\,,
\\
x [\exp(t D), \phi^{-1}]\ =\ 
x+x \exp(t D) \sum_{n\geq 0} \frac{(-1)^n t^n}{n!} C_n\in 
x+x \exp(t D) \sum_{n=0}^{l'} \mathbb{F}((t)) C_n\,, 
\end{gather*}
${l'=l(x \exp(t D))}$, ${x [\phi^{-1}, \exp(t D)], x [\exp(t D), \phi^{-1}]\in 
x+x (M^{\overline{R}((t))}(S))_{M(\overline{R}((t)))}}$. Therefore 
\begin{multline*}
[U(\overline{R}((t)), \overline{I}((t))), G_a(R, t)]\ =\ 
\langle [\phi, \exp(t D)]\mid \phi\in U(\overline{R}((t)), \overline{I}((t))),\ 
D\in \Mder_a(R)\rangle \subseteq
\\ 
\prod_{S\subseteq I,\ |S|< \infty} V_a(R, S, t)\subseteq V_a(R, I, t)\,,
\end{multline*} 
where ${V_a(R, B, t)\lhd G_a(R, t)}$ for all ${B\subseteq R}$, 
\begin{gather*}
V_a(R, B, t)\ =\ \{\phi\in G_a(R, t)\mid x \phi, x \phi^{-1}\in 
x+x (M^{\overline{R}((t))}(B))_{M(\overline{R}((t)))}\
\forall x\in \overline{R}((t))\}\,,
\\
V_a(R, B, t)^k\subseteq \{\phi\in G_a(R)\mid 
x \phi, x \phi^{-1}\in x+x (M^{\overline{R}((t))}(B))_{M(\overline{R}((t)))}^k\ 
\forall x\in \overline{R}((t))\}\quad (k\geq 1)\,.
\end{gather*}
Consequently, as in Lemma 3.2, ${G_a(R, \prr_M(R), t)\subseteq \prr(G_a(R, t))}$. 

Since ${\ad_D\in \Der_*(A)}$, ${\ad_D\in \Der_*(\overline{A}((t)))\cap 
\Der_{*, \mathbb{F}}(\overline{A}((t)))}$ for any ${D\in \Der_*(R)}$, 
${*=a}$, ${A=M(R)}$ and ${*=sa}$, ${A=\End_{\mathbb{F}}(R)}$ (the 
observations before Corollary 2.16, Remark 2.9),
${{}_{\overline{A}((t))} I \exp(t D)=\exp(-t \ad_D)}$ for ${D\in \Der_a(R)}$, 
${{}_{\overline{A}((t))} I: E_a(R, t)\longrightarrow E_a(A, t)}$, 
$G_a(R, t)\longrightarrow \Inn_a(A, t)$, ${A=M(R)}$, and ${D\in \Der_{sa}(R)}$, 
${{}_{\overline{A}((t))} I: 
\overline{E}_a(R, t)\longrightarrow \overline{\Inn}_a(A, t)}$, 
${A=\End_{\mathbb{F}}(R)}$, where for any $\mathbb{F}$--algebra $S$ 
\[
\Inn_a(S, t)\ =\ \langle \exp(t D)\mid D\in \Inder_a(S)\rangle\,,\quad 
\overline{\Inn}_a(S, t)\ =\ \langle \exp(t D)\mid D\in \Inder_{sa}(S)\rangle\,,
\]
${}_{\overline{A}((t))} I$ is an embedding if ${\Ann R=\{0\}}$ or (and) 
${R=R^2}$ ($R$ can be replaced by $\overline{R}((t))$)
(the observations before Theorem 3.4). 

The algebras $R$ and $\overline{R}((t))$ satisfy the same homogeneous 
identities over the field $\mathbb{F}$. If $A$ is a semiprime associative 
$\mathbb{F}$--algebra, then the associative algebra $\overline{A}((t))$ is 
semiprime and in the case of a $PI$--algebra $A$, $\overline{A}((t))$ and 
${M(\overline{A}((t)))\cong \overline{M(A)}((t))}$ are semiprime 
$PI$--algebras with 
${\pideg \overline{M(A)}((t))=\pideg M(\overline{A}((t)))=
\pideg M(A)=(\pideg A)^2}$ 
(the observations before and above Theorem 3.4 taking into account 
${\prr=\prr_N=\prr_S=\prr_M}$ for associative algebras). 

\begin{teor}
Let $R$ be a $\mathbb{F}$--algebra, ${\tilde{C}=\overline{C}((t))}$ for 
${C=B_i, A_i}$, ${i\geq 0}$ from Theorem 3.4. Then  
\begin{multline*}
H_0\ =\ G_a(R, B_0, t)\subseteq 
H_1\ =\ 
\pi_{\tilde{B_0}}^{-1} {}_{\tilde{A_0}} I^{-1}(G_a(A_0, B_1, t))
\subseteq \ldots \subseteq
\\ 
H_{i+1}\ =\ 
\pi_{\tilde{B_0}}^{-1} {}_{\tilde{A_0}} I^{-1} 
\pi_{\tilde{B_1}}^{-1} {}_{\tilde{A_1}} I^{-1}\cdots 
\pi_{\tilde{B_i}}^{-1} {}_{\tilde{A_i}} I^{-1}(G_a(A_i, B_{i+1}, t))
\subseteq \ldots \subseteq \prr(G_a(R, t))\,.
\end{multline*}
If ${k\geq 0}$, $A_k/B_{k+1}$ is a $PI$--algebra 
(${\sup\limits_{R\ne P\in \Spec_{\prr_M}(R)} \dim_{\CM(R/P)} P(R/P)< \infty}$), 
$H$ is a subgroup of $G_a(R, t)$, then ${\prr(H)=T(H)}$ is a solvable 
extension of ${H\cap H_{k+1}}$ (${H\cap H_0}$).
\end{teor}

\begin{proof}
Similar to Theorem 3.4 for ${A_{-1}=R}$ and all ${i\geq 0}$ 
\[
\pi_{\tilde{B_i}}\,:\ G_a(A_{i-1}, t)\longrightarrow G_a(A_{i-1}/B_i, t)\,,\quad 
\Ker \pi_{\tilde{B_i}}\ =\ G_a(A_{i-1}, B_i, t)\subseteq \prr(G_a(A_{i-1}, t))
\]
and ${{}_{\tilde{A_i}} I: G_a(A_{i-1}/B_i, t)\hookrightarrow \Inn_a(A_i, t)}$ 
(see above), 
\begin{gather*}
{}_{\tilde{A_i}} I^{-1}(G_a(A_i, B_{i+1}, t))\subseteq 
{}_{\tilde{A_i}} I^{-1}(\prr({}_{\tilde{A_i}} I(G_a(A_{i-1}/B_i, t))))\ =\ 
\prr(G_a(A_{i-1}/B_i, t))\,,
\\
\pi_{\tilde{B_i}}^{-1}(\prr(G_a(A_{i-1}/B_i, t)))
\subseteq
\pi_{\tilde{B_i}}^{-1}(\prr(\pi_{\tilde{B_i}}(G_a(A_{i-1}, t))))\ =\ 
\prr(G_a(A_{i-1}, t))\,,
\\
\pi_{\tilde{B_i}}^{-1} {}_{\tilde{A_i}} I^{-1}(G_a(A_i, B_{i+1}, t))\subseteq 
\pi_{\tilde{B_i}}^{-1} {}_{\tilde{A_i}} I^{-1}(\prr(G_a(A_i, t)))\subseteq 
\prr(G_a(A_{i-1}, t))\,,
\\
H_{i+1}\subseteq 
\pi_{\tilde{B_0}}^{-1} {}_{\tilde{A_0}} I^{-1}\cdots 
\pi_{\tilde{B_i}}^{-1} {}_{\tilde{A_i}} I^{-1}(\prr(G_a(A_i, t)))\subseteq 
\prr(G_a(R, t))\,.
\end{gather*}
If there exists ${k\geq 0}$ such that $A_k/B_{k+1}$ is a $PI$--algebra, $H$ is a 
subgroup of $G_a(R, t)$, then ${B_{k+j}=0}$ and ${H_{k+1}=H_{k+j}}$ for all 
${j\geq 2}$. Since for any ${A_k\ne P\in \Spec(A_k)}$ the group  
${G_a(A_k/P, t)=\overline{G}_a(A_k/P, t)}$ is embeddable in the group 
$GL_{n_P}(\CM(\overline{A_k/P}((t))))$, where 
\[
n_P\ =\ \dim_{\CM(\overline{A_k/P}((t)))} P(\overline{A_k/P}((t)))\leq 
m\ =\ (\pideg A_k)^2\,,
\]
$T(H_P)$ is a solvable radical of the group  
${H_P=\pi_{\overline{P}((t))} {}_{\tilde{A_k}} I \pi_{\tilde{B_k}} 
\cdots {}_{\tilde{A_0}} I \pi_{\tilde{B_0}}(H)\subseteq G_a(A_k/P, t)}$ 
of solvability degree at most ${l(n_P)\leq l(m)}$, and 
\begin{multline*}
{}_{\tilde{A_k}} I \pi_{\tilde{B_k}} \cdots {}_{\tilde{A_0}} I 
\pi_{\tilde{B_0}}(T(H)^{(l(m))})\subseteq 
\bigcap_{A_k\ne P\in \Spec(A_k)} \pi_{\overline{P}((t))}^{-1}(T(H_P)^{(l(m))})
\subseteq 
\\
\bigcap_{A_k\ne P\in \Spec(A_k)} G_a(A_k, P, t)\ =\ G_a(A_k, B_{k+1}, t)\,,
\end{multline*}
${T(H)^{(l(m))}\subseteq H\cap H_{k+1}\subseteq \prr(H)=T(H)}$ 
(the observation before Theorem 3.6).

When ${\sup\limits_{R\ne P\in \Spec_{\prr_M}(R)} 
\dim_{\CM(R/P)} P(R/P)=n< \infty}$ $A_0'$ satisfies ${\st_{2 n}=0}$, 
${H_1=H_j}$ for all ${j\geq 2}$, similar reasoning without going to 
${G_a(A_0, t)}$ gives ${T(H)^{(l(n))}\subseteq H\cap H_0}$, 
$\prr(H)=T(H)$, in the presence of ${l\geq 1}$, 
${(M^R(B_0))_{M(R)}^l=\{0\}}$, 
${(M^{\overline{R}((t))}(B_0))_{M(\overline{R}((t)))}^l=\{0\}}$ and   
$T(H)$ is a solvable radical of $H$ of solvability degree at most 
${l+l(n)+1}$.
\end{proof}

If $R$ is a $m.s.p.$--algebra, then   
${{}_{\overline{M(R)}((t))} I: \overline{E}_a(R, t)\hookrightarrow 
\overline{\Inn}_a(Q^s_m(M(R)), t)_{\overline{M(R)}((t))}}$
and for $R$ with a $SGPI$--algebra $M(R)$,
${{}_{\overline{M(R)}((t))} I: E_{c, a}(R, t)\hookrightarrow 
\Inn_a(Q^s_m(M(R)), t)_{\overline{M(R)}((t))}}$, where  
\begin{multline*}
E_{c, a}(R, t)\ =\ 
\langle \exp(t D)\mid D\in \Der_a(R),\ [\overline{D}, \CM(R)]=\{0\}\rangle\,,
\\
\shoveleft{
\Inn_a(Q^s_m(M(R)), t)_{\overline{M(R)}((t))}\ =}
\\ 
\shoveright{
\langle \exp(t D)|_{\overline{M(R)}((t))}\mid D\in \Inder(Q^s_m(M(R))),\ 
(M(R))D\subseteq M(R),\ D|_{M(R)}\in \Der_a(M(R))\rangle\,,}
\\
\shoveleft{
\overline{\Inn}_a(Q^s_m(M(R)), t)_{\overline{M(R)}((t))}\ =}
\\ 
\{\phi|_{\overline{M(R)}((t))}\mid \phi\in \langle \exp(t D)\mid D\in 
\Inder_{sa, \mathbb{F}}(Q^s_m(M(R))),\ (M(R))D\subseteq M(R)\rangle\}
\end{multline*}
(Remark 2.9, Corollary 2.11 and the observations before it, Corollary 2.13). 
For $R$ with a $PI$--algebra $M(R)$ ${Q(M(R))=\hat{M(R)}}$, 
\begin{multline*}
\overline{\Inn}_a(\hat{M(R)}, t)_{\overline{M(R)}((t))}\cong 
\langle \exp(t D)\mid D\in \Inder_{sa, \mathbb{F}}(\hat{M(R)}),\ (M(R))D\subseteq M(R)\rangle 
\subseteq 
\\
U(\overline{M(\hat{M(R)})}((t)))\cong 
U(M(\overline{\hat{M(R)}}((t))))\,,
\end{multline*}
${\pideg \overline{M(\hat{M(R)})}((t))=\pideg M(\hat{M(R)})= 
(\pideg M(R))^2}$. If, in addition, $\hat{M(R)}$ is a finite subdirect 
product of $\hat{M(R)}/P_i$, ${\hat{M(R)}\ne P_i\in \Spec(\hat{M(R)})}$, 
${i=1, \ldots, k}$, then from $D\in \Mder(\hat{M(R)})$, 
$(M(R))D\subseteq M(R), D|_{M(R)}\in \Der_a(M(R))$ it follows that 
${D\in \Mder_{sa, \mathbb{F}}(\hat{M(R)})}$ and 
${\Inn_a(\hat{M(R)}, t)_{\overline{M(R)}((t))}=
\overline{\Inn}_a(\hat{M(R)}, t)_{\overline{M(R)}((t))}}$, 
${E_{c, a}(R, t)=\overline{E}_a(R, t)}$. Indeed, in view of 
${\dim_{{}_{\hat{M(R)}/P_i} C} \widehat{\hat{M(R)}/P_i}< \infty}$, there 
exists ${f_i(t)\in \mathbb{F}[t]}$, ${((\hat{M(R)}+P_i)/P_i)f_i(D_i)=\{0\}}$ 
for ${D_i\in \Mder(\hat{M(R)}/P_i)}$ induced by $D$, ${i=1, \ldots k}$, 
${(\hat{M(R)})f(D)=\{0\}}$, ${f(t)=\prod\limits_{i=1}^k f_i(t)}$, and it 
only remains to apply the observations after Remark 2.10. Similar to 
Corollary 3.5, we get from here 

\begin{co}
If $R$ is a $m.s.p.$--algebra with a $PI$--algebra $M(R)$, then 
any subgroup $G\subseteq \overline{E}_a(R, t)$ has a solvable 
radical ${\prr(G)=T(G)}$ of solvability degree at most 
$n(\overline{\hat{M(R)}}((t)))\leq l((\pideg M(R))^2)$.
\end{co}

If $A$ is a semiprime associative $\mathbb{F}$--algebra, then one has   
${\overline{E}_a(A, t)\subseteq 
\overline{\Inn}_a(Q^s_m(A), t)_{\overline{A}((t))}}$, 
${E_{c, a}(A, t)\subseteq 
\Inn_a(Q^s_m(A), t)_{\overline{A}((t))}}$ for the $SGPI$--algebra $A$, where 
\begin{gather*}
\Inn_a(Q^s_m(A), t)_{\overline{A}((t))}\ =\ 
\langle \exp(t D)|_{\overline{A}((t))}\mid D\in \Inder(Q^s_m(A)),\ 
(A)D\subseteq A,\ D|_A\in \Der_a(A)\rangle\,,
\\
\overline{\Inn}_a(Q^s_m(A), t)_{\overline{A}((t))}\ =\ 
\{\phi|_{\overline{A}((t))}\mid \phi\in \langle \exp(t D)\mid 
D\in \Inder_{sa, \mathbb{F}}(Q^s_m(A)),\ (A)D\subseteq A\rangle\}\,,
\end{gather*}
and in the case of the $PI$--algebra $A$, ${\overline{E}_a(A, t)\hookrightarrow 
U(\overline{M(\hat{A})}((t)))\cong U(M(\overline{\hat{A}}((t))))}$, any 
subgroup ${G\subseteq \overline{E}_a(A, t)}$ has a solvable radical 
${\prr(G)=T(G)}$ of solvability degree at most 
${n(\overline{\hat{A}}((t)))\leq l((\pideg A)^2)}$, 
${E_{c, a}(A, t)=\overline{E}_a(A, t)}$ for the $PI$--algebra $A$ with 
$\hat{A}$ representable as a finite subdirect product of its prime quotient 
algebras (the observations before Corollary 3.7 with $M(R)$ replaced by $A$). 
If the field $\mathbb{F}$ is algebraically closed, $A$ is a $PI$--algebra and 
a finite subdirect product of ${A/P_i, A\ne P_i\in \Spec(A), 
i=1, \ldots, k}$, then ${\Der_{1, sa}(A)=\Der_{sa}(A)}$, 
${\overline{E}_a(A, t)=\langle \exp(t D)\mid D\in \Der_{1, sa}(A)\rangle}$.
It is sufficient to note that for any ${D\in \Der_{1, sa}(A)}$, 
${i=1, \ldots, k}$, ${(P_i)D\subseteq P_i}$, 
${(Z(A/P_i))D_i=(Z(\hat{A/P_i}))D_i=\{0\}}$, 
where $D_i$ is the extension to $\hat{A/P_i}$ of ${D_i\in \Der(A/P)}$ induced 
by $D$ (the observations after Corollary 3.5, Corollary 2.3, 
Lemma 2.8), ${\dim_{Z(\hat{A/P_i})} \hat{A/P_i}< \infty}$ and  
${(\hat{A/P_i})f_i(D_i)=\{0\}}$ for some ${f_i(t)\in \mathbb{F}[t]}$, 
$(A)f(D)=\{0\}$, ${f(t)=\prod\limits_{i=1}^k f_i(t)}$.

\section*{Addition 1. Speciality of the Slin'ko --- McCrimmon radical}

\emph{Absolute zero divisors} of quadratic Jordan and right-alternative 
algebras are elements whose quadratic multiplication operators are equal 
to zero; algebras from these classes that do not have non-zero absolute 
zero divisors are \emph{non-degenerate algebras}; radicals on these classes 
with non-degenerate semisimple algebras are \emph{non-degenerate radicals}. 
The non-degenerate radical with the largest class of semisimple algebras 
(the upper radical determined by the class of non-degenerate algebras) is the 
McCrimmon and Slin'ko --- McCrimmon radical of Jordan and right-alternative 
algebras, $Mc$. To each non-degenerate radical $\mathcal{T}$ of quadratic 
Jordan $F$--algebras there corresponds a non-degenerate radical of 
right-alternative $F$--algebras $\mathcal{T}^{+}$, 
${\mathcal{T}^{+}(R)=\mathcal{T}(R^+)}$ for any right-alternative $F$--algebra 
$R$, where ${R^{+}=(R, U, {}^2)}$ is the quadratic Jordan algebras of $R$ with   
${U: x\longmapsto U_x}$, ${y U_x=(x y) x}$, and ${{}^2: x\longmapsto x^2}$, 
${x, y\in R}$ \cite[Theorem~5]{AT2}. In particular, ${Mc(R)=Mc^{+}(R)}$ 
is the smallest among the ideals of $R$ and $R^{+}$, the quotient algebras by 
which are non-degenerate. The speciality of $\mathcal{T}$ (the decomposabi\-lity 
of $\mathcal{T}$--semisimple algebras into subdirect products of prime 
$\mathcal{T}$--semisimple algebras ($\mathcal{T}$--semisimple algebras with 
non-zero intersections of any two non-zero ideals)) is inherited by 
$\mathcal{T}^{+}$ \cite[Theorem~1]{AT2} (the primeness of quadratic Jordan 
algebras is the absence of ideals $I$, $J$, ${I\cap J=\{0\}\ne I, J}$, and 
ideals ${I\ne \{0\}}$, ${(I)U_x=\{0\}}$, ${x^2=0}$ for all ${x\in I}$). 
If ${1/2\in F}$, then the radical $Mc$ of Jordan (linear and their quadratic 
representations) and right-alternative $F$--algebras are special, since 
$Mc(R)$ of any such algebra $R$ is equal to the set of its \emph{strict absolute 
zero divisors}
\[
SA(R)\ =\ 
\{x\in R\mid \forall \{x_i\}_{i=0}^{\infty},\ x_0=x,\ x_{i+1}\in (R)U_{x_i},\ 
\exists k\geq 0,\ x_k=0\}
\]
\cite{ZelM} (the conclusion of the speciality of $Mc$ from ${Mc(R)=SA(R)}$ 
is similar to the elementwise description of $\prr$).

\begin{rem}
The radical $Mc$ is special on the class of alternative algebras over any 
ring $F$, ${Mc(R)=SA(R)}$ for any alternative $F$--algebra $R$.
\end{rem}

\begin{proof} 
Since non-degenerate algebras do not contain non-zero strict absolute zero 
divisors, ${SA(R)\subseteq Mc(R)}$ in any case.

If ${x\in R\setminus SA(R)}$, $\{x_i\}_{i=0}^{\infty}$ is a chain of 
${0\ne x_{i+1}\in U_{x_i}(R)}$, ${i\geq 0}$, with ${x_0=x}$, $P$ is the 
maximal among ${I\lhd R}$, ${I\cap \{x_i\}_{i=0}^{\infty}=\emptyset}$, 
then ${P\in \Spec(R)}$ and when  
${Mc(R/P)\ne \{0\}}$ ${3 R/P=\{0\}}$ 
\cite[Theorem~9~and~corollary,~p.~229,~230]{ZShS}, $R/P$ can be considered as an algebra over the 
integrity domain $F/\Ann_F R$ with $1/2$, 
${x_i+P\in Mc(R/P)=SA(R/P)}$ (the choice of $P$), ${x_j\in P}$ for some 
${i\geq 0}$, ${j\geq i}$, a contradiction. Consequently, ${Mc(R/P)=\{0\}}$, 
$R/P$ is either associative or a Cayley --- Dickson ring (by \cite[Theorem~9,~p.~229]{ZShS} again), 
${R\setminus SA(R)\subseteq R\setminus \bigcap\limits_{P\in \Spec_{Mc}(R)} P}$ 
and ${Mc(R)=\bigcap\limits_{P\in \Spec_{Mc}(R)} P=SA(R)}$. 
\end{proof} 

The prime radical $\prr(R)$ of any alternative algebras $R$ is equal to its 
lower nil-radical $RN(R)$, ${\prr(R)=RN(R)}$, and if $R/\prr(R)$ is either 
associative or (and) has no $3$--torsion or (and) is finitely generated, 
then ${\prr(R)=Mc(R)}$ \cite[Corollary~2,~p.~215,~Corollary\linebreak 
of~Theorem~9,~p.~230,~Theorem~5',~Remark~to~it,~p.~222,~223]{ZShS}. Due to 
\cite{Pch, Pch1} and other well-known results on exceptional 
algebras, there exist alternative $PI$--algebras $R$ with $\prr(R)\ne Mc(R)$. 
Remark 3.8, the description of the prime non-degenerate alternative algebras 
and the coincidence of nil-radicals of associative $PI$--algebras 
\cite[Corollary~of~Proposition~3,~Theorem 9,~p.~229]{ZShS}, 
\cite[Theorem~3,~p.~86]{AR}, \cite[Corollary~1.6.26,~Theorem~ 1.6.36,~p.~46,~48]{Row} allows us to deduce

\begin{co}
${Mc(R)=\mathcal{N}(R)}$ for any alternative $PI$--algebra $R$. 
\end{co}

This is also true for any right-alternative $PI$--algebra $R$ over a ring $F$ 
with $1/2$, since non-degenerate right-alternative algebras over such $F$ are 
alternative \cite[Theorem~1]{Skos}. In addition, 
${Mc(R)=M(R)=T^{+}(R)=\mathcal{N}(R)}$, where ${M=LN^{+}}$ is the Mikheev 
radical (locally right-nilpotent in the sense of Shirshov radical), 
${\mathcal{N}=\mathcal{N}^{+}}$ is the upper nil-radical of right-alternative 
algebras \cite{AT2} taking into account \cite{ACZ} and 
\cite[Proposition~2]{Skos1}, the weakly solvable radical $T$ of quadratic Jordan algebras 
is introduced for the operation ${x\cdot y=(x+y)^2-x^2-y^2}$ ($T^{+}$ is the 
weakly solvable radical defined for the operation 
${x\cdot y=x y+y x}$, ${T(R)\subseteq T^{+}(R)}$), ${M=LN=T}$ on the class 
of alternative algebras over any ring 
\cite[Theorem~5,~p.~143]{ZShS} and\linebreak 
\cite[the~proof~of~Theorem~7]{Sh} 
(${LN=T}$ follows from its validity for associative algebras, the 
nil-semisimplicity of Cayley --- Dickson rings and the decomposability  
of $LN$--semisimple alternative algebras into subdirect products of prime 
$LN$--semisimple associative algebras and Cayley --- Dickson rings, 
\cite[Theorem~8,~Corollary~of~Proposition~3,~Theorem~11,~p.~197,~229,~232]{ZShS}).

Remark 3.8 and Corollary 3.9 correspond in \cite{B7} to Corollary 3.1.16 of 
Theorem 3.1.15 on the non-degeneracy of prime quotient algebras of 
non-degenerate alternative algebras and to Lemma 3.3.1, but the conclusion 
of Remark 3.8 is different from \cite{B7}.  

\section*{Addition 2. Non-radicality of the prime radical of Lie 
algebras and groups}

The justification of the non-radicality of the prime radical $\prr$ on a class 
of algebras (groups) $\mathfrak{M}$ closed under taking ideals (normal 
subgroups) and homomorphic images is reduced to the construction of an example 
of ${R\in \mathfrak{M}}$ with ${I=\prr(I)\lhd R}$, ${I\not\subseteq \prr(R)}$.  

Let $\mathbb{F}_p$ be a field, ${\Ch \mathbb{F}_p=p> 0}$, 
${X=\{x_i\}_{i=1}^{\infty}}$, ${I_p=(x_i^p\mid i\geq 1)_{\mathbb{F}_p[X]}}$, 
${A_p=\mathbb{F}_p[X]/I_p}$ and ${\overline{f}=f+I_p}$, ${f\in \mathbb{F}_p[X]}$. 
Then ${\prr(A_p)=\overline{X} A_p}$ is locally nilpotent, not nilpotent and 
satisfies ${x^p=0}$, 
\begin{gather*}
\prr(M_n(A_p))\ =\ M_n(\prr(A_p))\ =\ \overline{X} M_n(A_p)\,,
\\
M_n(\prr(A_p))^{(-)}\ =\ \prr(M_n(\prr(A_p))^{(-)})\quad (n\geq 1)
\end{gather*}
(${\overline{x_i} M_n(A_p)\lhd M_n(A_p), (\overline{x_i} M_n(A_p))^p=\{0\}}$). 
As is known, any set ${\{f_i\}_{i=1}^{\infty}\subseteq \mathbb{F}[X]}$ 
corresponds to ${D\in \Der(\mathbb{F}_p[X])}$ with ${D=f_i, 
f D=\sum\limits_{j=1}^k f_j (f D_j), f\in \mathbb{F}_p[x_1, \ldots, x_k], 
D_i\in \Der(\mathbb{F}_p[X])}$ is the partial derivation with respect to $x_i$, 
${x_j D_i=\delta_{ij}}$, ${i\geq 1}$, and all ${D\in \Der(\mathbb{F}_p[X])}$ have 
this form. 

Let ${L_p=M_3(\prr(A_p))^{(-)}\rtimes W_p}$ be a split extension of a 
Lie algebra $M_3(\prr(A_p))^{(-)}$ by an Abelian Lie algebra 
${W_p=\sum\limits_{q\geq 0} \mathbb{F}_p \overline{H}_q\subset 
\Der(M_3(A_p))}$, where ${H_q\in \Der(\mathbb{F}_p[X])}$, ${x_i H_q=x_{i+q}}$, 
${i\geq 1}$, and ${\overline{H}_q\in \Der_{\mathbb{F}_p}(M_3(A_p))}$, 
${(\overline{f_{st}}) \overline{H}_q=
(\overline{f_{st} H_q})}$, ${f_{st}\in \mathbb{F}_p[X]}$, ${1\leq s, t\leq 3}$  
(${I_p\lhd_{\Der} \mathbb{F}_p[X]}$; $\overline{H}_q$ is the extension 
of ${\overline{H}_q\in \Der(A_p)}$ induced by $H_q$ onto $M_3(A_p)$, 
${\overline{f}\, \overline{H}_q=\overline{f H_q}}$, ${f\in \mathbb{F}_p[X]}$).

\begin{rem}
For any ${m\geq 0}$ there is ${n\geq 1}$, ${\binom{m+n}n\not\equiv 0\pod p}$.
\end{rem}

\begin{proof}
The case ${m=0}$ is obvious. If ${m> 0}$ and 
${\binom{k}{k-m}=\binom{k}m\equiv 0\pod p}$ for all ${k> m}$, then 
${\binom{t+1}m-\binom{t}m=\binom{t}{m-1}\equiv 0\pod p}$ for all ${t> m}$ 
and therefore, ${m> 1}$ (${\binom{t}0=1}$), 
$\binom{t+1}{m-1}-\binom{t}{m-1}=\binom{t}{m-2}\equiv 0\pod p$ for all 
${t> m}$, ${m> 2}$ and so on, in the end ${m> m}$, a contradiction. 
\end{proof}

\begin{lemma}
For any ${0\ne f\in \prr(A_p)}$ there are ${s, t\geq 1}$, 
${f \overline{H}_s^t f\ne 0}$.
\end{lemma}

\begin{proof}
The classes $\overline{1}$ and ${\prod\limits_{j=1}^k \overline{x_{i_j}}^{l_j}}$,
${k\geq 1}$, ${i_1< \ldots < i_k}$, ${1\leq l_j< p}$, form the 
$\mathbb{F}_p$--basis of the algebra 
${A_p=\bigoplus\limits_{d\geq 0} A_{p, d}}$, the subspaces 
\[ 
A_{p, d}\ =\ \sum\limits_{0\leq l_i< p,\ \sum\limits_{i\geq 1} l_i=d} 
\mathbb{F}_p \prod_{i\geq 1} \overline{x_i}^{l_i}\quad (d\geq 0)\,,
\]
are $W_p$--invariant. In view of ${f=f_1+\ldots+f_l}$ for some uniquely defined 
${0\ne f_i\in A_{p, d_i}}$ and ${1\leq d_1< \ldots< d_l}$, it is sufficient 
to choose ${s, t\geq 1}$, ${f_1 \overline{H}_s^t f_1\ne 0}$. 

By Remark 3.10, for any ${j\geq 1}$ there exist ${k_i\geq 1}$, 
${\binom{1+k_1+\ldots+k_i}{k_i}\not \equiv 0\pod p}$, ${i=1, \ldots, j}$, 
${\prod\limits_{i=1}^j \binom{1+k_1+\ldots+k_i}{k_i}=
\frac{(1+k_1+\ldots+k_j)!}{k_1!\cdots k_j!}\not \equiv 0\pod p}$.
We set ${t=1}$ for ${d_1=1}$, ${t=1+k_1+\ldots+k_{d_1-1}}$ for ${d_1> 1}$, 
${s=n+1}$ assuming  
${f_1\in \langle \overline{x_1}, \ldots, \overline{x_n}\rangle}$, 
${M=Z_{p-1, d_1}^n}$ and ${T=Z_{t, t}^{d_1}}$, where 
\[
Z_{k, l}^n\ =\ \{\tilde{z}=(z_1, \ldots, z_n)\mid 0\leq z_i\leq k,\ 
z_1+\ldots+z_n=l\}\quad (n\geq 1,\ k, l\geq 0)\,.
\] 

For any ${\tilde{m}=(m_1, \ldots, m_n)\in M}$ and  
${\tilde{t}=(t_1, \ldots, t_{d_1})\in T}$, we set 
${\{j_k\}_{k=1}^{k(\tilde{m})}=\{i\mid m_i\ne 0\}}$, 
${1\leq j_1< j_2< \ldots< j_{k(\tilde{m})}\leq n}$, 
${n_0=0}$, ${n_{k+1}=n_k+m_{j_{k+1}}}$, 
${k=0, \ldots, k(\tilde{m})-1}$ (${n_{k(\tilde{m})}=d_1}$), 
${\{t_{n_{k-1}+i}\}_{i=1}^{m_{j_k}}=\{t'_{n_{k-1}+i}\}_{i=1}^{m_{j_k}}}$, 
${t'_{n_{k-1}+1}\leq t'_{n_{k-1}+2}\leq \ldots \leq t'_{n_k}}$, 
${k=1, \ldots, k(\tilde{m})}$, and form 
$\tilde{t}'=\tilde{t}'(\tilde{m}, \tilde{t})=(t'_1, \ldots, t'_{d_1})$, 
${T(\tilde{m}, \tilde{t}')= 
\{\tilde{d}\in T\mid \tilde{d}'(\tilde{m}, \tilde{d})=\tilde{t}'\}}$,
${T'(\tilde{m})=\{\tilde{t}'(\tilde{m}, \tilde{t})\mid \tilde{t}\in T\}}$, 
where 
\[
|T(\tilde{m}, \tilde{t}')|\ =\ \prod_{k=1}^{k(\tilde{m})} 
\frac{m_{j_k}!}{l_{1k}!\cdots l_{l_k k}!}
\not\equiv 0\pod p
\]
for ${l_k=|\{t'_i\}_{i=n_{k-1}+1}^{n_k}|}$, 
${\{t'_{jk}\}_{j=1}^{l_k}=\{t'_i\}_{i=n_{k-1}+1}^{n_k}}$, 
${l_{jk}=|\{i\mid t'_i=t'_{jk},\ i=n_{k-1}+1, \ldots, n_k\}|}$, 
\[
g(\tilde{m})\ =\ \prod_{i=1}^n \overline{x_i}^{m_i}\ =\ 
\prod_{k=1}^{k(\tilde{m})} \overline{x_{i_{j_k}}}^{m_{j_k}}\,,  
\quad 
g(\tilde{m}, \tilde{t})\ =\ \prod_{k=1}^{k(\tilde{m})} 
\prod_{i=1}^{m_{j_k}} \overline{x_{i_{j_k}+s t_{n_{k-1}+i}}}\,.
\]
Since ${g(\tilde{m}, \tilde{t})=g(\tilde{m}, \tilde{d})}$ for 
${\tilde{t}, \tilde{d}\in T}$ is equivalent to 
${\tilde{t}'(\tilde{m}, \tilde{t})=\tilde{t}'(\tilde{m}, \tilde{d})}$ 
and ${i+s q=i'+s q'}$ for ${1\leq i, i'\leq n, q, q'\geq 0}$ is equivalent 
to ${i=i'}$, ${q=q'}$, ${g(\tilde{u}, \tilde{t})=g(\tilde{v}, \tilde{d})}$ 
for ${\tilde{u}, \tilde{v}\in M}$, 
${\tilde{t}, \tilde{d}\in T}$ if and only if ${\tilde{u}=\tilde{v}}$, 
${\tilde{t}'(\tilde{u}, \tilde{t})=\tilde{t}'(\tilde{v}, \tilde{d})}$. 
Then ${f_1=\sum\limits_{\tilde{m}\in M} b_{\tilde{m}} g(\tilde{m})}$, 
${\{0\}\ne \{b_{\tilde{m}}\}\subset \mathbb{F}_p}$, 
\begin{multline*}
f_1 \overline{H}_s^t\ =\ \sum_{\tilde{m}\in M} b_{\tilde{m}}
\biggl(\sum_{\tilde{t}=(t_1, \ldots, t_{d_1})\in T} 
\frac{t!}{t_1!\cdots t_{d_1}!} g(\tilde{m}, \tilde{t})\biggr)\ =
\\ 
\sum_{\tilde{m}\in M,\ \tilde{t}'=(t'_1, \ldots, t'_{d_1})\in T'(\tilde{m})} 
\frac{t!}{t'_1!\cdots t'_{d_1}!}
|T(\tilde{m}, \tilde{t}')| b_{\tilde{m}} g(\tilde{m}, \tilde{t}')\ =\ 
\sum_{\tilde{m}\in M,\ \tilde{t}'\in T'(\tilde{m})} c_{\tilde{m} \tilde{t}'}
g(\tilde{m}, \tilde{t}')\ne 0\,,
\end{multline*} 
since ${g(\tilde{u}, \tilde{d}')}$ for ${\tilde{u}\in M}$, 
${b_{\tilde{u}}\ne 0}$, and ${\tilde{d}'=\tilde{t}'(\tilde{u}, \tilde{d})}$, 
${\tilde{d}=(1, k_1, \ldots, k_{d_1-1})\in T}$, is included in the record 
of $f_1 \overline{H}_s^t$ with the coefficient 
${c_{\tilde{u}, \tilde{d}'}=\frac{t!}{1! k_1!\cdots k_{d_1-1}!} 
|T(\tilde{u}, \tilde{d}')| b_{\tilde{u}}\ne 0}$. It remains to be noted that 
${g(\tilde{u}) g(\tilde{u}, \tilde{d}')=g(\tilde{m}) g(\tilde{v}, \tilde{t}')}$ 
for ${\tilde{m}, \tilde{v}\in M}$, ${\tilde{t}'\in T'(\tilde{v})}$ is 
equivalent to ${\tilde{u}=\tilde{m}=\tilde{v}}$, ${\tilde{d}'=\tilde{t}'}$ 
(${k_i\geq 1}$), 
\[
f_1 \overline{H}_s^t f_1\ =\ 
b_{\tilde{u}} c_{\tilde{u}, \tilde{d}'} 
g(\tilde{u}) g(\tilde{u}, \tilde{d}')+
\sum_{\genfrac{}{}{0pt}{1}
{\tilde{v}, \tilde{m}\in M,\ \tilde{t}'\in T'(\tilde{m}),}
{(\tilde{v}, \tilde{m}, \tilde{t}')\ne (\tilde{u}, \tilde{u}, \tilde{d}')}}
b_{\tilde{v}} c_{\tilde{m}, \tilde{t}'} g(\tilde{v}) g(\tilde{m}, \tilde{t}')
\ne 0\,.
\]
\end{proof}

If $A$ is an associative commutative $F$--algebra, ${\Ann A=\Ann_t A=\{0\}}$, 
${t=l, r}$, ${n\geq 1}$, then  
${A\cong Z(M_n(A))=\{\diag(a, \ldots, a)\mid a\in A\}\cong A}$ 
(${a\longmapsto \diag(a, \ldots, a)}$, ${a\in A}$) and there is a homomorphism 
${\psi_n: \Der_F(M_n(A))\longrightarrow \Der(A)}$, 
\begin{gather*}
\diag(a (\psi_n D), \ldots, a (\psi_n D))\ =\ \diag(a, \ldots, a) D\quad 
(D\in \Der_F(M_n(A)),\ a\in A)\,,
\\
\Ker \psi_n\ =\ \{D\in \Der_F(M_n(A))\mid (Z(M_n(A)))D=\{0\}\}\,,
\end{gather*}
${\Ker \psi_n\cong \Der_A(M_n(A))=\Inder(M_n(A))}$ for $A$ with 1.

\begin{lemma}
The Lie algebra $L_p$ is weakly solvable, but not locally solvable, 
\[
M_3(\prr(A_p))^{(-)}\ =\ \prr(M_3(\prr(A_p))^{(-)})\not\subseteq 
\prr(L_p)\ =\ Z(M_3(\prr(A_p)))\,,
\]
\end{lemma}

\begin{proof}
In this case ${\Ann \prr(A_p)=\{0\}}$, ${Z(M_3(\prr(A_p)))=\prr(A_p) E}$, 
$E=\diag(1, 1, 1)=E_{11}+E_{22}+E_{33}$, ${E_{ij}\in M_3(A_p)}$, and  
${\bigcap\limits_{q\geq 1} \Ker \overline{H}_q=M_3(\mathbb{F}_p)}$ 
(Lemma 3.11). Let $(b_{ij})+\overline{H}\in \prr(L_p)\setminus \prr(A_p) E$ 
for ${(b_{ij})\in M_3(\prr(A_p))}$, ${\overline{H}\in W}$, 
${H\in \sum\limits_{q\geq 0} \mathbb{F}_p H_q}$. For 
${(b_{ij})\in \prr(A_p) E}$ we have ${[E_{12} \overline{x_1}, (b_{ij})+\overline{H}]=
E_{12} \overline{x_1 H}\ne 0}$, for ${(b_{ij})\notin \prr(A_p) E}$ there exists 
${\overline{D}\in W}$, $0\ne [(b_{ij})+\overline{H}, \overline{D}]=
(b_{ij}) \overline{D}\in M_3(\prr(A_p))\setminus \prr(A_p) E$ and, as a 
consequence, we can move on to 
${(c_{ij})\in (\prr(L_p)\cap M_3(\prr(A_p)))\setminus \prr(A_p) E}$, 
${[a E_{st}, (c_{ij})]=(a_{ij})=
\sum\limits_{l=1}^3 (a_{sl} E_{sl}+a_{lt} E_{lt})\ne 0}$ for suitable 
${a\in \prr(A_p)}$, ${1\leq s, t\leq 3}$. If there exists ${k\ne s, t}$, 
${a_{sk}\ne 0}$ (${a_{kt}\ne 0}$), one can find ${d, b\in \prr(A_p)}$, 
${[d E_{kk}, (a_{ij})]=d a_{kt} E_{kt}-d a_{sk} E_{sk}}$ and  
${[b E_{lk}, [d E_{kk}, (a_{ij})]]=b d a_{kt} E_{lt}\ne 0}$   
(${[b E_{kl}, [d E_{kk}, (a_{ij})]]=b d a_{sk} E_{sl}\ne 0}$) for  
${l=s}$ (${l=t}$) if ${s\ne t}$ and ${l\ne k, s=t}$, otherwise. 
If ${a_{sk}=a_{kt}=0}$ for all ${k\ne s, t}$, then either ${s=t}$, 
${(a_{ij})=a_{ss} E_{ss}}$, or ${s\ne t}$, 
${(a_{ij})=\sum\limits_{i, j=s, t} a_{ij} E_{ij}}$ and  
${[[u E_{kl}, (a_{ij})], v E_{rr}]=u v a_{lr} E_{kr}\ne 0}$ 
for some ${u, v\in \prr(A_p), l, r\in \{s, t\}, 
\{l, l'\}=\{s, t\}}$ and any ${k\ne s, t}$. Hence there exists 
${0\ne z E_{mn}\in \prr(L_p)\cap M_3(\prr(A_p))}$ and without loss 
of generality ${m=1, n=2}$ (the transition to such $m$, $n$ is 
possible, due to ${\Ann \prr(A_p)=\{0\}}$, 
${[a E_{ii}, b E_{ij}]=a b E_{ij}}$, ${[a E_{ij}, b E_{jk}]=a b E_{ik}}$, 
${i\ne j\ne k}$). Lemma 3.11 allows us to construct for ${z_0=z}$ a chain 
$\{z_q E_{12}\}_{q\geq 0}$, 
${z_{q+1}=(z_q \overline{H}_{s_q}^{t_q} z_q) \overline{x}_{i_{q+1}} 
\overline{x}_{j_{q+1}} \overline{x}_{k_{q+1}}\ne 0}$, 
${s_q, t_q, i_{q+1}, j_{q+1}, k_{q+1}\geq 1}$, 
\[
z_{q+1} E_{12}\ =\ [[[[z_q \overline{H}_{s_q}^{t_q} E_{12}, 
\overline{x}_{i_{q+1}} E_{23}], \overline{x}_{j_{q+1}} E_{21}], z_q E_{12}], 
\overline{x}_{k_{q+1}} E_{32}]\,,\quad 
z_q \overline{H}_{s_q}^{t_q} E_{12}\ =\ 
z_q E_{12} \ad_{\overline{H}_{s_q}}^{t_q}\,,
\]
and so, ${z_0 E_{12}\notin \prr(L_p)}$, a contradiction. Consequently, 
${\prr(L_p)=\prr(A_p) E}$. 
           
By construction, the Lie algebra $L_p$ is weakly solvable. If $Z$ is a 
subalgebra of matrices from $M_3(\mathbb{F}_p)^{(-)}$ with zero trace, 
then ${Z=\mathbb{F}_p Y}$, 
\[
Y\ =\ \{E_{11}-E_{22},\ E_{22}-E_{33}\}\cup \{E_{ij}\mid 1\leq i\ne j\leq 3\}\,, 
\]
${Y\subseteq [Y, Y]}$, ${Z=[Z, Z]=Z^{(n)}}$, 
${(\prr(A_p) Z)^{(n)}=\prr(A_p)^{2^n} Z\ne \{0\}}$ for all ${n\geq 0}$ and 
\[
\prr(A_p) Z\ =\ \langle \overline{X} Y\rangle\subset 
\langle \overline{x_1} Y,\ \overline{H}_1\rangle\,. 
\]
Thus, $L_p$ is not locally solvable.
\end{proof}

For ${p> 2}$ the Lie algebra $L_p$ is also strongly degenerate, ${L_p=K(L_p)}$, since 
\begin{gather*}
M_3(\prr(A_p))^{(-)}\ =\ K(M_3(\prr(A_p))^{(-)})\subseteq K(L_p)\,,
\\ 
L_p/M_3(\prr(A_p))^{(-)}\ =\ K(L_p/M_3(\prr(A_p))^{(-)})\cong W_p\ =\ 
K(W_p)\,,
\end{gather*} 
where $K$ is the Kostrikin radical of Lie algebras 
\cite[Lemma~4]{Zel}.

Since algebras over quotient rings of a ring $F$ can be considered as algebras 
over $F$, and ${\mathbb{N} 1\not\subseteq U(F)}$ ($F$ is not a 
$\mathbb{Q}$--algebra) is equivalent to the presence among the quotient rings 
of $F$ of a field of characteristic ${p> 0}$, it immediately follows that 

\begin{teor}
If a ring $F$ is not a $\mathbb{Q}$--algebra, then the prime radical 
$\prr$ is not a radical in the sense of Kurosh --- Amitsur on the 
class of all Lie $F$--algebras.
\end{teor}

Let $\{p_i\}_{i\geq 1}$ be the set of all primes ${p> 0}$, numbered in any 
way, ${\Bbbk_i=\mathbb{F}_{p_i}}$, ${R_i=\prr(L_{p_i})}$, 
${J_i\in \{(M_3(\overline{x_j}^{p_i-1} A_{p_i})^{(-)}+R_i)/R_i\mid j\geq 1\}}$,
\[
\{0\}\ =\ J_i^2\ \ne J_i\lhd V_i\ =\ M_3(\prr(A_{p_i}))^{(-)}/R_i\ =\ 
\prr(V_i)\lhd N_i\ =\ L_{p_i}/R_i\,,
\]
${K=\prod\limits_{i\geq 1} \Bbbk_i}$, ${N=\prod\limits_{i\geq 1} N_i}$, 
${V=\prod\limits_{i\geq 1} V_i}$, ${J=\prod\limits_{i\geq 1} J_i}$ and   
${N_D=N/U_D(N)}$ be the ultraproduct of algebras $\{N_i\}_{i\geq 1}$ 
with respect to the Frechet ultrafilter $D$ over the set of natural numbers 
$\mathbb{N}$, considered as an algebra over the ultraproduct 
${\Bbbk_D=K/U_D(K)}$ of fields $\{\mathbb{F}_i\}_{i\geq 1}$ by $D$, 
${U_D(N)=\{x\in N\mid \{i\geq 1\mid x(i)=0\}\in D\}}$ ($U_D(K)$ is the maximal 
among ${U\lhd K}$, ${\bigoplus\limits_{i\geq 1} \Bbbk_i\subseteq U}$, and 
${N_D\cong N/U_D(K) N}$). Then in the Lie algebra $N_D$ over the field 
$\Bbbk_D$, ${\Ch \Bbbk_D=0}$, 
\[
\{0\}\ =\ J'^2\ne J'\ =\ (J+U_D(N))/U_D(N)\lhd 
V'\ =\ (V+U_D(N))/U_D(N)\lhd N_D\,,
\]
${J'\cong J/U_D(J)}$, ${V'\cong V/U_D(V)}$, 
${\prr(V')\lhd_{\Der} V'}$, ${\prr(V')\lhd N_D}$. Since 
\begin{gather*}
\overline{x_i}^{p-1} y_1\cdots y_{p-2} \overline{x_i}^{p-1}\ =\ 0
\quad (i\geq 1,\ y_j\in A_p\cup \{\overline{H}_q\}_{q\geq 0},\ p> 2)\,,
\\
[(M_3(\overline{x_i}^{p-1} A_p))\ad(L_p)^{p-2}, 
M_3(\overline{x_i}^{p-1} A_p)]\ =\ \{0\}\quad (i\geq 1,\ p\geq 2)\,,
\end{gather*}
${(J)_N^2=\sum\limits_{y, z\in J,\ n\geq 1} y \ad(N)^n \ad_z \Ad(N)}$, 
${y(i) \ad(N_i)^n \ad_{z(i)} \Ad(N_i)=\{0\}}$ for $i$, ${n< p_i-1}$, and 
${(J')_{N_D}^2=\{0\}}$, ${(J')_{N_D}\subseteq \prr(N_D)}$, where the 
powers of $\ad(M)$, ${M=L_p, N, N_i}$, are their powers as submodules of 
the associative algebra $\Ad(M)$. Thus, unlike associative algebras, the 
semiprimeness of Lie algebras is not inherited by their ultraproducts.  

If ${\{z_k E_{12}\}_{k\geq 0}\subset L_p}$, ${z_0=\overline{x_1}}$, 
${z_k=z_{k-1}^2 \overline{x_{3 (k-1)+1}} \overline{x_{3 (k-1)+2}} 
\overline{x_{3 k}}}$,
\[
z_k E_{12}\ =\ 
[[[[z_{k-1} E_{12}, \overline{x_{3 (k-1)+1}} E_{23}], 
\overline{x_{3 (k-1)+2}} E_{21}], z_{k-1} E_{12}], 
\overline{x_{3 k}} E_{32}]\,,
\]
then ${z_{p-1}\ne z_p=0}$, 
${\{d_i\}_{i=0}^{\infty}\subset N\setminus U_D(N)}$ for 
${d_{i+1}=[[[[d_i, a_i], b_i], d_i], c_i]}$, ${d_i(j)=z_i E_{12}+R_j}$, 
${a_i(j)=\overline{x_{3 i+1}} E_{23}+R_j}$, 
${b_i(j)=\overline{x_{3 i+2}} E_{21}+R_j}$, 
${c_i=\overline{x_{3 (i+1)}} E_{32}+R_j\in N_j}$, 
${i\geq 0}$, ${j\geq 1}$. As a consequence, ${V'\ne \prr(V')}$.

The question of whether the prime radical is radical on the class of all Lie 
algebras over a field of characteristic zero remains open. Since the prime 
radical of algebras are equal to their prime radicals as rings, a negative 
answer to this question for one of such fields carries over to all fields of 
characteristic zero.

\bigskip

Let us move on to group analogues of the example of Lemma 3.12. Let 
${k\geq 2}$, ${Y=\{y_i\}_{i\in \mathbb{Z}}}$, 
${U_k=(y_i^k\mid i\in \mathbb{Z})_{F[Y]}}$, ${B_k=F[Y]/U_k}$, 
${\overline{f}=f+U_k}$, ${f\in F[Y]}$. As above, 
${\prr(B_k)=\overline{Y} B_k}$ is locally nilpotent, but not nilpotent, 
${\prr(M_3(B_k))=M_3(\prr(B_k))}$, 
\begin{multline*}
G\ =\ E+M_3(\prr(B_k))\ =\ \prr(G)\subset 
\\
\prr(GL_3(B_k))\ =\ 
\{(g_{ij})\in GL_3(B_k)\mid [(g_{ij}), GL_3(B_k)]\subseteq G\}\,,
\end{multline*}
\cite[Corollary~3.8]{Gol7}. To bijections ${\sigma\in S(\mathbb{Z})}$ 
there correspond automorphisms ${\phi_{\sigma}\in \Aut_F(F[Y])}$ and  
${\overline{\phi}_{\sigma}\in \Aut_F(M_3(B_k))}$, 
${y_i \phi_{\sigma}=y_{\sigma(i)}}$, ${i\in \mathbb{Z}}$, 
${(\overline{f_{st}})\overline{\phi}_{\sigma}=
(\overline{f_{st} \phi_{\sigma}})}$, 
${f_{st}\in F[Y]}$, ${1\leq s, t\leq 3}$ 
($\overline{\phi}_{\sigma}$ is the extension of  
${\overline{\phi}_{\sigma}\in \Aut_F(B_k)}$ induced by  
$\phi_{\sigma}$ (${U_k=(U_k)\phi_{\sigma}}$) onto $M_3(B_k)$, 
${\overline{f}\, \overline{\phi}_{\sigma}=\overline{f \phi_{\sigma}}}$, 
${f\in F[Y]}$). We set ${\phi_i=\phi_{\sigma_i}}$, 
${\overline{\phi}_i=\overline{\phi}_{\sigma_i}}$, 
${\sigma_i\in S(\mathbb{Z})}$, ${\sigma_i: j\longmapsto j+i}$, 
${i, j\in \mathbb{Z}}$. We identify the groups  
$G$ and ${\langle \overline{\phi}_1\rangle_{\infty}=
\{\overline{\phi}_i\}_{i\in \mathbb{Z}}}$ with the subgroups of their 
external semidirect product 
${Z=G\rtimes \langle \overline{\phi}_1\rangle_{\infty}}$ 
(${\prr(M_3(B_k))\lhd_{\Aut} M_3(B_k)}$), ${G\lhd Z}$, 
${\langle \overline{\phi}_1\rangle_{\infty}\cong Z/G}$, 
${(\overline{g_{st}})^{\overline{\phi}_i}=
(\overline{g_{st} \phi_i})}$, ${i\in \mathbb{Z}}$, 
${(\overline{g_{st}})\in G}$, where ${x^y=y x y^{-1}}$, 
${\overline{\phi}_i=\overline{\phi}_1^i}$. 

\begin{lemma}
The group $Z$ is weakly solvable, but not locally solvable, 
\[
G\ =\ \prr(G)\ =\ LN(Z)\not\subseteq \prr(Z)\,, 
\]
where $LN(Z)$ is the largest locally nilpotent normal subgroup 
of $Z$ (the Plotkin --- Hirsh radical of $Z$).
\end{lemma}

\begin{proof}
First, let us note that 
${\{t_{12}(g_q)\}_{q\geq 0}\subseteq G\setminus \prr(Z)}$ for 
${g_q=\prod\limits_{i=1}^{s_q} \overline{y_i}}$, ${s_q=2^{q+2}-3}$, 
since ${s_{q+1}=2 s_q+3}$,
\begin{multline*}
t_{12}(g_{q+1})\ =\ t_{12}((g_q \overline{\phi}_{s_q+3}) g_q 
\overline{y_{s_q+1}} \overline{y_{s_q+2}} \overline{y_{s_q+3}})\ =
\\
[[[[t_{12}(g_q)^{\overline{\phi}_{s_q+3}}, 
t_{23}(\overline{y_{s_q+1}})], t_{21}(\overline{y_{s_q+2}})], 
t_{12}(g_q)], t_{32}(\overline{y_{s_q+3}})]\quad (q\geq 0)\,.
\end{multline*}
To each ${z\in B_k}$ there corresponds a subgroup 
${E(z)=\langle t_{ij}(z)\mid 1\leq i\ne j\leq 3\rangle\subset GL_3(B_k)}$, 
\[
E(z+z')\subseteq \langle E(z),\, E(z')\rangle\,,\quad 
E(z z')\subseteq [E(z), E(z')]\quad (z, z'\in B_k)\,. 
\]
If $B$ is a subring of $B_k$, ${E(B)=\langle E(b)\mid b\in B\rangle}$, then 
${E(B^{2l})\subseteq E(B)^{(l)}}$ for any ${l\geq 0}$ and, as a consequence, 
the groups $E(\prr(B_k))$, ${E(\overline{Y \mathbb{Z}[Y]})=
\langle E(\overline{y_i})=E(\overline{y_1})^{\overline{\phi}_1^{i-1}}
\mid i\in \mathbb{Z}\rangle}$, 
${\langle E(\overline{y_1}), \overline{\phi}_1\rangle}$ are not solvable 
(${\prr(B_k)=F \overline{Y \mathbb{Z}[Y]}}$), the group $Z$ is not locally 
solvable, but weakly solvable by construction. In view of 
${G=LN(G)\subseteq LN(Z)}$ and 
\[
[\ldots [t_{12}(\overline{y_1}), 
\underbrace{\overline{\phi}_i]\ldots \overline{\phi}_i]}_n\ =\ 
t_{12}\biggl(\sum_{j=0}^n (-1)^j \binom{n}j \overline{y}_{1+j i}\biggr)\ne 
E\quad (0\ne i\in \mathbb{Z},\ n\geq 1)\,,
\]
${G=LN(Z)}$. 
\end{proof} 

It follows from the proof of Lemma 3.14 that 
${\tilde{G}=\prr(\tilde{G})=LN(\tilde{Z})\not\subseteq \prr(\tilde{Z})}$ for  
\begin{gather*}
\tilde{G}\ =\ \langle t_{12}(\overline{y_i}),\ t_{21}(\overline{y_i}),\ 
t_{23}(\overline{y_i}),\ t_{32}(\overline{y_i})\mid i\in \mathbb{Z}\rangle\,,
\\ 
\tilde{Z}\ =\ \tilde{G}\rtimes \langle \overline{\phi}_1\rangle_{\infty}\ =\ 
\langle t_{12}(\overline{y_1}),\ t_{23}(\overline{y_1}),\ 
t_{21}(\overline{y_1}),\ t_{32}(\overline{y_1}),\ \overline{\phi}_1\rangle\,.
\end{gather*}

Another similar example is the construction of A.~Yu.~Ol'shanskii, close to 
the example of Chan Van Hao, \cite{Chan}. Let $A$ be a semiprime associative 
ring with 1, $T$ be a group of infinite upper unitriangular matrices 
${(a_{ij})_{i, j\in \mathbb{Z}}}$, ${a_{ij}\in A}$, with a finite number of 
non-zero overdiagonal elements, ${a_{ii}=1}$, 
${a_{ij}=0}$ for ${i> j}$, ${|\{a_{ij}\ne 0\mid i< j\}|< \infty}$, $T_k$ be 
an Abelian subgroup of ${(a_{ij})\in T}$ with 
${\{(i, j)\mid i< j,\ a_{ij}\ne 0\}\subseteq \{(i, j)\mid i< k< j\}}$, 
${k\in \mathbb{Z}}$, ${T_k\lhd T}$, 
\[
T\ =\ \langle T_k\mid k\in \mathbb{Z}\rangle\ =\ 
\langle t_{st}(a)=E+a E_{st}\mid a\in A,\ s< t,\ s, t\in \mathbb{Z}\rangle
\ =\ \prr(T)\ =\ LN(T)
\]
for ${E=(\delta_{ij})_{i, j\in \mathbb{Z}}}$, 
${E_{st}=(\delta_{is} \delta_{jt})_{i, j\in \mathbb{Z}}}$, 
and ${\psi\in \Aut(T)}$, ${\psi((a_{ij}))=(b_{ij})}$,  
${b_{ij}=a_{i-1 j-1}}$, ${i, j\in \mathbb{Z}}$, ${(a_{ij})\in T}$, 
${\psi(T_k)=T_{k+1}}$, ${k\in \mathbb{Z}}$. We identify the groups $T$ and 
$\langle \psi\rangle_{\infty}$ with the subgroups of their external 
semidirect product ${S=T\rtimes \langle \psi\rangle_{\infty}}$, ${T\lhd S}$, 
${\langle \psi\rangle_{\infty}\cong S/T}$, 
${(a_{ij})^{\psi^k}=\psi^k((a_{ij}))}$, ${(a_{ij})\in T}$, ${k\in \mathbb{Z}}$. 
Let us recall one of the consequences of the presence of non-unitary elements 
of the center of a nilpotent group in its non-unitary normal subgroups:

\begin{rem}
If $B$ is an associative ring with 1, ${n> 1}$, $UT_n(B)$ is the group 
of upper unitriangular ${n\times n}$ matrices with coefficients from 
$B$,${\{E\}\ne P\lhd UT_n(B)}$, then there exists ${E\ne t_{1n}(b)\in P}$.
\end{rem}

\begin{proof} 
In this case ${C(UT_n(B))=E+B E_{1n}=\{t_{1n}(b)\mid b\in B\}}$ is the 
center of the nilpotent group $UT_n(B)$, ${E, E_{ij}\in M_n(B)}$.
\end{proof}

\begin{lemma}
The group $S$ is semiprime, weakly solvable, but not locally solvable, and  
${T=\prr(T)=LN(S)}$.
\end{lemma}

\begin{proof}
If ${E \psi^0\ne (a_{ij}) \psi^n\in S}$, ${(a_{ij})\in T}$, 
${n\in \mathbb{Z}}$, where ${E \psi^{0}=E=\psi^0}$ is the unity of $S$, 
$T$, ${\langle \psi\rangle_{\infty}\subseteq S}$, then for 
${(a_{ij})\ne E}$ there exists ${k\geq 0}$, 
${(a_{ij}) \psi^k((a_{ij})^{-1})=[(a_{ij}) \psi^n, \psi^k]\ne E}$, 
for ${(a_{ij})=E}$ and ${n\ne 0}$ one has 
${[t_{ij}(a), E \psi^n]=E+a (E_{ij}-E_{i+n j+n})\ne E}$, ${i, j\in \mathbb{Z}}$, 
${0\ne a\in A}$. As a consequence, any non-unitary normal subgroup 
${Q\lhd S}$ contains ${E\ne (b_{ij})\in T\cap Q}$. Since 
${\{(i, j)\mid i< j,\ b_{ij}\ne 0\}\subset \{(i, j)\}_{i, j=s+1}^{s+m}}$ for 
some ${s\in \mathbb{Z}}$, ${m\geq 1}$, $(b_{ij})$ is the image of 
${(c_{pq}=b_{s+p s+q})_{p, q=1}^m\in UT_m(A)}$ under the action of the 
embedding ${\tau_s: UT_m(A)\longrightarrow T}$, 
\[
\tau_s((d_{pq}))\ =\ E+\sum_{1\leq p< q\leq m} d_{pq} E_{s+p\, s+q}\quad 
((d_{pq})\in UT_m(R))\,.
\]
By Remark 3.15, there is ${E\ne E+z E_{1m}\in ((c_{pq}))_{UT_m(A)}}$, 
\[
\tau_s(t_{1m}(z))\ =\ t_{s+1\, s+m}(z)\in \tau_s(((c_{pq}))_{UT_m(A)})
\subseteq ((b_{ij}))_T\subseteq Q\,.
\] 
The semiprimeness of the ring $A$ allows one to choose ${a\in A}$, 
${z a z\ne 0}$, 
\begin{multline*}
E\ne [[t_{s+1\, s+m}(z), t_{s+m\, s+2 m}(a)], t_{s+1\, s+m}(z)^{\psi^{2 m-1}}]
\ =
\\ 
[t_{s+1\, s+2 m}(z a), t_{s+2 m\, s+3 m-1}(z)]\ =\ 
t_{s+1\, s+3 m-1}(z a z)\in [Q, Q]\,,
\end{multline*}
Therefore the group $S$ is semiprime. By construction it is weakly solvable, but 
not locally solvable, due to  
\begin{multline*} 
g^{gr}_k(t_{12}(1), t_{12}(1)^{\psi}, \ldots, t_{12}(1)^{\psi^{2^k-1}})\ =
\\ 
g^{gr}_k(t_{12}(1), t_{23}(1), \ldots, t_{2^k\, 2^k+1}(1))\ =\ 
t_{1\, 2^k+1}(1)\in \langle t_{12}(1), \psi\rangle^{(k)}\quad (k\geq 0)\,.
\end{multline*}
It remains to be noted that for any ${i\ne \pm 1}$
\[
[\ldots[t_{12}(1), \underbrace{\psi^i]\ldots \psi^i]}_n\ =\ 
E+\sum_{j=0}^n (-1)^j \binom{n}j E_{1+ij\, 2+ij}
\]
(${1+i j\ne 2+i j'}$, ${j, j'\in \mathbb{Z}}$), ${\psi^i\notin LN(S)}$ for 
${|i|> 1}$, ${\psi\notin LN(S)}$, ${T=LN(S)}$.
\end{proof} 

From Lemmas 3.14, 3.16 it immediately follows

\begin{teor}
The prime radical $\prr$ is not a radical in the sense of Kurosh --- Amitsur 
on the class of all groups.
\end{teor}

This is true for any class of groups that is closed under taking normal 
subgroups and homomorphic images and contains a group of the form $Z$ 
($\tilde{Z}$) or (and) $S$ from Lemmas 3.14, 3.16. The existence of groups 
that are equal to their prime radicals and therefore weakly solvable but 
not locally solvable was established in \cite{KN}.


\begin{thebibliography}{99} 

\bibitem{AR} V.~A.~Andrunakievich, V.~M.~Ryabukhin Radicals of algebras and 
structure theory. --- Moscow: Nauka, 1979.

\bibitem{ACZ} J.~A.~Anquela, T.~Cort\'es, E.~Zelmanov Local Nilpotency of the 
McCrimmon radical of a Jordan system// Trudy Mat. Inst. im. V.~A.~Steklova. --- 
2016. --- Vol. 292. --- P. 7--15.

\bibitem{ArmS} E.~P.~Armendariz, S.~Steinberg Regular self-injective rings with 
polynomial identities// Trans. Amer. Math. Soc. --- 1974. --- Vol. 190. --- 
P. 417--425.

\bibitem{Baer} R.~Baer. Nilgruppen // Math. Z. --- 1955. --- Bd. 62, Hft. 4. 
--- S. 402--437. 

\bibitem{B3} K.~I.~Beidar Rings with generalized identities III// Vestnik MSU, 
Mat., Mech. --- 1978. --- no. 4. --- P. 66--73.

\bibitem{B4} K.~I.~Beidar Semiprime rings with generalized identity// 
Uspekhi Mat. Nauk. --- 1977. --- Vol. 32, no. 4. --- P. 249--250. 

\bibitem{B5} K.~I.~Beidar Classical quotient rings of $PI$--algebras// Uspekhi 
Mat. Nauk. --- 1978. --- Vol. 33, no. 6. --- P. 197--198.

\bibitem{B6} K.~I.~Beidar Rings of quotients of semiprime rings// Vestnik MSU, 
Mat., Mech. --- 1978. --- no. 5. --- P. 36--43. 

\bibitem{B7} K.~I.~Beidar Radicals, orthogonal completeness and ring struction. 
Diss. d. ph.-m. n. LSU. --- Moscow, 1989.

\bibitem{BGr} K.~I.~Beidar, P.~Grzeszczuk Actions of Lie algebras on rings 
without nilpotent elements// Algebra Colloq. --- 1995. --- Vol. 2, no. 2. --- 
P. 105--116.

\bibitem{BMar} K.~I.~Beidar, V.~T.~Markov A semiprime $PI$--ring having an 
exact module with a Krull dimension is a Goldie ring// Uspekhi Mat. Nauk. 
--- 1993. --- Vol. 48, no. 6. --- P. 141--142.
                                      
\bibitem{BMM} K.~I.~Beidar, W.~S.~Martindale III, A.~V.~Mikhalev Rings with 
generalized identities. --- New York: Marcel Dekker, 1996.

\bibitem{BMich} K.~I.~Beidar, A.~V.~Mikhalev Orthogonal completeness and 
algebraic systems// Uspekhi Mat. Nauk. --- 1985. --- Vol. 40, no. 6. --- 
P. 79--115.

\bibitem{BP} K.~I.~Beidar, S.~A.~Pikhtil'kov On the prime radical of special 
Lie algebras// Uspekhi Mat. Nauk. --- 1994. --- Vol. 49, no. 1. --- P. 233.

\bibitem{ADN} J.~Brox, E.~Garc{\'\i}a, M.~G{\'o}mez~Lozano, 
R.~Mu{\~n}oz~Alc{\'a}zar, G.~Vera~De~Salas A description of Ad-nilpotent 
elements in semiprime rings with involution// Bull. Malays. Math. Sci. Soc. ---
2021. --- Vol. 44, no. 4. --- P. 2577--2602. ---
DOI 10.1007/s40840-020-01064-w.

\bibitem{Cab3} J.~C.~Cabello, M.~Cabrera, R.~Roura $\pi$--complementation 
in the unitisation and multiplication algebras of a semiprime algebra// Comm. 
Algebra. --- 2012. --- Vol. 40, no. 9. --- P. 3507--3531.

\bibitem{Cab4} M.~Cabrera, A.~Fernandez~Lopez, A.~Yu.~Golubkov, A.~Moreno 
Algebras whose multiplication algebra is $PI$ or $GPI$// J. Algebra. --- 
2016. --- no. 459. --- P. 213--237.

\bibitem{Chan} Chan Van Hao On the minimal radical class over a class of Abelian 
groups// Dokl. Akad. Nauk SSSR. --- 1963. --- Vol. 149, no. 6. --- P. 1270--1273.

\bibitem{Chua} C.-L.~Chuang GPIs having coefficients in Utumi quotients rings// 
Proc. Amer. Math. Soc. --- 1988. --- Vol. 103, no. 3. --- P. 723--728. 

\bibitem{ChuaL} C.-L.~Chuang, T.-K.~Lee Nilpotent derivations// J. Algebra. --- 
2005. --- Vol. 287. --- P. 381--401.

\bibitem{Chung} L.~O.~Chung, J.~Luh Nilpotency of derivatives on an ideal// 
Proc. Amer. Math. Soc. --- 1984. --- Vol. 90, no 2. --- P. 211--214.

\bibitem{Fe} C.~Faith Algebra: Rings, Modules and Categories I. --- 
Berlin; Heidelberg; New York: Springer --- Verlag, 1973; 
Algebra II: Ring Theory. --- Berlin; Heidelberg; New York: 
Springer --- Verlag, 1976.
                                              
\bibitem{Fish} J.~W.~Fisher Structure of semiprime $PI$--rings// Proc. Amer. 
Math. Soc. --- 1973. --- Vol. 39, no. 3. --- P. 465--467.

\bibitem{Gol1} A.~Yu.~Golubkov Constructions of special radicals of algebras// 
J. Math. Sci. --- 2017. --- Vol. 223, no. 5. --- P. 530--580. 

\bibitem{Gol6} A.~Yu.~Golubkov Quasi-regular radicals of nonassociative 
algebras// J. Math. Sci. --- 2024. --- Vol. 284, no. 4. --- P. 460--497.

\bibitem{Gol7} A.~Yu.~Golubkov Radical $RN$ and weakly solvable radical of 
linear groups over associative rings// J. Math. Sci. --- 2008. --- 
Vol. 154, no. 2. --- P. 143--203.

\bibitem{Gol8} A.~Yu.~Golubkov Heredity of radicals and ideals of algebras 
generated by subideals and subinvariant subalgebras// J. Math. Sci. --- 
2023. --- Vol. 269, no. 5. --- P. 708--724.

\bibitem{Gol9} A.~Yu.~Golubkov The weakly solvable radical and locally strongly 
algebraic derivations of locally generalized special Lie algebras// J. Math. 
Sci. --- 2022. --- Vol. 262, no. 5. --- P. 652--659.

\bibitem{Gol10} A.~Yu.~Golubkov Orthogonal completion of algebraic systems// arXiv:2609.01854v2 [math.RA].

\bibitem{Grz} P.~Grzeszczuk On nilpotent derivations of semiprime rings// 
J. Algebra. --- 1992. --- Vol. 149. --- P. 313--321.

\bibitem{Her} I.~N.~Herstein Non-commutative rings. --- Math. Ass. Amer., 1968.

\bibitem{JacSR} N.~Jacobson Structure of rings. --- Amer. Math. Soc. Coll. Pub., 
Vol. 37, 1956.

\bibitem{Jain} S.~K.~Jain Prime rings having one-sided ideal with polynomial 
identity coincide with special Johnson rings// J. Algebra. --- 1971. --- 
Vol. 19. --- P. 125--130.

\bibitem{Har1} V.~K.~Kharchenko Differential identities of prime rings// 
Algebra i Logika. --- 1978. --- Vol. 17, no. 2. --- P. 220--238. 

\bibitem{Har2} V.~K.~Kharchenko Differential identities of semiprime rings// 
Algebra i Logika. --- 1979. --- Vol. 18, no. 1. --- P. 86--119.

\bibitem{Har3} V.~K.~Kharchenko The actions of groups and Lie algebras on 
non-commutative rings// Uspekhi Mat. Nauk. --- 1980. --- Vol. 35, no. 2. --- P. 67--90.

\bibitem{Har4} V.~K.~Kharchenko Generalized identities with automorphism// 
Algebra i Logika. --- 1975. --- Vol. 14, no. 2. --- P. 215--237.

\bibitem{KN} L.~G.~Kov{\'a}cs, B.~H.~Neumann An embedding theorem for some 
countable groups// Acta scient. math. --- 1965. --- Vol. 26, no. 1. --- 
P. 139--142.

\bibitem{Lamb} J.~Lambek Lectures on rings and modules. --- Blaisdell, 
Waltham, Mass., 1966. 

\bibitem{LeeL} P.~H.~Lee, T.-K.~Lee Note on nilpotent derivations// Proc. 
Amer. Math. Soc. --- 1986. -- Vol. 98, no. 1. --- P. 31--32.

\bibitem{Lee} T.-K.~Lee Semiprime rings with differential identities// Bull. 
Inst. Math. Acad. Sinica. --- 1992. --- Vol. 20, no. 1. --- P. 27--38.

\bibitem{LeeA} T.-K.~Lee Ad-nilpotent elements of semiprime rings with 
involution// Can. Math. Bull. --- 2018. --- Vol. 61, no. 2. --- P. 318--327.

\bibitem{Mar} V.~T.~Markov On the dimension of non-commutative affine algebras// 
Izv. Ak. Nauk SSSR. Ser. Mat. --- 1973. --- Vol. 37. --- P. 284--288.
        
\bibitem{Mar1} V.~T.~Markov On $PI$--rings having a faithful module with 
Krull dimension// Fundament. i prikl. Mat. --- 1995. --- Vol. 1, no. 2. --- P. 557--559.

\bibitem{Mar2} V.~T.~Markov $PI$--rings and semiprime rings having faithful 
modules with Krull dimension// Algebra i Logika. --- 1997. --- Vol. 36, 
no. 5. --- P. 562--572. 

\bibitem{Mar3} V.~T.~Markov Rings of quotients of semiprime $PI$--rings, and 
irreducible subdirect products// Uspekhi Mat. Nauk. --- 1975. --- Vol. 30, 
no. 4. --- P. 253--254.
 
\bibitem{Mart} W.~S.~Martindale III Prime rings satisfying a generalized 
polynomial identity// J. Algebra. --- 1969. --- Vol. 12. --- P. 576--584.

\bibitem{Mart1} W.~S.~Martindale III On semiprime P.I. rings// Proc. Amer. 
Math. Soc. --- 1973. --- Vol. 40, no. 2. --- P. 365--369.

\bibitem{Parf} V.~A.~Parfenov On a weakly solvable radical of Lie algebras// 
Sib. Mat. Zh. --- 1971. --- Vol. 12, no. 1. --- P. 171--176.

\bibitem{Pch} S.~V.~Pchelintsev Exceptional prime alternative algebras// 
Sib. Mat. Zh. --- 2007. --- Vol. 48, no. 6. --- P. 1322--1337.

\bibitem{Pch1} S.~V.~Pchelintsev Degenerate alternative algebras// Sib. Mat. 
Zh. --- 2014. --- Vol. 55, no. 2. --- P. 396--411.

\bibitem{Pos} E.~C.~Posner Prime rings satisfying a polynomial identity// 
Proc. Amer. Math. Soc. --- 1960. --- Vol. 11, no 2. --- P. 180--183. 

\bibitem{Raz} Yu.~P.~Razmyslov Identities of algebras and their 
representations. --- Moscow: Nauka, 1989.

\bibitem{Row} L.~H.~Rowen Polynomial identities in ring theory. --- London: 
Academic Press, 1980. --- (Pure Appl. Math.; Vol. 84). 

\bibitem{Row1} L.~H.~Rowen Some results on the center of ring with polynomial 
identity // Bull. Amer. Math. Soc. --- 1973. --- no. 1. --- P. 219--223.

\bibitem{Sch} K.~K.~Schukin The $RI^*$--solvable radical of groups// Mat. Sb. --- 
1960. --- Vol. 52, no. 4. --- P. 1024--1031.

\bibitem{Sh} A.~I.~Shirshov On some non-associative null-rings and algebraic 
algebras// Mat. Sb. --- 1957. --- Vol. 83, no. 3. --- P. 381--394.

\bibitem{Skos} V.~G.~Skosyrskii Right-alternative algebras// Algebra i Logika. 
--- 1984. --- Vol. 23, no. 2. --- P. 185--192.

\bibitem{Skos1} V.~G.~Skosyrskii Nilpotency in Jordan and right-alternative 
algebras// Algebra i Logika. --- 1979. --- Vol. 18, no. 1. --- P. 73--85.
 
\bibitem{AT2} A.~Thedy Radicals of right-alternative and Jordan rings// 
Comm. Algebra. --- 1984. --- Vol. 12, no. 7. --- P. 857--887. 

\bibitem{Ut} Y.~Utumi On quotient rings// Osaka Math. J. --- 1956. --- 
Vol. 8, no. 1. --- P. 1--18.

\bibitem{Zel1} E.~I.~Zel'manov Absolute zero divisors and algebraic Jordan 
algebras// Sib. Mat. Zh. --- 1980. --- Vol. 23, no. 6. --- P. 100--116. 
                                                                                                                               
\bibitem{Zel} E.~I.~Zel'manov Lie algebras with an algebraic adjoint 
representation// Mat. Sb. --- 1983. --- Vol. 121 (163), no. 4(8). --- P. 545--561.

\bibitem{ZelM} E.~I.~Zel'manov Characterization of the McCrimmon radical// 
Sib. Mat. Zh. --- 1984. --- Vol. 25, no. 5. --- P. 190--192.

\bibitem{ZShS} K.~A.~Zhevlakov, A.~M.~Slin'ko, I.~P.~Shestakov, A.~I.~Shirshov 
Rings that are nearly associative. --- New York: Academic Press, 1982. 

\end{thebibliography}
\end{document}